\documentclass{amsart}

\title[GK-dimension and the $p$-adic JL-correspondence]{Gelfand--Kirillov dimension and the $p$-adic Jacquet--Langlands correspondence}

\author{Gabriel Dospinescu, Vytautas Pa\v{s}k\={u}nas and Benjamin Schraen}

\date{\today.}

\theoremstyle{plain}
\newtheorem{prop}{Proposition}[section]
\newtheorem{defi}[prop]{Definition}
\newtheorem{conj}[prop]{Conjecture}
\newtheorem{lem}[prop]{Lemma}
\newtheorem{thm}[prop]{Theorem}
\newtheorem{cor}[prop]{Corollary}

\theoremstyle{definition} 
\newtheorem{examp}[prop]{Example}
\newtheorem{remar}[prop]{Remark}

\usepackage{hyperref}
\usepackage{stmaryrd}
\usepackage[all]{xy}
\usepackage{leftidx}
\RequirePackage{amsmath}
\RequirePackage{amssymb}
\RequirePackage{amsxtra}
\RequirePackage{amsfonts}
\RequirePackage{latexsym}
\RequirePackage{euscript}
\RequirePackage{amscd}
\RequirePackage{amsthm}

\DeclareSymbolFont{largesymbols}{OMX}{yhex}{m}{n}
\DeclareMathAccent{\wideparen}{\mathord}{largesymbols}{"F3}

\RequirePackage{mathtools}

\DeclareMathOperator{\Spf}{Spf}

\DeclareMathAlphabet{\mathpzc}{OT1}{pzc}{m}{it}

\DeclareMathOperator{\End}{End}
\DeclareMathOperator{\Hom}{Hom}

\DeclareMathOperator{\Ind}{Ind}

\DeclareMathOperator{\Res}{Res}
\DeclareMathOperator{\Sym}{Sym}

\DeclareMathOperator{\GL}{GL}
\DeclareMathOperator{\SL}{SL}
\DeclareMathOperator{\Ker}{Ker}

\DeclareMathOperator{\Gal}{Gal}

\DeclareMathOperator{\tr}{tr}

\newcommand{\gr}{\mathrm{gr}}
\DeclareMathOperator{\Spec}{Spec}

\DeclareMathOperator{\Mod}{Mod}

\DeclareMathOperator{\Sp}{Sp}

\DeclareMathOperator{\Tor}{Tor}

\DeclareMathOperator{\Ban}{Ban}

\DeclareMathOperator{\dualcat}{\mathfrak C}

\DeclareMathOperator{\Art}{Art}

\newcommand{\Q}{\mathbb{Q}}
\newcommand{\Qp}{\mathbb {Q}_p}
\newcommand{\Zp}{\mathbb{Z}_p}
\newcommand{\Qpbar}{\overline{\mathbb{Q}}_p}

\newcommand{\PP}{\mathbb P}

\newcommand{\FF}{\mathcal F}

\newcommand{\ZZ}{\mathbb Z}

\newcommand{\QQ}{\mathbb Q}

\newcommand{\ev}{\mathrm{ev}}

\newcommand{\Fp}{\mathbb F_p}

\newcommand{\Fbar}{\overline{F}}

\newcommand{\mm}{\mathfrak m}

\newcommand{\Fpbar}{\overline{\mathbb{F}}_p}

\newcommand{\OO}{\mathcal O}

\newcommand{\TT}{\mathbb T}

\newcommand{\gal}{\Gal_{\Qp}}

\DeclareMathOperator{\wtimes}{\widehat{\otimes}}

\newcommand{\md}{\mathrm m}

\newcommand{\pp}{\mathfrak p}

\newcommand{\rig}{\mathrm{rig}}

\newcommand{\br}[1]{\llbracket #1\rrbracket}

\newcommand{\sm}{\mathrm{sm}}

\newcommand{\pro}{\mathrm{pro}}

\newcommand{\adm}{\mathrm{adm}}

\newcommand{\cont}{\mathrm{cont}}

\newcommand{\ladm}{\mathrm{l.adm}}

\newcommand{\Z}{\mathbb{Z}}

\def\et{\mathrm{\acute{e}t}}

\DeclareMathOperator{\Frob}{Frob}

\newcommand{\Fbreve}{\breve{F}}
\newcommand{\Cp}{\mathbb C_p}
\newcommand{\rbar}{\bar{r}}

\newcommand{\rhobar}{\bar{\rho}}
\newcommand{\ps}{\mathrm{ps}}
\newcommand{\univ}{\mathrm{univ}}

\newcommand{\Shat}{\check{\mathcal S}}
\newcommand{\AfF}{\mathbb{A}_F^{\infty}}
\newcommand{\cyc}{\mathrm{cyc}}
\newcommand{\Nrd}{\mathrm{Nrd}}

\newcommand{\Ghat}{\widehat{G}}

\newcommand{\That}{\widehat{T}}

\newcommand{\ghat}{\widehat{\mathfrak g}}

\DeclareMathOperator{\Lie}{Lie}
\newcommand{\fgaug}{\mathrm{fg.aug}}
\newcommand{\Sch}{\mathcal{\check{S}}}
\newcommand{\gl}{\mathfrak{gl}}
\newcommand{\la}{\mathrm{la}}
\DeclareMathOperator{\JL}{JL}
\newcommand{\Kla}{K\text{-}\la}
\newcommand{\Lla}{L\text{-}\la}
\newcommand{\ZgL}{Z(\mathfrak g_L)}
 
\begin{document}
\maketitle

\begin{abstract} We bound the Gelfand--Kirillov dimension of unitary Banach space
representations of $p$-adic reductive groups, whose locally analytic vectors afford 
an infinitesimal character. We use the bound to study Hecke eigenspaces in completed 
cohomology of Shimura curves and $p$-adic Banach space representations of the group of units of a quaternion algebra over $\Qp$ appearing in the 
$p$-adic Jacquet--Langlands correspondence, deducing finiteness results in favourable cases. 

% Insert your abstract here. Remember that online systems rely heavily on the content
% of titles and abstracts to identify articles in electronic bibliographic databases 
% and search engines. We ask you to take great care in preparing the abstract
% and to not use references to the bibliography.
\end{abstract}

%% \tableofcontents %% Just for papers exceeding 50 pages.

\section{Introduction}\label{sec:intro}

        This paper is a continuation of \cite{DPS}, which proves that the locally analytic vectors of many $p$-adic Banach representations of global origin have an infinitesimal character, which can be explicitly computed from 
        $p$-adic Hodge-theoretic data. Here we are concerned with applications of those ideas to a very specific problem: the study of the $p$-adic Jacquet-Langlands correspondence between $p$-adic Banach representations of 
        $G:=\GL_2(\Qp)$ and of $D^{\times}$, where $D$ is a quaternion division algebra over $\Qp$. Along the way, we prove a new upper bound for the Gelfand--Kirillov dimension of admissible Banach representations whose locally analytic vectors have an infinitesimal character, which plays a key role in our main result concerning the $p$-adic Jacquet--Langlands correspondence. 
This bound also gives a proof of the fact that Hecke eigenspaces in the (rational) completed cohomology of the tower of modular curves have finite length, independent of the $p$-adic local Langlands correspondence and of the classification of irreducible mod 
$p$ smooth representations of $\GL_2(\Qp)$.
        
         Let us fix a finite extension $L$ of $\Q_p$, with ring of integers $\mathcal{O}$ and residue field $k$, which will serve as 
        coefficient field for all representations that we deal with. Let $C$ be the completion of a fixed algebraic closure $\Qpbar$ of $\Qp$. 
        A candidate for the $p$-adic Jacquet--Langlands correspondence was constructed\footnote{We stick to the $\GL_2(\Qp)$ case here: Scholze's constructions are more general, but very little is known beyond the case of $\GL_2(\Qp)$, which is already far from being completely understood.} by Scholze in a purely geometric way in \cite{scholze}, using the cohomology of the infinite-level Lubin--Tate space $\mathcal{LT}_{\infty}$, which is a pro-\'etale $G$-torsor over the analytic projective line  $\mathbb{P}^1_C$. 
           More precisely, he constructed functors 
      $\mathcal S^i$ (for $0\leq i\leq 2$) from smooth $\mathcal{O}$-torsion $\mathcal{O}[G]$-modules to smooth $\mathcal{O}$-torsion 
      $\mathcal{O}[D^{\times}]$-modules by setting
      \[\mathcal S^i(\pi)=H^i_{\et}(\mathbb{P}^1_C, \mathcal{F}_{\pi}),\]
      where $\mathcal{F}_{\pi}$ is the descent of the 
      constant sheaf $\underline{\pi}$ along the 
      pro-\'etale $G$-torsor $\mathcal{LT}_{\infty}\to \mathbb{P}^1_C$. The most interesting case is $i=1$, and we let 
      $\JL(\pi)=\mathcal S^1(\pi)$. 
         Despite the simple-looking definition, 
      it seems very hard to compute $\JL(\pi)$ or even to establish some of its basic properties (for instance whether it is nonzero...). Here we depart from the torsion situation and study what happens in characteristic $0$, exploiting the action of the Lie algebra of $D^{\times}$, more precisely of its centre. 
      
       Let $\Pi$ be an admissible unitary $L$-Banach space representation of $G$, and let $\Theta\subset \Pi$ be an open and bounded $G$-invariant lattice in $\Pi$. 
       Define\footnote{In the main text $\JL$ will be denoted by $\Sch^1$.} 
       \[\JL(\Pi)=(\varprojlim_{n} \JL(\Theta/p^n\Theta))_{\rm tf}\otimes_{\mathcal{O}} L,\]
       where the subscript tf refers to the maximal Hausdorff torsion-free quotient.
       By one of the main results of Scholze's paper \cite{scholze}, $\JL(\Pi)$ is an admissible unitary $L$-Banach space representation of $D^{\times}$
       (there is also an action of the absolute Galois group of $\Qp$ on $\JL(\Pi)$, commuting with $D^{\times}$, but we will not make use of it in this paper).
              
        Let $\widehat{G}_{\rm ss}$ be the set of (isomorphism classes of) absolutely irreducible unitary, admissible $L$-Banach space representations of $G$, which are \textit{non-ordinary} (also known as
        \textit{supersingular}), i.e. which are not subquotients of the parabolic induction of a unitary character. The $p$-adic local Langlands correspondence \cite{CDP} gives a bijection 
        \[  V: \widehat{G}_{\rm ss}  \to {\rm Irr}_2(\gal)\]
        between $\widehat{G}_{\rm ss}$ and the set of (isomorphism classes of) $2$-dimensional absolutely irreducible representations 
        of $\gal:={\rm Gal}(\Qpbar/\Qp)$ over $L$. We assume that $L$ is a splitting field of $D$ and we identify 
        ${\rm Lie}(D^{\times})\otimes_{\Qp} L$ and ${\rm Lie}(G)\otimes_{\Qp} L$, thus also their centers. 
                
    \begin{thm}\label{one_one} If $p>2$ then 
        for all $\Pi\in \widehat{G}_{\rm ss}$ the $D^{\times}$-representation 
        $\JL(\Pi)^{\la}$ has the same infinitesimal character as $\Pi^{\la}$.
       \end{thm}
            
        Even though all objects involved in the statement of the above theorem are of local nature, its proof uses heavily global methods.  We use the results of \cite{forum}, which in turn make use 
        of the $p$-adic Langlands correspondence for $\GL_2(\Qp)$, to show that every non-ordinary 
        $\Pi$ occurs as a subquotient in a Hecke eigenspace of the patched $0$-th  cohomology of 
        certain quaternionic Shimura sets. We then show, following Scholze \cite{scholze}, that 
        if we apply his functor to this Hecke eigenspace we obtain a Hecke eigenspace in the 
        patched $1$-st cohomology of certain Shimura curves. The results of 
        our previous paper \cite{DPS} allow us to conclude that the infinitesimal characters obtained 
        from the action of $\GL_2(\Qp)$ and the action of $D^{\times}$ on the locally analytic vectors 
        in the  respective eigenspaces are the same. We refer to Corollary \ref{the_end_is_nigh} for a more precise statement; in Remark \ref{unitary} we explain how some of our arguments can be extended to ${\rm GL}_n(F)$ 
       using the results of Caraiani and Scholze \cite{CS} and Kegang Liu \cite{kegang_liu}.  The 
       assumption $p>2$ enters when we globalize the local Galois representation and carry out the patching argument.
        
             The previous theorem is one of the key inputs in the proof of our main result (see Theorem \ref{the_same_inf}): 
        \begin{thm}\label{FL} Assume that $p>2$ and 
         let $\Pi\in \widehat{G}_{\rm ss}$ be a representation such that 
         the difference of the Hodge--Tate--Sen weights of $V(\Pi)$ is not a nonzero integer. Then 
         $\JL(\Pi)$ is a Banach representation of finite length, more precisely its restriction to any compact open subgroup of $D^{\times}$ has finite length.
        \end{thm}
        
        \begin{remar}
              We expect that the result still holds when the difference of the Hodge--Tate--Sen weights of $V(\Pi)$ is a nonzero integer. We refer to theorem $1.1$ in the recent work \cite{HW} of Hu and Wang for evidence supporting this expectation.
           The following basic question should shed light on what is happening in the complicated case excluded by the Theorem, but we could make no progress on it: is there an admissible Banach representation $\Pi$ of $D^{\times}$ such that its smooth vectors (or, more generally, its locally algebraic vectors) form an infinite dimensional vector space and such that $\Pi^{\la}$ has an infinitesimal character?
             
             \end{remar}
             
             \begin{remar}
             
               It is probably not unreasonable to think that ${\rm Hom}_{\gal}(V(\Pi), \JL(\Pi))$ is irreducible as a representation of $D^{\times}$, at least when $\Pi$ is ``generic'' (which in particular excludes the case when $\JL(\Pi)$ has locally algebraic vectors), but we cannot establish this. Recently Yongquan Hu and Haoran Wang proved this when $\Pi$ is attached to a generic irreducible $2$-dimensional Galois representation of global origin in \cite{HW} (see Theorem $1.1$ and Remark $1.2$ in loc.cit. for the precise statements and assumptions).                
        
        \end{remar}
        
        The second key input (and maybe the most surprising result of the paper) in the proof of the previous theorem is a new upper bound for the Gelfand--Kirillov (or canonical) dimension of Banach space representations $\Pi$ for which $\Pi^{\la}$ has an infinitesimal character. In order to avoid too many repetitions, let us call a Banach representation $\Pi$ \textit{quasi-simple} if $\Pi^{\la}$ has an infinitesimal character.
          
             Let $G$ be a connected reductive group over $\Qp$, let $H$ be an open subgroup of $G(\Qp)$ and let $d_G$ be the dimension of the flag variety of 
        $G_{\overline{\mathbb{Q}}_p}$. It is not difficult\footnote{Schmidt and Strauch \cite[Proposition 7.2]{SS} established this bound under  assumptions on the prime $p$ and the group $G$, as a consequence of deep work of Ardakov and Wadsley \cite{awannals}. In section \ref{proof_a} we give an elementary proof of this result, see Theorem \ref{astuce} (a), without imposing these assumptions and without making use of \cite{awannals}.} to see that 
        the GK-dimension of 
            any admissible locally analytic representation $\Pi$ of $H$ with an infinitesimal character cannot exceed $2d_G$. When $H$ is compact, it is also quite easy to see that this upper bound is optimal (see Remark \ref{optimal}). Even when $H$ is not compact, this upper bound can be optimal (for instance, if $H=\GL_2(\Qp)$ one can cook up examples using the completed cohomology of the tower of modular curves). The next Theorem shows that in the presence of an admissible {\it{Banach}} representation, this upper bound $2d_G$ is never attained. We refer the reader to Theorem \ref{astuce} for a more general statement. 
        
        \begin{thm}\label{GKdim}
         Let $G$ be a connected reductive group over $\Qp$ and let 
       $\Pi$ be an admissible quasi-simple Banach representation of an open subgroup of $G(\Qp)$. Then 
        the Gelfand--Kirillov dimension of $\Pi$ is {\emph{strictly}} less than twice the dimension of the flag variety of 
        $G_{\overline{\mathbb{Q}}_p}$.
        \end{thm}

      It is an interesting and intriguing problem to find the maximal possible value of the Gelfand--Kirillov dimension of an admissible quasi-simple Banach representation $\Pi$. It is maybe too optimistic to conjecture that the result is $d_G$, but we have no counter-example so far. 
         
         We mention the following simple consequences of the previous theorem (see Corollary \ref{gltwoqp} for the first one and Corollary 
         \ref{cor:fin_length} for the second one). 
                
                 \begin{cor}\label{funcor}
                  Let $\Pi$ be an admissible, unitary, quasi-simple Banach representation of 
                 $G:=\GL_2(\Qp)$ over $L$, having a central character. Let $\Theta$ be an $\OO$-lattice in 
                 $\Pi$ stable under $G$. Then $\Theta\otimes_{\OO} k$ has finite length, in particular $\Pi$ itself has finite length.
                                  
                \end{cor}
                
                \begin{remar}
                
          This result combined with \cite{DPS} (Section 9.7 and Theorem 9.20) gives a proof \emph{independent} of the 
         $p$-adic local Langlands correspondence for $\GL_2(\Qp)$ or the classification of irreducible smooth mod $p$ representations of  $\GL_2(\Qp)$ that Hecke eigenspaces (corresponding to weakly non-Eisenstein maximal ideals of the Hecke algebra, in the sense of section 9.7 of \cite{DPS}) in the completed cohomology of the tower of modular curves are Banach representations of finite length, and that their reduction modulo $p$ has finite length as well (but we cannot show using this method that Hecke eigenspaces modulo $p$ have finite length). 
         
         \end{remar}
         
         \begin{remar} Let $G=\GL_2(\Qp)$ and let $Z(\mathfrak{g})$ be the centre of the enveloping algebra of 
         ${\rm Lie}(G)$. It follows from the $p$-adic local Langlands correspondence that for an admissible unitary 
         $L$-Banach representation $\Pi$ of $G$, with a central character, the following statements are equivalent:
         
         $\bullet$ $\Pi^{\la}$ is $Z(\mathfrak{g})$-finite, i.e. killed by an ideal of finite codimension in $Z(\mathfrak{g})$. 
         
         $\bullet$ the ``reduction mod $p$'' of $\Pi$ has finite length, i.e. there is (equivalently for any) a 
         $G$-stable open and bounded $\OO$-lattice $\Theta\subset \Pi$ such that $\Theta\otimes_{\OO} k$ has finite length as a 
         $k[G]$-module.
         
         Neither of these statements seems to be easy to establish directly, but it is probably more approachable to establish the characteristic zero result. The Corollary above would then allow one to get the characteristic $p$ statement as well.

         \end{remar}
         
           Let us mention two other consequences of the bound on the Gelfand--Kirillov dimension. Let 
           $K$ be a finite extension of $\Qp$ and let $D$ be a quaternion algebra (we allow $D$ to be split) over 
           $K$. Let $G=D^{\times}$ and let $H$ be an open $K$-uniform subgroup of $G$, see section \ref{lazard_iso}. The space $\Pi^{\Kla}$ of $K$-locally analytic vectors in $\Pi$ has an increasing filtration by $H$-Banach space representations $\Pi^{\Kla}_{r_n}$ defined by radius of analyticity, see
           section \ref{rad_an}. (These subrepresentations are neither unitary nor admissible, and are not 
           closed in $\Pi^{\Kla}$.)
                 
           \begin{cor} 
                   Keep the above notations and let $\Pi$ be an admissible, quasi-simple Banach representation of  $G$, with a central character.

         a) The $H$-representation $\Pi^{\Kla}_{r_n}$ 
                  is topologically of finite length.
                                    
         b) If $K=\Qp$ and if $\Pi$ has no finite dimensional $H$-stable subquotient, then $\Pi$ has finite length as a topological $H$-representation. This is the case if the infinitesimal character of $\Pi^{\la}$ is not algebraic.

         \end{cor}
\begin{remar} When $\Pi$ is unitary and irreducible and $G=\GL_2(\Qp)$ Andrea Dotto has obtained 
more refined results \cite{dotto} on the structure of $\Pi$ as $H$-representation for $H=\GL_2(\Zp)$ 
or Iwahori subgroup with characteristic $p$ methods. 
\end{remar}              

\begin{remar}
Part b) of the above Corollary together with Theorem \ref{one_one} imply Theorem \ref{FL}. The hypothesis on the Hodge--Tate--Sen weights 
of $V(\Pi)$ in Theorem \ref{FL} is precisely what is needed to make sure that the infinitesimal character of $\Pi^{\la}$ is not algebraic (this implicitly uses Theorem $1.2$ in \cite{DAnn}, which is given a new proof in Theorem $9.29$ in \cite{DPS}). 
\end{remar}
    
        The proof of Theorem \ref{GKdim} relies on a fine study of the central reductions 
        \[D(H, L)_{\chi}:=D(H, L)\otimes_{Z(\mathfrak{g}_L), \chi} L\]
         and some of their completions, where 
        $H$ is a uniform pro-$p$ group in $G(\Qp)$, $\mathfrak{g}={\rm Lie}(H)$, $D(H,L)$ is the algebra of $L$-valued distributions on $H$ and 
        $\chi: Z(\mathfrak{g}_L)\to L$ is a morphism of $L$-algebras. Moreover, in order to exploit the 
        fact that we deal with Banach space representations, it is crucial to understand the relation between the above rings and the algebra $L\br{H}=L\otimes_{\mathcal{O}} \mathcal{O}\br{H}$ of the $L$-valued measures on $H$. Roughly speaking, the argument goes as follows: if 
        $\Pi$ has Gelfand--Kirillov dimension exactly $2d_G$, then standard results in homological algebra and work of Schneider--Teitelbaum show that \[{\rm Hom}_{L\br{H}}(\Pi^*, D(H, L)_{\chi})\ne 0.\] 
     % where $L\br{H}=L\otimes_{\mathcal{O}} \mathcal{O}\br{H}$ is the algebra of $L$-valued measures on $H$.        
        Here we are allowed to take $H$ as small as we like. Next, 
        $\Pi^*$ is a torsion $L\br{H}$-module (see Lemma \ref{usefullater}), thus we get a contradiction if $D(H, L)_{\chi}$ is a torsion-free $L\br{H}$-module. For this it suffices to ensure that $D(H, L)_{\chi}$ is an integral domain and that the natural map 
        $L\br{H}\to D(H, L)_{\chi}$ is injective. We cannot really prove these statements, but we prove just enough to make a version of the previous argument work. The actual argument is a bit painful, in particular there are some nontrivial
        reductions to the case where $G$ is semisimple, simply connected. In that case, we show that the required properties of the rings are consequences of the work of Ardakov \cite{ardakovast} and Ardakov--Wadsley \cite{awannals}, \cite{awverma} on localization theory for locally analytic representations (in order to show that certain rings are integral domains), as well as on affinoid Verma modules (in order to show that the map from the Iwasawa algebra to certain central reductions of completions of distribution algebras is injective). 
       
Since the most innovative result (Theorem \ref{astuce}) gets proved in the middle of the paper we feel 
the need to justify why we spend so much space on patching and Scholze's functor instead of 
just referring to \cite{scholze}. We correct a minor error in \cite{scholze}: although Scholze 
in \cite[Footnote 7]{scholze} writes \textit{It would be enough to assume that they have perfect resolutions by injective $H$-representations which are of `bounded complexity' in a suitable sense. As in our application, they will actually be injective, we restrict to this simpler setup}, in the actual application \cite[Corollary 9.3]{scholze} the representations are not injective, because in the setting 
of Shimura curves the global units in the centre have to act trivially on the cohomology groups. One way to avoid this issue is to work with unitary groups, but 
this would be inconvenient for us, as we would like to use the results of \cite{ludwig} proved in the setting of Shimura curves. We fix the issue in Theorem \ref{main_ultra}; again all the fundamental ideas in its proof are due to Scholze.  The second issue is that Scholze in \cite[Section 9]{scholze} patches 
cohomology instead of homology and to get the objects analogous to the patched module in \cite{6auth}
we would like to patch homology, so we spend some time in Section \ref{sec_ultra} discussing how Pontryagin duality interacts with localization with respect to an ideal defined by an ultrafilter. We patch following 
Dotto--Le \cite{dotto-le}, who in turn follow Gee--Newton \cite{gee_newton}; both of these papers 
use  a variation of Scholze's idea to patch using  ultrafilters. The third issue is that in Kisin's paper on the Fontaine--Mazur conjecture \cite{kisin_fmc} there is a problem, fixed in \cite[Appendix B]{gee-kisin}, when $p=3$ and the image of Galois representation is not 'big' enough. For this reason and mostly convenience Scholze assumes in \cite[Section 9]{scholze} that $p\ge 5$. 
Since the statement in Theorem \ref{one_one} is local we are free to choose a global setting to prove it, and we globalize the local Galois representation, using \cite[Appendix A]{gee-kisin}, so that it has 
`big image', and then the problem goes away and we can handle $p=3$, see Lemma \ref{diagonal} and the remark following it.  We wrote out  the details 
of these arguments since the results are needed in the forthcoming work of Colmez--Dospinescu--Niziol.     
        
        \emph{Notations}: We fix in all the article a prime $p$. We let $q=p$ if $p>2$ and $q=4$ otherwise and for $n\geq 0$ we let $r_n=p^{-1/p^n}$.  

      We fix a finite extension $L$ of $\Qp$,     
         the coefficient field of all representations considered in this paper. We let $\mathcal{O}$ be the ring of integers of 
        $L$, $\varpi\in\mathcal{O}$ a uniformiser and
        $k=\mathcal{O}/(\varpi)$ the residue field of $L$. The norm on $L$ extending the 
        $p$-adic norm on $\Qp$ is denoted $|\cdot|$.

        If $X$ is a topological $L$-vector space then we let $X^*$ be its topological 
        $L$-dual.

        If $K\subset L$ is a finite extension of $\Qp$, $G$ a locally $K$-analytic group and $\Pi$ a locally $\Qp$-analytic representation of $G$ over $L$, we denote $\Pi^{\Kla}$ the subspace of locally $K$-analytic vectors in $\Pi$, see \cite[Section 3]{Schmidt_vec} for more details. When $\Pi$ is a Banach representation of $G$ over $L$, we use the notation $\Pi^{\Kla}$ for $(\Pi^{\la})^{\Kla}$. 

        \subsection{Acknowledgement} We would like to thank Yongquan Hu for pointing out 
        a blunder in an earlier draft. We thank Toby Gee, Lue Pan and
        Peter Scholze for their comments on a preliminary version of
        the paper. We thank Konstantin Ardakov for simplifying the
        proof of theorem \ref{domaineasy}, and Olivier Ta\"ibi for
        several enlightening discussions around the results in this
        paper. We thank the referees for their very useful comments
        and suggestions.
        
             \section{Distribution algebras and completed enveloping algebras} 
             
              The proof of Theorem \ref{astuce} below requires some acrobatics with distribution algebras and their interpretation in terms of completed enveloping algebras, and the purpose of this paragraph is to recall some basic results and constructions concerning them. 
              
  \subsection{Distribution algebras of $p$-adic groups}\label{distri}            Let $K\subset L$ be finite extensions of $\Qp$.               If $G$ is a compact locally $K$-analytic group we let 
              $D(G,L)^{\Kla}$ be the algebra of $L$-valued locally 
              $K$-analytic distributions on $G$, i.e. the strong dual of the space of 
              $L$-valued locally $K$-analytic functions on $G$. 
                            When 
              $K=\Qp$ we drop the corresponding superscript. There is a natural embedding $L\otimes_K U(\mathfrak{g})\to D(G,L)^{\Kla}$, where $\mathfrak{g}={\rm Lie}(G)$ is the associated $K$-Lie algebra.

               Let 
              $G_0$ be $G$ viewed as a locally $\Qp$-analytic group
           and let $\mathfrak{g}_0$ be the Lie algebra of $G_0$. There is a natural isomorphism of $\Qp$-Lie algebras 
             $\iota: \mathfrak{g}\simeq \mathfrak{g}_0$             
              and we let $I$ be the kernel of the $\Qp$-linear map $K\otimes_{\Qp} \mathfrak{g}_0\to \mathfrak{g}$ sending $g\otimes X$ to $g\iota^{-1}(X)$. 
             There are natural surjective maps of $L$-algebras 
               \[D(G_0, L)\to D(G, L)^{\Kla}, \quad L\otimes_{\Qp} U(\mathfrak{g}_0)\to 
             L\otimes_{K} U(\mathfrak{g})\]
compatible with the embeddings $L\otimes_{\Qp} U(\mathfrak{g}_0)\to D(G_0, L)$ and 
$L\otimes_{K} U(\mathfrak{g})\to D(G, L)^{\Kla}$ and one shows \cite[Lemma 5.1]{SchmidtAR}
            that $I$ generates the kernel of these maps. 
                                       
From now on we assume that $G_0$ is a uniform pro-$p$ group and we let $g_1,\dots,g_d$ be topological generators of $G_0$, with $d=\dim_{\Qp} G_0$.
             Letting \[b^{\alpha}=(g_1-1)^{\alpha_1}\dots(g_d-1)^{\alpha_d}\in \Zp[G]\] for 
             $\alpha=(\alpha_1,\dots,\alpha_d)\in \Z_{\geq 0}^d$, each $\lambda\in D(G_0,L)$ has a unique convergent expansion \[\lambda=\sum_{\alpha\in \Z_{\geq 0}^d} c_{\alpha} b^{\alpha},\] where $c_{\alpha}\in L$ are such that $\lim_{|\alpha|\to\infty} |c_{\alpha}|\cdot r^{|\alpha|}=0$ for all $0\leq r<1$ (here $|\alpha|=\alpha_1+\dots+\alpha_d$). 
             
             If $n\geq 0$ set 
             $r_n=p^{-1/p^n}$, and let $q=p$ if $p>2$ and $q=4$ if $p=2$. The following results are proved in \cite{adm} and \cite{SchmidtAR}. One defines a multiplicative norm 
             $||\cdot||_{r_n}$ on $D(G_0,L)$ by setting 
             \[||\sum_{\alpha} c_{\alpha} b^{\alpha}||_{r_n}=\sup_{\alpha} |c_{\alpha}| q^{-\frac{|\alpha|}{p^n}}.\] The $L$-Banach algebra $D_{r_n}(G_0,L)$ obtained by completing 
             $D(G_0,L)$ with respect to this norm turns out to be Noetherian, the natural map 
             $D(G_0,L)\to D_{r_n}(G_0,L)$ is injective and we have an isomorphism of 
             $L$-Fr\'echet algebras  
             \[D(G_0,L)\simeq \varprojlim_{n} D_{r_n}(G_0,L).\]
            The elements of $D_{r_n}(G_0,L)$ have unique expansions 
            $\sum_{\alpha} c_{\alpha} b^{\alpha}$, whose coefficients $c_{\alpha}\in L$ satisfy
            $|c_{\alpha}| q^{-\frac{|\alpha|}{p^n}}\to 0$. Modding out by the ideal 
            $I=\ker(K\otimes_{\Qp} \mathfrak{g}_0\to \mathfrak{g})$ one obtains algebras 
            $D_{r_n}(G,L)^{\Kla}$ and an isomorphism $D(G,L)^{\Kla}\simeq 
            \varprojlim_{n} D_{r_n}(G,L)^{\Kla}$. The algebras $D(G,L)$ and $D(G,L)^{\Kla}$ are Fr\'echet-Stein algebras in the sense of \cite{adm}.

\subsection{Lazard's isomorphism}\label{lazard_iso}

  Recall that one can attach a $\Zp$-Lie algebra $L_H$ to any uniform 
  pro-$p$ group $H$, with underlying set $L_H=H$ and operations induced\footnote{More precisely $g+h=\lim_{n\to \infty} (g^{p^n}h^{p^n})^{p^{-n}}$ and 
  $[g,h]=\lim_{n\to\infty} (g^{-p^n}h^{-p^n}g^{p^n}h^{p^n})^{p^{-2n}}$}
 by those in 
  $H$. For psychological reasons we prefer to distinguish between $L_H$ and $H$, so we write 
  $X_h\in L_H$ for the element corresponding to $h\in H$. The construction $H\mapsto L_H$ 
  induces an equivalence of categories between uniform pro-$p$ groups and $\Zp$-Lie algebras $\mathfrak{g}$ which are finite free $\Zp$-modules and satisfy $[\mathfrak{g}, \mathfrak{g}]\subset q\mathfrak{g}$, using the standard notation $q=p$ for $p>2$ and $q=4$ if $p=2$. In particular $q^{-1}L_H$ is a $\Zp$-Lie algebra for any uniform pro-$p$ group $H$. The logarithm map of the Lie group
  $H$ combined with the bijection between $H$ and $L_H$ induce an injective map of 
  $\Zp$-Lie algebras $L_H\to {\rm Lie}(H)$, which in turn induces an isomorphism 
  $\Q_p\otimes_{\Zp} L_H\simeq {\rm Lie}(H)$. 
  
   As above, let $K\subset L$ be finite extensions of $\Qp$. 
 The locally $K$-analytic group $H$ is 
called $K$-uniform if $H$ is a uniform pro-$p$ group and if $L_H$ is a sub-$\OO_K$-module of ${\rm Lie}(H)$. Let $H$ be a $K$-uniform group and write for simplicity\footnote{As everywhere in this paper $\OO=\OO_L$.}
\[\mathfrak{g}_{H,\OO}=\OO\otimes_{\OO_K} q^{-1}L_H,\]
an $\OO$-Lie algebra, finite free over $\OO$ of rank equal to the dimension of the locally $K$-analytic group $H$. The natural embedding $\mathfrak{g}_{H,\OO}\to L\otimes_K {\rm Lie}(H)$ yields a map
\begin{equation}\label{eq:iota} \iota: U(\mathfrak{g}_{H,\OO})\to L\otimes_K U({\rm Lie}(H))\to D(H, L)^{\Kla}\to D_{1/p}(H,L)^{\Kla}.\end{equation}
Explicitly, for $h\in H$ we have (by identifying $h$ with the associated Dirac distribution)
\[\iota(1\otimes q^{-1}X_h)=q^{-1} \log(h):=-q^{-1}\sum_{n\geq 1} \frac{(1-h)^n}{n}.\]
Using this formula, one easily checks that the image lands in the unit ball of $D_{1/p}(H,L)^{\Kla}$ and so the map $\iota$ extends by continuity to a map of 
$L$-Banach algebras 
\[\iota: \widehat{U(\mathfrak{g}_{H,\OO})}[1/p]\to D_{1/p}(H,L)^{\Kla},\]
the completion involved here being $p$-adic. The construction is functorial in 
$H$. The following result is classical, at least if $p>2$, and we refer to \cite[Theorem 6.5.3]{ardakovast}  
 for the general case. 

\begin{prop}\label{Laz} (Lazard's isomorphism) The map $\iota$ constructed above is an isomorphism of 
$L$-Banach algebras 
\[\widehat{U(\mathfrak{g}_{H,\OO})}[1/p]\simeq D_{1/p}(H,L)^{\Kla}.\]

\end{prop}

\begin{remar} a) Note that $H^{p^n}$ is also $K$-uniform and $\mathfrak{g}_{H^{p^n},\OO}=p^n \mathfrak{g}_{H,\OO}$. Applying Lazard's isomorphism to $H^{p^n}$ we obtain compatible isomorphisms of $L$-Banach algebras
   \[D_{1/p}(H^{p^n}, L)^{\Kla}\simeq \widehat{U(p^n\mathfrak{g}_{H,\OO})}[1/p].\]

b) There is a natural map $L\br{H}\to D(H,L)\to D_{1/p}(H,L)^{\Kla}$, and the composite with Lazard's isomorphism $L\br{H}\to \widehat{U(\mathfrak{g}_{H,\OO})}[1/p]$ is uniquely determined by the fact that the image of 
 $h\in H$ (seen as Dirac measure in $L\br{H}$) is 
 \[e^{q\cdot \frac{X_h}{q}}:=\sum_{n\geq 0} \frac{q^n}{n!} \left(1\otimes \frac{X_h}{q}\right)^n.\]

 \end{remar}

    We end this paragraph with another very useful and standard (at least for $p>2$, see \cite[Theorem 6.5.11]{ardakovast} for the general case) result.   

\begin{prop}\label{distrib} For any $K$-uniform group $H$ 
 there are compatible embeddings 
\[L\br{H^{p^n}}\to L\br{H},\,\, D_{1/p}(H^{p^n}, L)^{\Kla}\to D_{r_n}(H,L)^{\Kla}\] for 
$n\geq 1$, making the right-hand side a finite free module over the left-hand side with basis given by (the Dirac measures of) any system of representatives of $H/H^{p^n}$.
    \end{prop}

  Combining this proposition and the previous discussion, we deduce that for any $K$-uniform group 
  $H$ we have a natural embedding $\widehat{U(p^n\mathfrak{g}_{H,\OO})}[1/p]\to D_{r_n}(H,L)^{\Kla}$ inducing an isomorphism of left $\widehat{U(p^n\mathfrak{g}_{H,\OO})}[1/p]$-modules
  \[D_{r_n}(H,L)^{\Kla}\simeq \bigoplus_{h\in H/H^{p^n}} \widehat{U(p^n\mathfrak{g}_{H,\OO})}[1/p] \delta_h.\]

\subsection{Radius of analyticity}\label{rad_an} Let $\Pi$ be an admissible Banach representation of a locally $\Q_p$-analytic group $G$ over $L$ and fix an open pro-$p$ uniform subgroup $H$ of $G$. If $g_1,\ldots, g_d$ is a minimal system of topological generators of $H$, we let $\Pi_{r_n}^{\la}$ be the subspace of $\Pi$ consisting of those vectors $v\in \Pi$ "of radius of analyticity $r_n$", i.e. such that (the notation $b^{\alpha}$ is the one introduced in paragraph \ref{distri}) the sequence of real numbers $(q^{\frac{|\alpha|}{p^n}} ||b^{\alpha}v||)_{\alpha\in \Z_{\geq 0}^d}$ is 
bounded. 
It is a Banach space for the norm 
\[||v||_{r_n}=\sup_{\alpha} q^{\frac{|\alpha|}{p^n}} ||b^{\alpha}v||,\]
where $||\cdot||$ is the norm on $\Pi$. 
One can check (by using the same arguments as in chapter IV of \cite{CD}) that $\Pi_{r_n}^{\la}$ does not depend on the choice of the global coordinates $g_1,\ldots,g_d$ on $H$, but we warn the reader that $\Pi_{r_n}^{\la}$ \emph{does depend} on 
the subgroup $H$, thus the notation is slightly misleading. By construction $\Pi_{r_n}^{\la}$ is a 
topological $D_{r_n}(H,L)$-module and one checks that 
there is a natural isomorphism of topological $D_{r_n}(H,L)$-modules 
\[(\Pi_{r_n}^{\la})^*\simeq D_{r_n}(H,L)\otimes_{L\br{H}} \Pi^*.\]

 If $G$ is a locally $K$-analytic group and $H$ is a $K$-uniform open subgroup of $G$, we define 
 $\Pi^{\Kla}_{r_n}$ as $\Pi^{\Kla}\cap \Pi^{\la}_{r_n}$, where $\Pi^{\la}_{r_n}$ is defined with respect to $H$ seen as a uniform pro-$p$ subgroup of $G$.  One then has an isomorphism of topological $D_{r_n}(H,L)^{\Kla}$-modules
 \[(\Pi^{\Kla}_{r_n})^{*}\simeq D_{r_n}(H,L)^{\Kla}\otimes_{L\br{H}} \Pi^*.\]
 The reader who prefers working with modules over the distribution algebras can simply ignore the previous discussion and take the last isomorphism as a definition of the left-hand side. 
                
 \section{Dimension theory}\label{GK}  We recall in this rather long paragraph some standard results 
 concerning the dimension theory of finitely generated modules over Auslander-regular and 
 Aus\-lan\-der--Gorenstein rings. We will use these results systematically in the next paragraphs, so we decided to make the presentation relatively self-contained.
 
 \subsection{Auslander--Gorenstein rings and the dimension filtration}
 
   We start by recalling some of the main features of (non commutative) Auslander--Gorenstein rings. All rings below will be associative, unital and Noetherian, the only exception being paragraph \ref{STdim}, where non Noetherian rings will appear. We write ${\rm Mod}^{fg}(A)$ for the category of finitely generated (left) $A$-modules.
      If $A$ is a ring and $M$ is an $A$-module, we set $E^q(M)=E_A^q(M)={\rm Ext}^q_A(M,A)$ and denote by 
      \[j(M)=j_A(M)=\inf \{q| \,E^q(M)\ne 0\}\in \Z_{\geq 0}\cup \{\infty\}\]
      the grade of $M$ (see \cite[sec.1]{levasseur} and \cite [III.2.1] {zarfilt} for more information). 
      Clearly $j(0)=\infty$. On the other hand, if $A$ has finite injective dimension $d$, which will always be the case in our applications below, one has $j(M)\leq d$ for all nonzero $M\in {\rm Mod}^{fg}(A)$. The following result is standard, see for instance Theorem $8.34$ in \cite{Rotman}.
      
        \begin{lem}\label{Rees} (Rees' Lemma) If $x\in A$ is central, not a unit and not a zero-divisor, then 
    for all $M\in {\rm Mod}^{fg}(A/xA)$ and all $q\geq 0$ we have $E_A^{q+1}(M)\simeq E^{q}_{A/xA}(M)$, in particular $j_{A/xA}(M)=j_A(M)-1$. 
          \end{lem}

      We say that $A$ is Auslander--Gorenstein of dimension $d<\infty$ if 
   ${\rm injdim}(A)=d$ and for all $M\in {\rm Mod}^{fg}(A)$, all $q\geq 0$ and all 
   $A$-submodules $N\subset E^q(M)$ we have $j(N)\geq q$. If moreover the global homological dimension of $A$ is finite, we say that $A$ is Auslander regular. When $A$ is commutative, 
   $A$ is Auslander--Gorenstein (resp. Auslander-regular) if and only if $A$ is Gorenstein (resp. regular). The following result summarizes standard properties of Auslander--Gorenstein rings and we refer the reader to sections 2 and 4 of \cite{levasseur} for the proofs. The filtration appearing in part b) is called the dimension filtration of $M$. 
   
   \begin{prop}\label{toolbox}
    Let $A$ be an Auslander--Gorenstein ring of dimension $d$ and let 
    $M$ be a finitely generated $A$-module. 
    
    a) For any sub-module $N$ of $M$ we have 
      $j(M)=\min(j(N), j(M/N))$. 
    
    b) Suppose that $M\ne 0$ and let $F^qM$ be the largest sub-module $X$ of $M$ such that 
    $j(X)\geq q$. Then $F^{d+1}M=0$, $j(M)$ is the largest $q$ such that $M=F^qM$ and there are exact sequences with $j(Q(q))\geq q+2$
    \[0\to F^qM/F^{q+1}M\to E^q(E^q(M))\to Q(q)\to 0.\]

    c) If $\dots\subset M_1\subset M_0=M$ is a chain of sub-modules of $M$, then $j(M_{i}/M_{i+1})\geq j(M)+1$ for all $i$ large enough. 
    
    d) If $x\in A$ is central and not a zero-divisor, then $A/xA$ is Auslander--Gorenstein of dimension $\leq d-1$. 
       \end{prop}
      
   \begin{remar}  Under stronger assumptions on $A$ one can say more, cf. Theorem 3.6 and Corollary 4.4 in \cite{levasseur}: if $A=\oplus_{n\geq 0} A_n$
   is positively graded, $x\in A_k$ (for some $k>0$) is a central non zero-divisor, then 
   $A$ is Auslander--Gorenstein of dimension $d$ if and only if $A/xA$ is Auslander--Gorenstein of dimension $d-1$, and then for any $M\in {\rm Mod}^{fg}(A)$ on which $x$ is regular (i.e. injective) we have $j_A(M/xM)=j_A(M)+1$. 
        \end{remar}

     \subsection{Filtrations and deformations}\label{Fildef}
     
              By convention all filtrations 
        $FA=(F_nA)_{n\in \Z}$ on a ring $A$ will be {\it{increasing}}, exhaustive (i.e. $A=\cup_{n\in \Z} F_nA$)
        and such that $1\in F_0A$ (thus $F_0A$ is a subring of $A$). Denote by 
        \[{\rm Gr}(A)=\bigoplus_{n} F_nA/F_{n-1}A, \quad R(A)=\bigoplus_{n} F_nA \cdot T^n\subset A[T,T^{-1}]\] the associated graded, respectively Rees ring. We say that $FA$ is a Zariskian filtration 
    if $R(A)$ is (left and right) Noetherian
        and if $F_{-1}A$ is contained in the Jacobson radical of $F_0A$. This is the case when 
        ${\rm Gr}(A)$ is noetherian and $FA$ is complete (e.g. positive or more generally discrete\footnote{A filtration $FM$ on an $A$-module $M$ is called positive if 
        $F_nM=0$ for all $n<0$, and it is called discrete if there is $n_0$ such that $F_nM=0$ for all $n<n_0$.}), which is the only case we really need.  A filtration $FM$ on $M\in {\rm Mod}^{fg}(A)$ is called \textit{good} 
         if one can find $m_1,...,m_k\in M$ and integers 
        $i_1,\dots,i_k$ such that \[F_nM=F_{n-i_1}A\cdot m_1+\dots+F_{n-i_k}A\cdot m_k\] for all integers $n$. 
      When the filtration on $A$ is Zariskian
        any good filtration on a finitely generated $A$-module is separated and induces a good filtration on any sub-module. If the filtration on $A$ is complete, a separated filtration 
        $FM$ on $M$ is good if and only if ${\rm Gr}(M)$ is finitely generated over ${\rm Gr}(A)$.
        See \cite{zarfilt} for these standard results. 
         The following fundamental result is due to Bjork, cf. Theorem
         3.9 in \cite{Bjork} and its proof.
        
        \begin{thm}\label{beurk}
         If $A$ is a ring with a Zariskian filtration for which ${\rm Gr}(A)$ is Aus\-lan\-der--Go\-ren\-stein (resp. Auslander-regular), then $A$ is Aus\-lan\-der--Go\-ren\-stein (resp. Aus\-lan\-der-re\-gu\-lar) and for any good filtration $FM$ on a finitely generated 
         $A$-module $M$ we have $j_A(M)=j_{{\rm Gr}(A)}({\rm Gr}(M))$. 
        \end{thm}
        
        Before giving some concrete applications of the previous theorem, we recall that 
         Ardakov and Wadsley introduced and studied in \cite{awannals} certain categories of filtered and doubly filtered $\OO$-algebras. We will restrict here to certain sub-categories, which seem to be the ones naturally encountered in practice.
       The first category\footnote{This is a sub-category of the category of {\it{deformable}} $\OO$-algebras introduced in \cite{awannals} (we impose the extra condition that ${\rm Gr}(A)$ is commutative and Noetherian).} ${\rm Def}(\OO)$ is that of 
        positively filtered $\OO$-algebras $A$ such that 
        $F_0A$ is an $\OO$-sub-algebra of $A$ and 
        ${\rm Gr}(A)$ is a commutative, Noetherian, flat $\OO$-algebra, morphisms being defined in the obvious way. For instance $U(\mathfrak{g})\in {\rm Def}(\OO)$ with its natural PBW filtration for any Lie algebra $\mathfrak{g}$ over $\OO$, which is finite free as $\OO$-module (indeed, ${\rm Gr}(U(\mathfrak{g}))\simeq S(\mathfrak{g})$ by the PBW Theorem). Another example used later on is that of the ring of crystalline 
        differential operators $\mathcal{D}(X)$ on a smooth affine scheme 
        $X/\OO$. This is a ring generated by $\OO(X)$ and the global vector fields 
        $\mathcal{T}(X)$, subject to obvious relations. For the natural filtration by order of differential operators on $\mathcal{D}(X)$ we have a natural isomorphism 
        ${\rm Gr}(\mathcal{D}(X))\simeq {\rm Sym}_{\OO(X)}(\mathcal{T}(X))$, showing that indeed $\mathcal{D}(X)\in {\rm Def}(\OO)$.
               Ardakov and Wadsley proved in \cite{awannals} the existence of an endo-functor for each $n\geq 0$, called the $n$th order deformation functor 
         \[{\rm Def}(\OO)\to {\rm Def}(\OO),\quad A_n=\sum_{i\geq 0} \varpi^{in} F_iA,\]
      coming with a natural isomorphism ${\rm Gr}(A)\simeq {\rm Gr}(A_n)$ (we endow $A_n$ with the induced filtration from $A$). For instance, for $\mathfrak{g}$ as above we have 
        $U(\mathfrak{g})_n=U(\varpi^n \mathfrak{g})$. 
         
         Consider now the category\footnote{For any $A\in CDF(\OO)$ the $L$-algebra $A[1/\varpi]$ is a complete doubly filtered $L$-algebra in the sense of \cite[def. 3.1]{awannals}, explaining our exotic notation $CDF(\OO)$.} 
             $CDF(\OO)$ of flat $\OO$-algebras 
         $A$ which are separated and complete for the $\varpi$-adic topology, 
          together with a complete filtration $F(A/\varpi A)$ on $A/\varpi A$ such that
         ${\rm Gr}(A/\varpi A)$ is Noetherian and commutative. Any
          $A\in CDF(\OO)$ is left and right Noetherian, since the filtration on 
         $A/\varpi A$ is Zariskian (thus $A/\varpi A$ is a  
          left and right Noetherian ring, and then so is $A$). 
                 Morphisms in $CDF(\OO)$ are defined in the obvious way\footnote{i.e. morphisms of $\OO$-algebras inducing filtered 
         morphisms between the reductions mod $\varpi$.}. There is a natural functor 
         \[{\rm Def}(\OO)\to CDF(\OO), \quad A\mapsto \hat{A}:=\varprojlim_{n} A/\varpi^n A,\]
         where
         $\hat{A}/\varpi\hat{A}\simeq A/\varpi A$ is endowed with the quotient filtration from 
         $A$, so that, using the $\OO$-flatness of ${\rm Gr}(A)$, ${\rm Gr}(\hat{A}/\varpi \hat{A})\simeq  {\rm Gr}(A)/\varpi {\rm Gr}(A)$. 
 The next result is essentially \cite[Theorem 3.3]{awannals}, but we give a proof for the reader's convenience. 
                          
        \begin{prop}\label{Bjork} Let $A\in CDF(\OO)$ be such that ${\rm Gr}(A/\varpi A)$ is Gorenstein (resp. regular). Then 
       
       a)  $A$ is Auslander--Gorenstein (resp. Auslander-regular). 
       
       b) If $M\in {\rm Mod}^{fg}(A)$ is 
      flat over 
        $\OO$ and if $F(M/\varpi M)$ is a good filtration on $M/\varpi M$, then 
        \[j_{A}(M)=j_{A[1/p]} (M[1/p])=j_{A/\varpi A}(M/\varpi M)=j_{{\rm Gr}(A/\varpi A)}({\rm Gr}(M/\varpi M)).\]
             \end{prop}
             
             \begin{proof} a) The filtration on $A/\varpi A$ being Zariskian, Theorem
             \ref{beurk} shows that $A/\varpi A$ is Auslander--Gorenstein (resp. Auslander-regular) under the asserted hypothesis. Apply now the same result to $A$ endowed with the Zariskian $\varpi$-adic filtration (defined by $F_nA=A$ for $n\geq 0$ and $F_n A=\varpi^{-n}A$ for $n\leq 0$), observing that\footnote{The first isomorphism follows from the definition of ${\rm Gr}(A)$, the second one sends $\sum_{n\geq 0} \pi_{n+1}(\varpi^{n} x_n)$ to $\sum_{n\geq 0} \pi_1(x_n) T^n$, where $\pi_n: A\to A/\varpi^nA$ is the natural projection.}
             \[{\rm Gr}(A)\simeq \bigoplus_{n\geq 0} \varpi^n A/\varpi^{n+1}A\simeq (A/\varpi A)[T] \]
              since $A$ is $\OO$-flat and that 
             $R[T]$ is Auslander--Gorenstein (resp. Auslander-regular) whenever $R$ is so. 
             
             b) The last equality follows directly from Theorem \ref{beurk}.
                          Let 
             $a=j_A(M)$.
             Since \[E_{A[1/p]}^q(M[1/p])\simeq E_A^q(M)[1/p],\] we have
             $j_{A[1/p]}(M[1/p])\geq a$ and in order to show that we
             have equality it suffices to show that $E_A^a(M)$ is
             $\OO$-flat. The long exact sequence associated to $0\to
             M\to M\to M/\varpi M\to 0$ reduces this to showing that
             $E^a(M/\varpi M)=0$, which holds since $j_{A}(M/\varpi
             M)\geq 1+j_A(M)$ by part c) of Proposition \ref{toolbox}. Next, we have $j_{A/\varpi A}(M/\varpi M)=j_A(M/\varpi M)-1$ by Lemma \ref{Rees}, so we are done if we prove that $j_{A}(M/\varpi M)\geq 1+j_A(M)=1+a$ is an equality. But if $j_A(M/\varpi M)>1+a$ then $E_A^{a+1}(M/\varpi M)=0$ and the long exact sequence above gives $E_A^a(M)=\varpi E_A^a(M)$, which by Nakayama's Lemma yields $E_A^a(M)=0$, a contradiction.                          \end{proof}
             
               \begin{examp}\label{example1}  Letting  
                  $A=U(\mathfrak{g})$ with $\mathfrak{g}$ as above, we have 
          \[\widehat{A_n}=\{\sum_{i\in \mathbb{N}^d} a_i X^i| \,
          a_i\in \varpi^{n|i|}\OO, \lim_{|i|\to\infty} a_i\cdot
          \varpi^{-n|i|}=0\},\]
    where $X_1,\dots,X_d$ is an $\OO$-basis of $\mathfrak{g}$, $X^i=X_1^{i_1}\dots X_d^{i_d}$ and $|i|=i_1+\dots+i_d$ for $i=(i_1,\dots,i_d)\in \mathbb{N}$. The previous discussion shows that \[{\rm Gr}(\widehat{A_n}/\varpi \widehat{A_n} )\simeq S(\mathfrak{g})\otimes_{\OO} k\simeq S(\mathfrak{g}\otimes_{\OO} k),\]
   and Theorem \ref{beurk}
         yields that $\widehat{A_n}$ are Auslander-regular. One can show \cite[chap. 9]{awannals}
         that its global homological dimension is $1+{\rm rk}_{\OO}(\mathfrak{g})$. 
\end{examp}

\begin{examp}\label{example2}
Consider a uniform pro-$p$ group $G$ and let $A=\OO\br{G}$ be its Iwasawa algebra, a local ring with maximal ideal $\mathfrak{m}$, the augmentation ideal. Endow 
$A/\varpi A=k\br{G}$ with the $\mathfrak{m}$-adic filtration. 
         By a fundamental theorem of Lazard 
         \[{\rm Gr}(A/\varpi)\simeq k[X_1,\dots,X_d], \quad d:=\dim G,\]
         in particular           
         $A\in CDF(\OO)$ is Auslander-regular (a result of Venjakob, cf. \cite[Theorem 3.26]{venjakob}) and one can show that its global homological dimension is $1+\dim G$. Also $k\br{G}$ and $L\br{G}:=L\otimes_{\OO} A$ are Auslander-regular of dimension 
         $\dim G$.
         \end{examp}

\subsection{Dimension theory for Fr\'echet-Stein algebras}\label{STdim}
          Consider a two-sided Fr\'e\-chet--Stein algebra $A=\varprojlim_{n} A_n$ over $L$ and a nonzero coadmissible $A$-module $M$ (see \cite{adm} for the definition). Let 
      $M_n=A_n\otimes_A M$ be the associated $A_n$-modules, so that $M\simeq \varprojlim_n M_n$.  Schneider and Teitelbaum prove in \cite[chap. 8]{adm} the following results. First of all, the module $E^l(M)={\rm Ext}_A^l(M,A)$ is a coadmissible right $A$-module and there are natural isomorphisms 
      \[E^l(M)\otimes_A A_n\simeq {\rm Ext}^l_{A_n}(M_n, A_n)\]
      for all $n$. Assume from now on that $A_n$ are Auslander-regular
      of global dimension uniformly (in $n$) bounded by some $d\geq 0$. Then the sequence $(j_{A_n}(M_n))_{n\geq 0}$ is nonincreasing and eventually constant with value 
      $j_A(M)\leq d$, for any coadmissible $A$-submodule $N$ of $M$ we have 
      $j_A(M)=\min(j_A(N), j_A(M/N))$ and any coadmissible submodule $L$ of $E^l(M)$ satisfies 
      $j_A(L)\geq l$. Finally, $M$ has a canonical dimension filtration 
      $0=M^{d+1}\subset\dots\subset M^1\subset M^0=M$ by coadmissible $A$-modules such that 
      $(A_n\otimes_A M^i)_{0\leq i\leq d+1}$ is the dimension filtration of the $A_n$-module 
      $M_n$ over the Auslander-regular ring $A_n$. In other words $A$ itself has all key properties of an Auslander-regular ring (but it is not necessarily noetherian). 
            
      The previous conditions are satisfied by the Fr\'echet-Stein algebra $D(G,L)^{\Kla}$ for any compact locally $K$-analytic group $G$, by a result of Schmidt \cite{SchmidtAR} (the case $K=\Qp$ being one of the main results of \cite{adm}). 
      
      \subsection{Dimensions of mod $p$, Banach and locally analytic representations}\label{Dim}
      
        Let $G$ be a locally $\Q_p$-analytic group and let $\pi$ be an admissible smooth representation of 
        $G$ over $k$. Then $\pi^*$ is a finitely generated module over $k\br{H}$ for any open uniform pro-$p$ subgroup 
        $H$ of $G$ (which exists by classical results of Lazard) and we call 
        \[d(\pi)=\dim H-j_{k\br{H}}(\pi^*)\]
        the dimension of $\pi$. This is independent of the choice of $H$ and should not be confused with the dimension of $\pi$ as a $k$-vector space since the latter is 
        infinite in almost all applications. 
 One can also interpret $d(\pi)$ as a sort of Gelfand--Kirillov dimension (see \cite[Proposition 5.4]{ardakov_brown}, \cite[Theorem 18]{CE} 
        and \cite[Proposition 2.18]{capture}): given an open pro-$p$ uniform subgroup $H$ of $G$, there are constants 
        $a,b>0$ such that for all $n$ big enough
        \[b\cdot p^{d(\pi) n}\leq \dim_{k} \pi^{H_n}\leq a\cdot p^{d(\pi)n},\]
        where $H_n$ is the closed subgroup of $H$ generated by the 
        $h^{p^n}$, $h\in H$. This interpretation clearly shows that $\dim_k(\pi)<\infty$ if and only if $d(\pi)=0$,  but the latter result also follows easily from Proposition \ref{Bjork} and Example 
        \ref{example2}. 
         Note that $[H: H_n]=p^{n\dim H}$, so the previous result yields
        \[\lim_{n\to \infty} \frac{\log \dim_k \pi^{H_n}}{\log [H:H_n]}=\frac{d(\pi)}{\dim H}.\]
        
        \begin{remar}
        It is well-known that if $\pi$ is an absolutely irreducible smooth representation of ${\rm GL}_2(\Qp)$, then 
        $d(\pi)\leq 1$, with equality unless $\pi$ is a character. This uses the classification due to Barthel--Livn\'e \cite{BL} and Breuil \cite{Breuil} and some decomposition of supersingular representations due to Pa\v{s}k\={u}nas (see Proposition 4.7 and Theorem 6.3 in \cite{extensions}). Such a result is not known already for ${\rm GL}_2(F)$ when $F$ is any non trivial extension of $\Qp$. Recent work of Breuil, Herzig, Hu, Morra and Schraen \cite{BHHMS} and Hu--Wang \cite{HH} shows that for $F$ unramified over $\Qp$, sufficiently generic Hecke eigenspaces in the completed cohomology of suitable Shimura curves have Gelfand--Kirillov dimension $[F: \Qp]$. See also \cite{SS} for other explicit examples
        in the Banach and locally analytic setting. 
        \end{remar}

              Similarly, if $\Pi$ is an admissible Banach representation of $G$ over $L$ we set 
        \[d(\Pi)=\dim H-j_{L\br{H}}(\Pi^*),\]
        a number independent of $H$.
        If $\Theta$ is an $H$-stable lattice in $\Pi$, which always exists by compactness of $H$, then 
        $\pi:=\Theta\otimes_{\OO} k$ is an admissible smooth representation of $H$ over $k$ and we have 
        $d(\Pi)=d(\pi)$, as follows from Proposition \ref{Bjork}. In particular $\dim_{L} (\Pi)<\infty$ if and only if 
        $d(\Pi)=0$.
        
        Finally, if $\Pi$ is an admissible locally analytic representation of $G$ over $L$, we set 
        \[d(\Pi)=\dim H-j_{D(H,L)}(\Pi^*)=\dim H-\min_{n} j_{D_{r_n}(H,L)}(D_{r_n}(H,L)\otimes_{D(H,L)} \Pi^*),\]
        the second equality being a consequence of results recalled in the previous paragraph. One subtle difference with respect to the mod $p$ and Banach setting is that one may have 
        $d(\Pi)=0$ and yet $\dim_{L} (\Pi)=\infty$. Actually $d(\Pi)=0$ if and only if $D_{r_n}(H,L)\otimes_{D(H,L)} \Pi^*$ are all finite dimensional over $L$. In somewhat more concrete terms, the subspace of vectors of a given radius of analyticity in $\Pi$ is finite dimensional. For instance $d(\Pi)=0$ for any admissible smooth representation $\Pi$ of $G$ over $L$. If $\Pi$ is an admissible Banach representation of $G$, then the space of locally analytic vectors $\Pi^{\la}$ in $\Pi$ is an admissible locally analytic representation 
        of $G$ and we have an isomorphism $(\Pi^{\la})^*\simeq D(H,L)\otimes_{L\br{H}} \Pi^*$. Moreover, the ring map 
        $L\br{H}\to D(H,L)$ is faithfully flat (these last two statements are the main results of \cite{adm}), so that standard homological algebra yields 
        $d(\Pi)=d(\Pi^{\la})$.
      Let us summarise some of the observations above:
        
        \begin{lem}\label{usefullater}
         If $\Pi$ is an admissible Banach representation of a locally $\Qp$-analytic group $G$, then 
         $d(\Pi)=d(\Pi^{\la})$, and we have $d(\Pi)=0$ if and only if $\dim_{L}(\Pi)<\infty$. Moreover, 
         if $d(\Pi)<\dim G$, then $\Pi^*$ is a torsion $L\br{H}$-module for all compact open subgroups 
         $H$ of $G$.
        \end{lem}     
        
        \begin{proof}
        Only the last sentence does not follow from the previous discussion. 
        Passing to a finite index subgroup, we may assume that $H$ is uniform pro-$p$. If $\ell\in \Pi^*$ is not 
        $L\br{H}$-torsion, then the natural morphism $L\br{H}\to \Pi^*$ sending $\lambda$ to $\lambda\ell$
         is injective, thus $j(L\br{H})\geq j(\Pi^*)>0$ (use Proposition \ref{toolbox}), a contradiction. 
                \end{proof}     

        From now on we assume moreover that $G$ is a 
       locally $K$-analytic group and $\Pi$ is an admissible locally $K$-analytic representation of $G$. Set
        \begin{equation}
        \begin{split}
        d^K(\Pi)&=\dim_K H-j_{D(H,L)^{\Kla}}(\Pi^*)\\
        &=\dim_K H-\min_{n} j_{D_{r_n}(H,L)^{\Kla}}(D_{r_n}(H,L)^{\Kla}\otimes_{D(H,L)^{\Kla}} \Pi^*),
        \end{split}
        \end{equation}
        where $H$ is an open $K$-uniform pro-$p$-subgroup of $G$ and $\dim_KH$ is the dimension of $H$ as a locally $K$-analytic variety.    When $\Pi$ is a locally analytic representation of $G$ we have the following relation between $d^K(\Pi^{\Kla})$ and $d(\Pi)$.
         
         \begin{lem}\label{commm} Let $R\twoheadrightarrow S$ be a surjection of regular commutative noetherian domains and let $M$ be a finitely generated $R$-module. Then 
         \[  \dim M -(\dim R - \dim S)\le \dim (M\otimes_R S)\le \dim M ,\]
         where $\dim$ denotes the Krull dimension.
         \end{lem}
         \begin{proof} By localizing at maximal ideals of $S$ we may assume that $R$ and $S$ 
         are both local. Since $R$ and $S$ are both regular the kernel of $R\twoheadrightarrow S$ 
         is generated by a regular sequence $x_1, \ldots, x_r$ with $r= \dim R - \dim S$. Since 
         $\dim M/xM$ is equal to either $\dim M$ or $\dim M -1$ we obtain the assertion.
         \end{proof}
         
        \begin{lem}\label{lemm:comp_dim}
          Let $\Pi$ be an admissible locally $\Qp$-analytic representation of $G$. We have
          \[ d(\Pi)-(\dim H-\dim_K H)\le d^K(\Pi^{\Kla})\leq d(\Pi).\]
        \end{lem}

        \begin{proof} We may pick $r=r_n$ such that $ j_{D(H,L)}(\Pi^*)=j_{D_{r}(H,L)}(D_{r}(H,L)\otimes_{D(H,L)} \Pi^*)$ and 
        \[j_{D(H,L)^{\Kla}}((\Pi^{\Kla})^*)=j_{D_{r}(H,L)^{\Kla}}(D_{r}(H,L)^{\Kla}\otimes_{D(H,L)^{\Kla}} (\Pi^{\Kla})^*).\]
        Since $(\Pi^{\Kla})^*= D(H,L)^{\Kla}\otimes_{D(H, L)} \Pi^*$, which follows from \cite[Theorem 7.1]{adm} together with \cite[Proposition 3.4]{Schmidt_vec},
        we have 
        \[D_{r}(H,L)^{\Kla}\otimes_{D(H,L)^{\Kla}} (\Pi^{\Kla})^*\cong 
        D_{r}(H,L)^{\Kla}\otimes_{D(H,L)} \Pi^*.\]
        Since $D(H,L)\rightarrow D_r(H, L)^{\Kla}$ factors through $D_r(H, L)\rightarrow D_r(H, L)^{\Kla}$ we have
        \[D_{r}(H,L)^{\Kla}\otimes_{D(H,L)^{\Kla}} (\Pi^{\Kla})^*\cong 
        D_{r}(H,L)^{\Kla}\otimes_{D_r(H,L)}(D_r(H,L)\otimes_{D(H,L)}  \Pi^*).\]
        
         We may further replace $H$ with $H^{p^m}$ for $m\ge 1$, 
        since $D_{r}(H,L)$ is a finitely generated $D_{s}(H^{p^m}, L)$-module, where $s=r^{p^m}$,
        see \cite[Lemma 7.4]{SchmidtAR}. If we  choose $m$ large enough then 
        the map \[R:=\gr^{\bullet}D_{s}(H^{p^m}, L)\twoheadrightarrow \gr^{\bullet} D_{s}(H^{p^m}, L)^{\Kla}=:S\]  is a surjection between polynomial rings by \cite[Proposition 5.6]{SchmidtAR}.
        Thus the assertion follows from Lemma \ref{commm} applied to the 
        associated graded module of $D_r(H,L)\otimes_{D(H,L)}  \Pi^*$ and Theorem \ref{beurk}.
        \end{proof}
           
           \begin{lem}\label{lafindim} If $M$ is a finitely generated $D_{r_n}(H, L)^{\Kla}$-module such that 
           \[j(M)=\dim_K H,\] then 
           $M$ is a finite dimensional $L$-vector space.
           \end{lem}
            
           \begin{proof} The ring $A:=\widehat{U(p^n \mathfrak{g}_{H, \OO})}$ is in $CDF(\OO)$ and 
           ${\rm Gr}(A/\varpi A)$ is a polynomial ring in $\dim_K H$ variables over $k$ (see  
           example \ref{example1}). 
            Since $D_{r_n}(H, L)^{\Kla}$ is a finite free module over $A[1/p]$ (see the discussion after proposition \ref{distrib}), we also have $j_{A[1/p]}(M)=\dim_K H$. Let $N\subset M$ be a finitely generated $A$-module such that $M=N[1/p]$ and let 
            $F(N/\varpi N)$ be a good filtration on $N/\varpi N$. 
             Proposition \ref{Bjork} shows that ${\rm Gr}(N/\varpi N)$
            is finite dimensional over $k$, thus 
             $N$ is finitely generated over $\OO$ and $\dim_L M<\infty$.
           \end{proof}

                 \section{A $p$-adic analogue of Conze's embedding}
                 
                 Let $\mathfrak{g}$ be a complex semisimple Lie algebra and let $\mathfrak{b}$ be a Borel subalgebra of 
                 $\mathfrak{g}$. If $J$ is the annihilator of an arbitrary Verma module with respect to $\mathfrak{g}$ and $\mathfrak{b}$, Conze \cite{conze} constructed an injective algebra homomorphism
                 $U(\mathfrak{g})/J\to D(\mathfrak{g}/\mathfrak{b})$, where $D(\mathfrak{g}/\mathfrak{b})$ is the ring of differential operators of the $\mathbb{C}$-vector space $\mathfrak{g}/\mathfrak{b}$ (if $x_i$ form a basis of $\mathfrak{g}/\mathfrak{b}$, then $D(\mathfrak{g}/\mathfrak{b})$ is the algebra of differential operators on $\mathbf{C}[x_1,\dots,x_n]$ generated by $x_i$ and $\frac{\partial}{\partial x_i}$). In particular 
                 $U(\mathfrak{g})/J$ is an integral domain, since it is a standard result that $D(\mathfrak{g}/\mathfrak{b})$ is one. It follows that for any character $\chi: Z(\mathfrak{g})\to \mathbb{C}$ the ring $U(\mathfrak{g})_{\chi}:=U(\mathfrak{g})\otimes_{Z(\mathfrak{g}), \chi} \mathbb{C}$ is an integral domain. 
                                  A more conceptual way of proving this is via the Beilinson--Bernstein localization theory \cite{Bebe}, which identifies $U(\mathfrak{g})_{\chi}$ with the ring of global sections of a suitable sheaf of twisted differential operators on the flag variety of $\mathfrak{g}$, and then, modulo Duflo's annihilation theorem, Conze's embedding is essentially the restriction of global sections to those on the big Bruhat cell (making this statement precise and proving it requires work, cf. \cite{SH}). 
                                  
                                  In this section we will establish a $p$-adic analogue of the above result by using the $p$-adic version of Beilinson--Bernstein localization theory due to Ardakov and Wadsley \cite{awannals} (when the prime $p$ is very good with respect to the ambient group), extended by Ardakov in \cite{ardakovast}. We are very grateful to Ardakov for suggesting the proof below, much easier than ours and which avoids any assumption on the prime $p$. Theorem \ref{domaineasy} below is one of the key ingredients in the proof of Theorem \ref{astuce}.             
                 
                 Let 
                   $\bold{G}$ be a connected, split semisimple and simply connected group scheme over 
                   $\OO$, with Lie algebra $\mathfrak{g}$. Let $\chi: Z(\mathfrak{g}_L)\to L$ be 
                   an $L$-algebra homomorphism. Let 
                   \[H_n=\ker(\bold{G}(\OO)\to \bold{G}(\OO/p^{n+1}\OO)),\]
                an open subgroup of $\bold{G}(\OO)$ and a
                 $L$-uniform pro-$p$ group if $n\geq 1$ (even $n\geq 0$ for $p>2$), with $L_{H_n}=p^{n+1}\mathfrak{g}$ (we refer the reader to the discussion in section $7$ of \cite{SK1} and the references therein). Therefore by Lazard's isomorphism (Proposition \ref{Laz}) we have a canonical isomorphism of $L$-Banach algebras 
                 \[D_{1/p}(H_n, L)^{\Lla}\simeq \widehat{U(p^n\mathfrak{g})}[1/p].\]
                 If $C$ is a $Z(\mathfrak{g}_L)$-algebra, write 
                 \[C_{\chi}=C/\ker(\chi)C=C\otimes_{Z(\mathfrak{g}_L), \chi} L.\]
                         
                         \begin{thm}\label{domaineasy}
                         With the previous notation, for all sufficiently large $n$ the ring 
                         \[D_{1/p}(H_n, L)^{\Lla}_{\chi}\simeq \widehat{U(p^n\mathfrak{g})}[1/p]_{\chi}\] is
                         an integral domain.
                         \end{thm}
                         
                         \begin{proof} Let $U=U(\mathfrak{g})$ and $A_n=\widehat{U(\varpi^n\mathfrak{g})}[1/p]_{\chi}$, where $\varpi$ is a uniformizer of $L$. It suffices to prove that 
                         $A_n$ is an integral domain for $n$ sufficiently large. 
                         
                              Fix a split maximal torus $\bold{T}$ in $\bold{G}$ and a Borel $\OO$-subgroup scheme $\bold{B}\subset \bold{G}$ containing $\bold{T}$, with unipotent radical 
       $\bold{N}$. Let $W$ be the Weyl group of $\bold{G}$ and let $\mathfrak{b}, \mathfrak{t},\mathfrak{n}$ be the corresponding 
       $\OO$-Lie algebras.
 By the (untwisted) Harish-Chandra isomorphism, the character 
       $\chi: Z(\mathfrak{g}_L)\to L$ corresponds to a $W$-orbit (for
       the dot action) of weights 
       $\lambda: \mathfrak{t}_L\to L$. We choose in the sequel an element 
       $\lambda$ of this $W$-orbit such that $\lambda+\rho$ is dominant, where 
       $\rho$ is the half-sum of positive roots. We will only consider positive integers $n$ for which $\lambda(\varpi^n \mathfrak{t})\subset \OO$. Clearly all sufficiently large $n$ have this property.       
       
        Let 
        $X:=\bold{G}/\bold{B}$ be the associated flag variety and let $U_1,...,U_m$ be the Weyl translates of the big cell in $X$. Each $U_i$ is an affine space of dimension 
       $\dim X$ and an affine open dense subscheme of $X$, and the $U_i$'s cover $X$.                    
        Ardakov and Wadsley (see \cite[Section 6.4]{awannals}) construct a ${\rm Def}(\OO)$-valued (cf. Section \ref{Fildef} for the category ${\rm Def}(\OO)$) sheaf $\mathcal{D}_n^{\lambda}$ on $X$ associated to $n$ and $\lambda$ and having the following properties:
        
        $\bullet$ there is a natural isomorphism of sheaves of graded $\OO$-algebras on $X$
        \[{\rm Sym}_{\OO_X}(\mathcal{T}_X)\simeq {\rm gr} (\mathcal{D}_n^{\lambda}),\]
        where $\mathcal{T}_X$ is the tangent sheaf of $X$;
        
        $\bullet$ for each $1\leq i\leq m$ there is an isomorphism of sheaves of filtered $\OO$-algebras 
 \[\mathcal{D}_n^{\lambda}|_{U_i}\simeq (\mathcal{D}_X)_{n}|_{U_i},\]
 where  $\mathcal{D}_X=U(\mathcal{T}_X)$ is the sheaf of crystalline differential operators\footnote{Recall that this is a sheaf of rings generated by $\OO_{X}$ and 
    $\mathcal{T}_{X}$ with the obvious relations.} on 
                    $X$, cf. \cite[def. 4.2]{awannals}.

                    $\bullet$ Let $U_n=U(\varpi^n \mathfrak{g})$ be the $n$th deformation of $\mathfrak{g}$ and define similarly
                    $(U^{\bold{G}})_n$. Let $HC: U^{\bold{G}}\to U(\mathfrak{t})$ be the untwisted\footnote{Thus $HC$ is obtained by restricting to $U^{\bold{G}}$ the linear projection $U\to U(\mathfrak{t})$ induced by the decomposition 
           $U=U(\mathfrak{t})\oplus (\mathfrak{n}U+U\mathfrak{n}^+)$.} Harish-Chandra homomorphism. By definition of $\lambda$ the map
           $\lambda\circ HC_n: (U^{\bold{G}})_n\to U(\mathfrak{t})_n\simeq U(\varpi^n \mathfrak{t})\to \OO$ is simply the restriction of 
           $\chi: Z(\mathfrak{g}_L)\to L$ to $(U^{\bold{G}})_n$ via the natural isomorphism $(U^{\bold{G}})_n[1/p]\simeq Z(\mathfrak{g}_L)$. In particular 
           $\chi( (U^{\bold{G}})_n)\subset \OO$. Ardakov and Wadsley prove the existence of a homomorphism of sheaves of $\OO$-algebras 
           \[\varphi_n^{\lambda}: U_n\otimes_{(U^{\bold{G}})_n, \chi} \OO\to \mathcal{D}_n^{\lambda},\]
           the sheaf on the left being the constant one.

  Define 
   \[\mathcal{F}_n:=(\varprojlim_{k} \mathcal{D}_n^{\lambda}/\varpi^k \mathcal{D}_n^{\lambda})\otimes_{\OO} L.\]
   This is denoted $\widehat{\mathcal{D}_{n,K}^{\lambda}}$ in \cite{awannals}. By passing to $\varphi$-adic completions and inverting $p$ (using also \cite[Lemma 6.5]{awannals}) we obtain from $\varphi_n^{\lambda}$ a morphism of $L$-algebras
           \[\widehat{U_n}[1/p]\otimes_{Z(\mathfrak{g}_L),\chi} L\to H^0(X, \mathcal{F}_n),\]
           which is an isomorphism by a 
        fundamental theorem of Ardakov--Wadsley \cite[Theorem 6.10]{awannals} and Ardakov \cite[Theorem 5.3.5]{ardakovast}.
       In other words we have  
               $A_n\simeq \Gamma(X, \mathcal{F}_n)$.
  
   With these results in hand, we can finally start doing business. 
     Let 
        $W=U_1\cap...\cap U_m$ be the intersection of the Weyl translates of the big cell in $X$, an affine dense open subset of $X$.  For each $i$ the isomorphism 
        $\mathcal{D}_n^{\lambda}|_{U_i}\simeq (\mathcal{D}_X)_n|_{U_i}$ above induces an isomorphism
        $\mathcal{F}_n|_{U_i}\simeq \widehat{(\mathcal{D}_X)_n}[1/p]|_{U_i}$, the completion being $p$-adic. In particular 
        $\mathcal{F}_n(U_i)\simeq \widehat{\mathcal{D}_X(U_i)_n}[1/p]$ (one can exchange sections over $U_i$ and 
        $p$-adic completion since we work over an affine open subscheme of $X$ and $\mathcal{D}_X$ is a quasi-coherent $\OO_X$-module) and this is an integral domain. Indeed,
        it suffices to check that $A_i:=\widehat{\mathcal{D}_X(U_i)_n}$ (which is an affinoid Weyl algebra) is an integral domain, and for this it suffices to check that 
        $A_i/\varpi A_i$ is a domain, and finally that ${\rm gr}(A_i/\varpi A_i)$ is a domain, but this is clear since the latter is naturally identified with 
        $\mathcal{O}_{T^*X_k}(q^{-1}(U_i))$, where $q: T^*X_k\to X_k$ is the cotangent bundle of $X_k$. A similar argument shows that 
        the natural map $\mathcal{F}_n(U_i)\to \mathcal{F}_n(W)$ is injective (reduce mod $\varpi$ and pass to graded pieces, then use the density of $W$ in $U_i$). It follows that all $\mathcal{F}_n(U_i)$ embed in $\mathcal{F}_n(W)$. We claim that 
        $\Gamma(X, \mathcal{F}_n)$ embeds in $\Gamma(U_i, \mathcal{F}_n)$ for any $i$, which will show that 
         $\Gamma(X, \mathcal{F}_n)$ is a domain. But if a global section $s$ on $X$ restricts to $0$ on $U_i$, then its restriction to 
         $U_j$ is $0$ (since the image of that restriction in $\Gamma(W, \mathcal{F}_n)$ is the same as the image of $s|_{U_i}$, thus it is $0$) for any $j$ and so $s=0$. This finishes the proof.
                                                                   \end{proof}
                                                                   
                                                                    \begin{remar}
        If $p$ is very good for $\bold{G}$, one can prove 
                           the above result in a slightly different way, by showing that (for $n$ large enough as in the above proof) the reduction mod 
                           $\varpi$ of the unit ball in $A_n$ has a natural filtration (induced by the PBW filtration on the 
                           enveloping algebra) whose associated graded ring is isomorphic to the algebra of regular functions on the nilpotent cone of $\mathfrak{g}^*_k$ (this is a deep theorem of Ardakov and Wadsley \cite[Theorem 6.10 c)]{awannals}, fully using results from invariant theory in characteristic $p$). This latter algebra is a domain since the nilpotent cone is irreducible, and by standard arguments we deduce that the unit ball in $A_n$ is a domain.                           
                         \end{remar}
                         
                         \begin{conj}\label{condomain}
                              Let $\mathfrak{h}$ a semi-simple $L$-Lie algebra and 
                          $\mathfrak{h}_0\subset 
                          \mathfrak{h}$ an $\OO$-Lie lattice. Let $\chi: Z(\mathfrak{h})\to L$ be an $L$-algebra homomorphism. For all sufficiently large $n$ the ring 
                      $\widehat{U(p^n\mathfrak{h}_0)}[1/p]_{\chi}$ is a domain.
                         \end{conj}
                         
                       Standard arguments reduce the proof to the case when 
                         $\mathfrak{h}_0$ is an $\OO$-Lie lattice in $\mathfrak{g}_L$, where 
                         $\bold{G}, \mathfrak{g}$, etc are as above. Choose $c>0$ such that 
                         $\varpi^c\mathfrak{h}_0\subset \mathfrak{g}$, so that 
                         $U(\varpi^{n+c}\mathfrak{h}_0)\subset U(\varpi^n \mathfrak{g})$. 
                                                  One might then be tempted to imitate the above proof and realize 
                                                            $\widehat{U(\varpi^{n+c}\mathfrak{h}_0)}[1/p]_{\chi}$ as global sections of some sub-sheaf 
                                                            $\mathcal{G}_{n}$ of the sheaf $\mathcal{F}_n$ on the flag variety $X$, which would imply that 
                                                              $\widehat{U(\varpi^{n+c}\mathfrak{h}_0)}[1/p]_{\chi}$ embeds in 
                                                              $\widehat{U(\varpi^n \mathfrak{g})}[1/p]_{\chi}$ and so it is itself a domain thanks to the Theorem above. Naturally one would like to consider the sub-sheaf $\mathcal{H}_n$ of $\mathcal{D}_n^{\lambda}$ generated by $\OO_X$ and $p^{n+c}\mathfrak{h}_0$. Quite a bit of work (consisting in adapting the proof of \cite[Theorem 5.3.5]{ardakovast} for the sheaf $\mathcal{H}_n$) shows that the global sections of $\mathcal{G}_n:=\widehat{\mathcal{H}_n}\otimes_{\OO} L$ are indeed $\widehat{U(\varpi^n\mathfrak{h}_0)}[1/p]_{\chi}$. The problem is that it is by no means clear that the embedding $\mathcal{H}_n\to \mathcal{D}_n^{\lambda}$ stays an embedding after passage to $p$-adic completion, i.e. that $\mathcal{G}_n$ is a subsheaf of $\mathcal{F}_n$.

  \section{An application of affinoid Verma modules}

             The aim of this paragraph is to establish a technical result which is the second key ingredient in the proof of Theorem \ref{astuce}. Only the first theorem below will be needed in the proof of Theorem \ref{astuce}, but we found the other statements interesting in their own right. 
             
         We fix throughout this section a finite extension $L$ of $\Qp$, with ring of integers $\OO$. If $G$  is a connected reductive group defined over a sub-extension 
             $K$ of $L$, split over $L$, with Lie algebra $\mathfrak{g}$ (a $K$-Lie algebra) then  we let $\mathfrak{g}_L:=\mathfrak{g}\otimes_K L$.
     If 
              $\chi: Z(\mathfrak{g}_L)\to L$ is an $L$-algebra homomorphism and 
             $H$ is a compact open subgroup of $G(K)$, then 
               ${\rm Lie}(H)={\rm Lie}(G(K))=\mathfrak{g}$, hence 
               $D_{1/p}(H,L)^{\Kla}$ is an $U(\mathfrak{g}_L)$-module and we can define  
               \[D_{1/p}(H,L)^{\Kla}_{\chi}=D_{1/p}(H,L)^{\Kla}\otimes_{Z(\mathfrak{g}_L),\chi} L.\]
               There is a natural map $L\br{H}:=L\otimes_{\OO} \OO\br{H}\to D_{1/p}(H,L)^{\Kla}_{\chi}$ (recall that $D_{1/p}(H,L)^{\Kla}$ is a quotient of 
               $D_{1/p}(H,L)$, which contains $D(H,L)$ and so also $L\br{H}$).
               
                The following theorem is a re-interpretation of the main result of \cite{awverma}. Compared to loc.\,cit., we drop the hypothesis that $p$ is a very good prime for the group $\bold{G}$ below, 
                and we explain how one can adapt the arguments in loc.\,cit. to avoid this hypothesis.
                          
               \begin{thm}\label{affverma} Suppose that $\bold{G}$ is a connected, split semisimple and simply connected group scheme over $\OO$ with Lie algebra $\mathfrak{g}$. Given a character 
               $\chi: Z(\mathfrak{g}_L)\to L$, for all 
                        $n$ large enough the natural map\footnote{The isomorphism being a consequence of Proposition \ref{Laz} and the discussion preceding Theorem \ref{domaineasy}.} 
                         \[L\br{H_n}\to D_{1/p}(H_n, L)^{\Lla}_{\chi}\simeq \widehat{U(p^n\mathfrak{g})}[1/p]_{\chi}\]
                         is injective, where $H_n=\ker(\bold{G}(\OO)\to \bold{G}(\OO/p^{n+1}\OO))$.
                         \end{thm}
                         
                         \begin{proof} Fix a maximal split torus $\bold{T}$ in 
                         $\bold{G}$ and a pair of opposite Borel subgroups 
                         $\bold{B}$, $\bold{B}^+$ containing $\bold{T}$. If 
                         $\lambda: \mathfrak{t}_L\to L$ is a weight (where
                         $\mathfrak{t}={\rm Lie}(\bold{T})$, $\mathfrak{b}={\rm Lie}(\bold{B})$, etc.), consider the Verma module $V^{\lambda}=U(\mathfrak{g}_L)\otimes_{U(\mathfrak{b}^+_L), \lambda} L$, where 
                         the map $U(\mathfrak{b}^+_L)\to L$ is induced by the natural projection 
                         $\mathfrak{b}^+_L\to \mathfrak{t}_L$. It is a standard result that 
                          $V^{\lambda}$ has an infinitesimal character, and it follows from the Harish-Chandra isomorphism that we can find $\lambda$ such that this infinitesimal character is precisely $\chi$. Fix such 
                          $\lambda$ and consider only $n$ for which $\lambda(p^n\mathfrak{t})\subset \OO$.
                     Consider the completion (an affinoid Verma module in the terminology of \cite{awverma})
                     \[\widehat{V^{\lambda}}=\widehat{U(p^n \mathfrak{g})}[1/p]\otimes_{\widehat{U(p^n \mathfrak{b}^+)}[1/p], \lambda} L\]
                     of $V^{\lambda}$, a $\widehat{U(p^n \mathfrak{g})}[1/p]$-module. By continuity and density, 
                     $Z(\mathfrak{g}_L)$ still acts by $\chi$ on this module. By the main result of \cite{awverma} (Theorem 5.4 in loc.\,cit., see however the discussion below for a slightly subtle point) $L\br{H_n}$ acts faithfully on
                     $\widehat{V^{\lambda}}$. It follows immediately that the map 
                     $ L\br{H_n}\to\widehat{U(p^n\mathfrak{g})}[1/p]_{\chi}$ is injective (any element in the kernel must kill $\widehat{V^{\lambda}}$ by the above remarks). 

                     In order to properly finish the proof, we still have to 
                    explain why Theorem 5.4 
                     in \cite{awverma} still holds without the assumption that $p$ is a very good prime for 
                     $\bold{G}$.
                                   The reader is advised to have a copy of that paper at hand when reading the following argument, since we will freely use the notations introduced there.  
                                   In particular $B$ and $B^+$ are $L$-uniform subgroups of $\bold{B}(\OO)$ and $\bold{B}^+(\OO)$ with associated $\OO$-Lie algebras $p^n \mathfrak{b}$ 
                                   and $p^n\mathfrak{b}^+$, and $RB$ is the Iwasawa algebra $\OO\br{B}$.

                                   Without any assumption on 
                     $p$ the proof of Theorem 5.4 shows that $\widehat{V^{\lambda}}$ is a faithful 
                     module over $RB$ (with the notations in loc.\,cit.). The only point where the fact that 
                     $p$ is very good for $\bold{G}$ is used is in citing Proposition 4.8 of loc.\,cit.                    
                    to deduce that $\widehat{V^{\lambda}}$ is faithful as $RB^+$-module.
                    The proof of that proposition seems to use their deep Theorem 4.6, a 
                     $p$-adic analogue of the famous Duflo theorem, stating that the annihilator 
                     $I$ of $\widehat{V^{\lambda}}$ is generated by $\ker(\chi)$, and that proof really uses the hypothesis on $p$. However, we can argue alternatively (and very closely to the proof of Proposition 4.8 in loc.\,cit.) as follows: since $I$ is a two-sided ideal of $\widehat{U(p^n \mathfrak{g})}[1/p]$, by Corollary 4.3 a) in loc.\,cit. $I$ is stable under the adjoint action of $\bold{G}(\OO)$. Thus if $w\in \bold{G}(\OO)$ is such that $w.\mathfrak{b}^+=\mathfrak{b}$, we have 
                     $w. B^+=B$ and 
                     \[w.{\rm Ann}_{RB^+} (\widehat{V^{\lambda}})=w.RB^+\cap w.I\subset RB\cap I={\rm Ann}_{RB}(\widehat{V^{\lambda}})=0,\]
                     thus ${\rm Ann}_{RB^+} (\widehat{V^{\lambda}})=0$ and we are done.                                                                                                               \end{proof}                          
        
        \begin{remar}                   
 It is unreasonable to expect such injectivity properties to hold without some assumptions. Indeed, it is easy to see that 
              for any $n$ the natural map \[L\br{p^n \Zp}\to D_{1/p}(p^n\Zp, L)/tD_{1/p}(p^n\Zp, L)\] is not injective, where $T=\delta_1-1$ and 
              $t=\log(1+T)$ (so that $D(\Zp,L)$ is the ring of $L$-rational
              analytic functions of the variable $T$ in the open unit disc). For instance, 
              $(1+T)^{p^{n+1}}-1\in L\br{p^n \Zp}$ is divisible by $t$ in $D_{1/p}(p^n\Zp, L)$. 
   \end{remar}
                                  
                In the following result we drop the hypothesis that the group is split and simply connected.

             \begin{prop}\label{inj} For any connected semisimple group $G$ over 
             $\Qp$, with Lie algebra $\mathfrak{g}$ and any $L$-algebra homomorphism
             $\chi: Z(\mathfrak{g}_L)\to L$ 
            there is an open uniform pro-$p$ subgroup $H$ of $G(\Qp)$ such that the natural map 
              $L\br{H}\to D_{1/p}(H, L)_{\chi}$
              is injective.
             \end{prop}
                            
               \begin{proof} We first reduce to the case when $G$ is simply connected. 
               Note that if $H$ is such a subgroup, then $L\br{H'}\to D_{1/p}(H',L)_{\chi}$ is still injective for any open uniform pro-$p$ subgroup 
               $H'$ of $H$ (indeed, it suffices to test that the composition with the natural map 
               $D_{1/p}(H',L)_{\chi}\to D_{1/p}(H, L)_{\chi}$ is injective, but this is simply the restriction of the natural map 
               $L\br{H}\to D_{1/p}(H, L)_{\chi}$ to $L\br{H'}\subset L\br{H}$). We deduce immediately from this that if 
               the result holds for the simply connected cover of $G$, then it also holds for $G$, by projecting the corresponding $H$ for the simply connected cover down to $G(\Qp)$.
               
              Next, suppose that $L'$ is a finite extension of $L$. Then $Z(\mathfrak{g}_{L'})=L'\otimes_L Z(\mathfrak{g}_L)$, thus $\chi: Z(\mathfrak{g}_L)\to L$ induces an 
           $L'$-algebra map $\chi': Z(\mathfrak{g}_{L'})\to L'$ with $\ker(\chi')=L'\otimes_L \ker(\chi)$, so that $D_{1/p}(H, L')_{\chi'}\simeq L'\otimes_L D_{1/p}(H,L)_{\chi}$ and the map $L'\br{H}\to D_{1/p}(H, L')_{\chi'}$ is the base change from $L$ to $L'$ of the map $L\br{H}\to D_{1/p}(H,L)_{\chi}$. In particular, one of these maps is injective if and only if the other one is so. We are therefore allowed to replace $L$ by a finite extension and so we may assume that $G_L:=G\otimes_{\Qp} L$ is split over $L$ and $L/\Qp$ is Galois.

           Let $\bold{G}$ be the canonical split semisimple and simply connected model of $G_L$ over $\OO$ and, as above, set 
               \[H_n=\ker(\bold{G}(\OO)\to \bold{G}(\OO/p^{n+1}\OO)).\]
               Then $H_n$ is an open $L$-uniform pro-$p$ subgroup of
               $\bold{G}(\OO)$, thus also of $\bold{G}(L)\simeq
               G(L)=G_L(L)$. By the previous theorem the natural
               map
               \[L\br{H_n}\to D_{1/p}(H_n, L)^{\Lla}_{\chi}\] is injective for $n\gg 0$ and since 
  $D_{1/p}(H_n, L)^{\Lla}_{\chi}$ is a quotient of $D_{1/p}(H_n, L)_{\chi}$, it follows that the natural map 
  $L\br{H_n}\to D_{1/p}(H_n, L)_{\chi}$ is injective for $n$ large enough.
  Thus we can find an open uniform pro-$p$ subgroup $H'$ of $G(L)=G_L(L)$ 
   such that the natural map $L\br{H'}\to D_{1/p}(H', L)_{\chi}$ is injective. 
          
               Set $H''=H'\cap G(\Qp)$. It is a compact open subgroup of $G(\Qp)$.
          Choose any open uniform pro-$p$ subgroup $H$ of $H''$. We claim that the natural map $L\br{H}\to D_{1/p}(H, L)_{\chi}$ is injective. It suffices to check that the composite $L\br{H}\to D_{1/p}(H,L)_{\chi}\to D_{1/p}(H',L)_{\chi}$ is injective, but this map is simply the composite $L\br{H}\to L\br{H'}\to D_{1/p}(H',L)_{\chi}$, and both maps are injective.                    
                                                           \end{proof}

               \section{A bound for the dimension of $Z(\mathfrak g)$-finite Banach representations}
   
     In this rather long section we prove that the existence of an infinitesimal character on the locally analytic vectors of an admissible Banach representation of a $p$-adic group has a serious impact on its dimension. 
   
   \subsection{The main result and some consequences}           
    
               Let $G$ be a locally $K$-analytic group with Lie
               algebra $\mathfrak{g}$ and let $\ZgL$ 
               be the center of the universal enveloping algebra
               $U(\mathfrak g_L)$. Here we consider $\mathfrak{g}$ as
               a $\Qp$-Lie algebra. We say 
               that a locally analytic representation 
               $\Pi$ of $G$ over $L$ is {\it $\ZgL$-finite} (resp. {\it quasi-simple}) if the 
                $Z(\mathfrak{g}_L)$-annihilator of $\Pi$ is of finite codimension in $Z(\mathfrak{g}_L)$
                (resp. a maximal ideal of $Z(\mathfrak{g}_L)$). The term {\it quasi-simple} is motivated 
                by the representation theory of real groups. 
               We say that a Banach representation $\Pi$ of $G$ is {\it $\ZgL$-finite} (resp. {\it quasi-simple})
                if 
                                          $\Pi^{\la}$ (the space of locally $\Qp$-analytic vectors in $\Pi$) is 
                                          $\ZgL$-finite (resp. quasi-simple) as locally analytic representation of $G$. The key result of this chapter is the following (see paragraph \ref{Dim} for the notation $d(\Pi)$ and $d^K(\Pi)$.) 
                                  
               \begin{thm}\label{astuce}
                Let $G$ be a connected reductive group over $K$, 
                $H$ a compact open subgroup of $G(K)$ and $\mathcal{N}$
                the nilpotent cone in the Lie algebra of $G_{\overline{K}}$.
                
                a) If $\Pi$ is a $\ZgL$-finite
                admissible locally $K$-analytic representation of 
                $H$, then \[d^K(\Pi)\leq \dim \mathcal{N}.\]
                
                b) If $\Pi$ is a $\ZgL$-finite
                 admissible 
               Banach representation of $H$, then 
                \[d^K(\Pi^{\Kla})<\dim \mathcal{N}.\]
                              \end{thm}
                              
                We recall that $\dim \mathcal{N}$ is twice the
                dimension of the flag variety of $G_{\overline{K}}$
                (with respect to a Borel subgroup). Part a) was
                already observed by Schmidt and Strauch \cite[prop
                7.3]{SS}, under some assumptions on $p$ depending on
                $G$, as a consequence of deep results of
                Ardakov--Wadsley \cite{awannals}. We give here a much
                easier proof, which also works uniformly with respect
                to the prime $p$. We note that the bound in part a) is
                optimal. On the other hand, part b) is new, slightly
                surprising and requires quite a bit more work than
                part a). 
                
                \begin{remar}
                Note that the dimension of the nilpotent cone of
                $\Res_{K/\Qp}G$ is $[K:\Qp]\dim\mathcal{N}$. Therefore
                Theorem \ref{astuce} b) implies
                \[ d(\Pi)=d(\Pi^{\la})<[K:\Qp]\dim\mathcal{N}.\]
                This inequality is far from being optimal. Based on
                the archimedean case, it is natural to ask if, for
                $\Pi$ a $\ZgL$-finite admissible 
                Banach representation of $G(K)$, we have
                \begin{gather*}
                  d(\Pi)\leq\frac{1}{2}[K:\Qp]\dim\mathcal{N}, \quad 
                  d^K(\Pi^{\Kla})\leq\frac{1}{2}\dim\mathcal{N}.
                \end{gather*}

                \end{remar}
                
                \begin{remar}
                In order to prove Theorem \ref{astuce} we are free to replace $L$ by a finite extension, in particular we may assume that $L$ is a splitting field for $G$. 
                \end{remar}
                
                \begin{cor}\label{gltwoqp}
                  Let $\Pi$ be an admissible, unitary, quasi-simple Banach representation of 
                 $G:=\GL_2(\Qp)$ over $L$, having a central character. Let $\Theta$ be an $\OO$-lattice in 
                 $\Pi$ stable under $G$. Then $\Theta\otimes_{\OO} k$ has finite length (and so $\Pi$ also has finite length).
                                  
                \end{cor}
                
                \begin{proof} By the previous Theorem $d(\Pi)\leq
                  1$. If $d(\Pi)=0$ then $\Pi$ is finite dimensional
                  over $L$ (Lemma \ref{usefullater}) and everything is clear, so assume that
                  $d(\Pi)=1$ and hence $d(\pi)=1$ where
                  $\pi=\Theta\otimes_{\OO} k$, an admissible smooth
                  representation of $G$ over $k$. Assume that $\pi$
                  has infinite length. Then we can construct a
                  strictly increasing sequence $(\pi_n)_{n\geq0}$
                  subspaces of $\pi$ which are $\GL_2(\Qp)$-stable and
                  such that $\pi_0=0$. It follows from Proposition
                  \ref{toolbox} c) that there exists $N\geq0$ such
                  that $d(\pi_{n+1}/\pi_{n})=0$ for all $n\geq N$
                  so that $\pi_{n+1}/\pi_{n}$ is finite dimensional
                  for all $n\geq N$ and $\pi_m/\pi_n$ is finite
                  dimensional for all $m\geq n\geq N$. As 
                  $\SL_2(\Qp)$ is a simple group acting trivially on irreducible smooth finite dimensional
                  representations of $\GL_2(\Qp)$, it must act trivially on $\pi_m/\pi_n$ for all
                  $m\geq n\geq1$. Let $Z_1$ be the pro-$p$-Sylow
                  subgroup of the center of $\GL_2(\Zp)$. As $\pi$ has
                  a central character, so has each subquotient
                  $\pi_m/\pi_n$ and we deduce that the group
                  $Z_1\SL_2(\Qp)$ acts trivially on $\pi_m/\pi_n$ for
                  $m\geq n\geq N$. As $\pi/\pi_N$ is an admissible smooth
                  representation of $\GL_2(\Qp)$ and $Z_1\SL_2(\Qp)$
                  is an open subgroup of $\GL_2(\Qp)$, the $k$-vector
                  space $(\pi/\pi_N)^{Z_1\SL_2(\Qp)}$ is finite
                  dimensional so that $\pi=\pi_m$ for $m$ big
                  enough. This contradicts our assumption so that
                  $\pi$ has finite length.
                \end{proof}
                
                 Let $G$ be a connected reductive group over $\Qp$, with Lie algebra 
                 $\mathfrak{g}$. Let $\mathfrak{z}$ be the center of
                 $\mathfrak{g}$ and $\mathfrak{g}_{ss}$ its derived
                 subalgebra. We have a decomposition
                 $\mathfrak{g}=\mathfrak{z}\oplus\mathfrak{g}_{ss}$
                 which induces an isomorphism of $L$-algebras
                 $Z(\mathfrak{g}_L)\simeq
                 U(\mathfrak{z}_L)\otimes_LZ(\mathfrak{g}_{ss,L})$. We say that an $L$-algebra homomorphism $\chi: Z(\mathfrak{g}_L)\to L$ is {\it algebraic} if 
                 its restriction to $Z(\mathfrak{g}_{ss,L})$ is the
                 infinitesimal character of an irreducible algebraic
                 representation of the derived subgroup of $G$. If the
                 derived group of $G$ is simply connected, this is
                 equivalent to the statement that this character is the infinitesimal character of a
                 finite dimensional simple
                 $\mathfrak{g}_{ss,L}$-module. If $\chi$ is algebraic,
                 there exists a continuous character $\psi$ of
                 $G$,                  and 
                 an algebraic irreducible representation $V$
                 of $G$ such that $\chi$ is the infinitesimal
                 character of $V\otimes\psi$.
                                  
                \begin{lem}\label{ST_Ugfin} Let $G=\mathbf G(\Qp)$, where $\mathbf G$ 
                 is an algebraic group over $\Qp$, 
                such that $\mathbf G_L $ is split  semi-simple and simply connected.  
                Let $H$ be a compact open subgroup of $G$ and let $\Pi$ be a continuous representation 
                of $H$ on a finite dimensional $L$-vector space. 
                Then the action of $H$ on $\Pi$ is locally analytic. Moreover, we have 
                an isomorphism of $D(H, L)$-modules 
                \[ \Pi^*\cong \bigoplus_{i=1}^m W_i^* \otimes_L V_i^*,\]
                where $V_i$ are irreducible algebraic representations of $\mathbf G_L$, and $W_i$ are 
                smooth irreducible representations of $H$. 
                \end{lem}
                \begin{proof} The action of $H$ on $\Pi$ is locally analytic by 
                \cite[Part II, Ch.\,V.9, Thm.\,2]{serre_lie}. The proof of 
                 \cite[Proposition 3.2]{Ugfin} carries over to our setting. 
                \end{proof} 
                
                \begin{lem}\label{remix} Let $G$ be as in the previous Lemma and let $H$ be an open 
               uniform pro-$p$  subgroup of $G$. Let $M$ be an irreducible $D_{r_n}(H,L)$-module 
               on a finite dimensional $L$-vector space, with $n\geq 1$. Then $M\cong W^*\otimes_L V^*$, where 
               $W$ is an irreducible smooth 
                representation of $H$ and $V$ an irreducible representation of $\mathbf G_L$. Moreover, 
                $H^{p^{n}}$ acts trivially on $W$.

                \end{lem} 
                \begin{proof} If we let $\Pi=M^*$ then $\Pi$ is a continuous representation of $H$ 
                on a finite dimensional $L$-vector space. Hence (using the irreducibility of $M$) $\Pi= W\otimes V$ by the previous Lemma, with 
                $W$ an irreducible smooth representation of $H$ and $V$ an irreducible representation of $\mathbf G_L$. We define a new action, denoted by $\Pi_1$,  of $H$ on $W\otimes_L V$ by letting 
                $H$ act trivially on $W$. We note that $\Pi_1$ is isomorphic to a finite direct 
                sum of copies of $V$. Then $\Pi\cong \Pi_1$ as $U(\mathfrak{g}_L)$-modules and hence
                as $\widehat{U(p^n\mathfrak{g}_{H,\OO})}[1/p]$-modules. 
                We have 
                \[\widehat{U(p^n\mathfrak{g}_{H,\OO})}[1/p]\cong
                D_{1/p}(H^{p^n},L)\subset D_{r_n}(H,L)\subset
                D_r(H,L),\]
                and thus $\Pi|_{H^{p^n}}\cong \Pi_1|_{H^{p^n}}$. This implies that 
                $\Hom_{H^{p^n}}(V,\Pi)= \Hom_{U(\mathfrak g_L)}(V, \Pi)$ and so $H^{p^n}$ acts trivially 
                on $W$. 
            \end{proof}

                \begin{cor}\label{cor:fin_length} Let $D$ be a
                  quaternion algebra over $\Qp$,
                let $G=D^*$ and let $G'$ be the derived subgroup of $G$. Let  
                  $\Pi$ be an admissible, quasi-simple Banach representation of 
                 $G$, with a central character. Let $H$ be a compact  open subgroup of $G'$.
                 
                 a) If $\Pi$ has no finite dimensional $H$-stable subquotient, then $\Pi$ has finite length as a topological $H$-representation. This is the case if the infinitesimal character of $\Pi^{\la}$ is not algebraic.
                 
                 b) Assume that $H$ is uniform pro-$p$. For any $n\geq 1$
                 the $D_{r_n}(H, L)$-module 
             \[(\Pi^{\la}_{r_n})^*:=D_{r_n}(H,L)\otimes_{L\br{H}}\Pi^*\] is of finite length. 
                \end{cor}
                
                \begin{proof} Since $Z(G)H$ is open in $G$ and the centre $Z(G)$ acts on $\Pi$ 
                by a character by assumption, $\Pi$ is an admissible Banach space representation 
                of $H$. Theorem \ref{astuce} yields $d(\Pi)\leq 1$. If $d(\Pi)=0$ then $\Pi$ is a finite 
                dimensional $L$-vector space by Lemma \ref{usefullater} and the assertion follows.  
                We assume from now on that $d(\Pi)=1$. 
                
                a) If $\Pi$ is not of finite length as a topological $H$-representation, then  dually $\Pi^*$ 
                has an infinite (strictly) decreasing sequence of 
                sub-$L\br{H}$-modules $M_0,M_1,...$. Part c) of Proposition \ref{toolbox} yields 
                $d(M_i/M_{i+1})=0$ for $i$ large enough, and so $(M_i/M_{i+1})^*$ is a finite dimensional 
                continuous
                representation of $H$ over $L$ for $i$ large enough.
                It follows from Lemma \ref{ST_Ugfin} that any finite dimensional subquotient of $\Pi$ will 
                also be a subquotient of $\Pi^{\la}$ and will have 
                an algebraic infinitesimal character. Thus if 
             the infinitesimal character of $\Pi^{\la}$ is not 
                algebraic, then $\Pi$ cannot have a finite dimensional 
                subquotient. 
                
                b) Assume that $(\Pi_{r_n}^\la)^*$ has infinite length as a $D_{r_n}(H,L)$-module. Then there
                exists a strictly decreasing sequence $(M_i)_{i\geq0}$
                of $D_{r_n}(H,L)$-submodules in $(\Pi_{r_n}^\la)^*$. As
                before, part c) of Proposition \ref{toolbox} combined with Lemma \ref{lafindim} shows
                that there exists $i_0\geq 0$ such that $M_i/M_{i+1}$
                is a finite dimensional $L$-vector space for 
                $i\geq i_0$. Let $\chi: Z(\mathfrak g'_L)\rightarrow L$ be the infinitesimal character of $\Pi^{\la}$. 
                If $\chi$ is not algebraic then  $M_i=M_{i+1}$ for
                $i\geq i_0$ as in a) and we obtain a 
                contradiction. Therefore we assume that $\chi$ is
                algebraic and let  $V$ be the unique
                irreducible algebraic representation of $G'$ with infinitesimal character $\chi$. 
               Let $W_1, \ldots, W_s$ be the set of isomorphism classes 
                of irreducible representations
                of the finite group $H/H^{p^n}$. It follows from Lemmas \ref{ST_Ugfin} and \ref{remix} that 
                for all $j\ge i \ge i_0$, 
                \[M_i/M_j\cong \bigoplus_{k=1}^s  (W_k^*\otimes_L V^*)^{m_{jk}},\]
                where $m_{jk}\ge 0$. 
               As $D_{r_n}(H,L)$ is noetherian, the
                $D_{r_n}(H,L)$-module $M_i$ is generated by a finite
                number, let say $d$, of elements. For a simple
                $D_{r_n}(H,L)$-module $N$ which is finite dimensional over
                $L$, we have, for all $j\geq i\geq i_0$,
                \begin{multline*}
                  \dim_L\Hom_{D_{r_n}(H,L)}(M_i/M_j,N)\leq\dim_L\Hom_{D_{r_n}(H,L)}(M_i,N)\\
                  \leq\dim_L\Hom_{D_{r_n}(H,L)}(D_{r_n}(H,L)^d,N)=d\dim_L
                  N.\end{multline*} As all the
                $M_i/M_j$ are isomorphic to direct sums of  simple
                $D_{r_n}(H,L)$-modules finite
                dimensional over $L$ and varying in a finite number of isomorphism
                classes, the dimension of the
                $M_i/M_j$ is bounded, which gives a contradiction.
                \end{proof}
  
                We have a similar version with locally $K$-analytic
                vectors.
                
                \begin{cor}\label{cor:finite_length_K} Let $K$ be a finite extension of $\Qp$ and
                  let $D$ be a quaternion division algebra over
                  $K$. Let $G=\GL_2(K)$ or $G=D^*$. Let $\Pi$ be an
                  admissible, quasi-simple Banach representation of
                  $G$, with a central character. Let $H$ be a compact
                  open subgroup of $G$.

                  a) We have $d^K(\Pi^{\Kla})\leq1$.

                  b) If $H$ is $K$-uniform pro-$p$ then for any $n\geq 1$ the $H$-representation
                  $\Pi^{\Kla}_{r_n}$ is topologically of finite
                  length.
                \end{cor}

                \begin{proof}
                  Part a) is a particular case of b) in Theorem
                  \ref{astuce}. Part b) follows exactly as b) in
              Corollary \ref{cor:fin_length}.
                \end{proof}

           \subsection{Proof of part a) of Theorem \ref{astuce}}\label{proof_a}  
           
           The key input is the following result
           
           \begin{prop}\label{prop:flat}
            Let $G$ be a connected reductive group over $K$, with Lie algebra $\mathfrak{g}$. Suppose that 
            $G_L$ is split, of rank $r$, and let 
            $H$ be an open $K$-uniform pro-$p$ subgroup of $G(K)$. Let 
             $J$ be an ideal of $\ZgL$ of finite codimension. 
             
             a) There are $x_1,x_2,\dots,x_r\in  J$ 
             forming a regular sequence in $Z(\mathfrak{g}_L)$.
             
             b) For any $n\geq 1$, $D_{r_n}(H,L)^{\Kla}$ is flat over $Z(\mathfrak{g}_L)$ and the elements 
             $x_1,x_2,\dots, x_r$ in a) also form a regular sequence in $D_{r_n}(H,L)^{\Kla}$, consisting of central elements.
                        
           \end{prop}
            
            \begin{proof} a)
            We claim that $Z(\mathfrak{g}_L)\simeq L[t_1,...,t_r]$ as $L$-algebras. Let $W$ be the Weyl group of $G_L$. The Harish-Chandra isomorphism identifies $Z(\mathfrak{g}_L)$ with $S(\mathfrak{t}_L)^W$, where 
            $\mathfrak{t}_L$ is the Lie algebra of a maximal split torus in $G_L$. It is a standard fact that 
            this is a polynomial algebra over the algebraic closure of $L$, 
   and we claim it is already the case over $L$. Indeed, the key input in the proof is the Chevalley--Shephard--Todd Theorem, which holds whenever the characteristic of the base field is prime to $|W|$, see \cite[Theorem 7.2.1]{benson}. Thus 
            $Z(\mathfrak{g}_L)$ is a polynomial algebra over $L$, of rank necessarily equal to 
            $r=\dim \mathfrak{t}_L$ since $S(\mathfrak{t}_L)$ is integral over $S(\mathfrak{t}_L)^W$.
           
                      Since $\ZgL/J$ is a finite dimensional $L$-vector space by assumption
             we may find non-zero polynomials $p_1(T), \ldots, p_r(T)$ in $L[T]$, such that 
           $p_1(t_1), \ldots, p_r(t_r)\in J$.   Let $x_i=p_i(t_i)$ for $1\le i \le r$. Since  
                      \[Z(\mathfrak{g}_L)/(x_1, \ldots, x_{i-1})\cong L[t_1,\ldots, t_{i-1}]/(x_1, \ldots, x_{i-1})
           \otimes_L L[t_i, \dots, t_r],\]
 and  $x_i$ is regular
           in the polynomial ring $L[t_i, \ldots, t_r]$, the sequence $x_1, \ldots, x_r$ is regular in 
           $Z(\mathfrak{g}_L)$.
                      
            b) The $x_i$ are central in $D_{r_n}(H,L)^{\Kla}$ since they are invariant under the adjoint action of $G(K)$ and so they commute with all Dirac distributions, which span a dense subspace of $D_{r_n}(H,L)^{\Kla}$. It suffices to prove that $D_{r_n}(H,L)^{\Kla}$ is flat over $Z(\mathfrak{g}_L)$, as this automatically implies that they still form a regular sequence in $D_{r_n}(H,L)^{\Kla}$. 
               We have a decomposition of left $U(\mathfrak{g}_L)$-modules
                 \[D_{r_n}(H,L)^{\Kla}=\bigoplus_{h\in H/H^{p^n}} A_n\delta_h,\]
                 where $A_n=\widehat{U(p^n \mathfrak{h})}[1/p]$ and 
                 $\mathfrak{h}=\OO_L\otimes_{\OO_K} q^{-1} L_H$ is an $\OO_L$-Lie lattice in 
                 $\mathfrak{g}_L$ (see Proposition \ref{distrib} and the discussion following it). By Kostant's Theorem $U(\mathfrak{h}[1/p])=U(\mathfrak{g}_L)$ is flat (even free) over its centre $Z(\mathfrak{g}_L)$, so it suffices to prove that $A_n$ is flat over $U(\mathfrak{h}[1/p])=U(p^n\mathfrak{h})[1/p]$.
                  This follows easily from the fact that 
                 $U(p^n\mathfrak{h})$ is noetherian (cf. \cite[Section 3.2.3]{berthelot}).       
                      \end{proof}
             \begin{remar} We note that if $J$ is the kernel of a homomorphism 
             $\chi: \ZgL\rightarrow L$ of $L$-algebras then the proof of Proposition \ref{prop:flat} shows that 
             $J$ is generated by 
             a regular sequence $x_1, \ldots, x_r$ in $\ZgL$. 
             \end{remar}

             \begin{cor}\label{HomExt} Let $\Pi$ be an admissible locally $K$-analytic representation of 
             $H$. Assume that $\Pi$ is annihilated by a regular sequence $\underline{x}=(x_1, \ldots, x_r)$
          in $\ZgL$ (where $r$ is the rank of the split group $G_L$). Then there is a canonical isomorphism
                \[{\rm Ext}^{\bullet+r}_{D(H,L)^{\Kla}}(\Pi^*, D(H,L)^{\Kla})\simeq {\rm Ext}^{\bullet}_{D(H,L)^{\Kla}/(\underline{x})} (\Pi^*, D(H,L)^{\Kla}/(\underline{x})).\] 
            In particular 
\[j_{D(H,L)^{\Kla}}(\Pi^*)=j_{D(H,L)^{\Kla}/(\underline{x})}(\Pi^*)+r\geq r\]
and so $d^K(\Pi)\leq \dim_K \mathfrak{g}-r=\dim_K \mathcal{N}$, where $\mathcal{N}$ is the nilpotent cone in $\mathfrak{g}_{\overline{K}}$.
             \end{cor}   
             
             \begin{proof} The second part is a consequence of the first. The rings          
             $A=B/(\underline{x})$ and 
             $B=D(H,L)^{\Kla}$ are 
       "two-sided" Fr\'echet-Stein algebras with presentations $A=\varprojlim_{n} A_n$ and $B=\varprojlim_n B_n$, where 
               $A_n=B_n/(\underline{x})$ and $B_n=D_{r_n}(H,L)^{\Kla}$ (cf. \cite[Proposition 3.7]{adm}). By 
               the results recalled in paragraph \ref{STdim} we obtain 
             \[{\rm Ext}^{\bullet}_{A} (\Pi^*, A)\simeq
             \varprojlim_{n} {\rm Ext}^{\bullet}_{A} (\Pi^*, A)
             \otimes_A A_n\simeq\varprojlim_n {\rm Ext}^{\bullet}_{A_n}(A_n\otimes_A \Pi^*, A_n).\]
             We have a similar isomorphism for ${\rm Ext}^{\bullet+r}_{B}(\Pi^*, B)$, thus it suffices to construct canonical (in particular compatible with $n$) isomorphisms
                 \[{\rm Ext}^{i}_{A_n}(A_n\otimes_A \Pi^*, A_n)\simeq {\rm Ext}^{i+r}_{B_n}(B_n\otimes_{B} \Pi^*, B_n).\] Letting 
                 $M_n=A_n\otimes_A \Pi^*$, we have a canonical isomorphism 
                 $M_n\simeq B_n\otimes_{B} \Pi^*$. The desired isomorphisms are then consequences of 
                 the result established in the previous Proposition and of the Rees Lemma \ref{Rees}.              \end{proof}
             
             \begin{remar}\label{optimal}
             Combining part b) of Proposition \ref{prop:flat}, Corollary \ref{HomExt} and Rees' Lemma \ref{Rees}, we deduce that for any character 
             $\chi: Z(\mathfrak{g}_L)\to L$ 
             \[\Pi:=\{f\in \mathcal{C}^{\Kla}(H, L)\, |\, X.f=\chi(X)f, \, \forall \, X\in Z(\mathfrak{g}_L)\}\]
             is an admissible locally $K$-analytic representation of $H$ such that $d^K(\Pi)=\dim_K \mathcal{N}$, i.e. the inequality 
             in Corollary \ref{HomExt} is optimal.
            
             \end{remar}
             
             \begin{remar}\label{triv}
             Corollary \ref{HomExt} implies that $d^K(\Pi)<\dim_K H$ whenever $\Pi$ is $Z(\mathfrak{g}_L)$-finite.
             \end{remar}
             
             \subsection{Proof of part b) of Theorem \ref{astuce}: reduction to the semisimple case}
             
             Let 
               $G'$ be the derived group of $G$, a connected semisimple group defined over 
               $K$, and let $Z$ be the centre of $G$. Using the central isogeny 
               $G'\times Z\to G$ and general results on the finiteness of Galois cohomology we deduce that $G'(K)Z(K)$ has finite index in $G(K)$. Moreover 
               $G'(K)\cap Z(K)$ is finite. We deduce that if 
               $H_1,H_2$ are compact open and $K$-uniform pro-$p$ subgroups of 
               $G'(K)$ and $Z(K)$ respectively, then $H_1H_2$ is a compact open subgroup of 
               $G(K)$ and it is isomorphic to $H_1\times H_2$ since elements of 
               $H_2$ commute with $G(K)$ and since $H_1\cap H_2=\{1\}$ (since the intersection is finite and contained in $H_1$, which is torsion-free).
               The next result reduces then the proof of part b) of Theorem \ref{astuce} to the 
               case when $G$ is semisimple:
               
               \begin{prop}\label{painful}
                 Let $H_1,H_2$ be $K$-uniform pro-$p$ groups, with $H_2$ commutative. Let
                 $\Pi$ be an admissible $L$-Banach representation of 
                 $H:=H_1\times H_2$ such that ${\rm Lie}(H_2)$ and $Z({\rm Lie}(H_1)_L)$ act by scalars on 
                 $\Pi^{\Kla}$. There is an open $K$-uniform subgroup $H_1'$ of $H_1$ and an admissible quasi-simple $L$-Banach representation 
                 $\Pi_1$ of
                 $H_1'$ such that $d_H^K(\Pi)\leq d_{H_1'}^K(\Pi_1)$.
               \end{prop}
                       
              \begin{proof} We write for simplicity $D(?)=D(?,L)^{\Kla}$ in the sequel, and similarly for the completions of this distribution algebra. We are allowed to replace $H_1$ and $H_2$ by open $K$-uniform pro-$p$ subgroups. 
              
                Let $(X_1,\dots,X_d)$ be an $\OO_K$-basis of $L_{H_2}$.
              By assumption $X_i$ acts by a scalar $c_i$ on $\Pi^{\Kla}$. 
               Choosing $N$ large enough such that $p^Nc_i\in p^2\OO$ for all $1\leq i\leq d$ and 
              replacing $H_2$ by $H_2^{p^N}$, we may assume that 
              $c_i\in p^2\OO$. Let $(\alpha_1,\dots,\alpha_m)$ be a $\Zp$-basis of $\OO_K$ and let, for all $1\leq i\leq d$, $1\leq j\leq m$, $g_{i,j}\in H$ be such that $X_{g_{i,j}}=\alpha_jX_i$.
              We obtain a character 
              $\delta: H=H_1\times H_2\to \OO^{\times}$ by setting $\delta(h_1, g^x)=e^{\sum_{i,j} x_{i,j}\alpha_jc_i}$ for $g^x=\prod_{i,j} g_{i,j}^{x_{i,j}}\in H_2$, $x_{i,j}\in\Zp$, and 
              $h_1\in H_1$. 
              Consider the $H$-representation $\Pi_0=\Pi\otimes \delta^{-1}$. By construction
              $X_i$ acts by $0$ on $\Pi_0^{\Kla}$ and since $\delta$ is trivial on $H_1$, 
              $Z({\rm Lie}(H_1)_L)$ still acts by scalars on $\Pi_0^{\Kla}$. Moreover 
              $d_H(\Pi)=d_H(\Pi_0)$ and $d_H^K(\Pi^{\Kla})=d_H^K(\Pi_0^{\Kla})$, so replacing $\Pi$ by $\Pi_0$ we may assume that 
              $\Pi^{\Kla}$ is smooth as representation of $H_2$.

              For $n\geq 1$ let 
                \[M_n=D_{r_n}(H)\otimes_{L\br{H}} \Pi^*=D_{r_n}(H)\otimes_{D(H)} (\Pi^{\Kla})^*.\]
By paragraphs \ref{STdim} and \ref{Dim} we know that the sequence $(j_{D_{r_n}(H)}(M_n))_{n\geq 1}$ is nonincreasing
and eventually equal to $j_{D(H)}((\Pi^{\Kla})^*)$. Fix once and for all 
$n$ such that 
\[j_{D_{r_n}(H)}(M_n)=j_{D(H)}((\Pi^{\Kla})^*)\]
and set 
\[\Pi_1=\{v\in \Pi|\, g^{p^n}.v=v,\,\, \forall g\in H_2\}=\Pi^{H_2^{p^n}}.\]
  We will show that $\Pi_1$ satisfies all required properties. It is clear that 
  $\Pi_1$ is a closed subspace of $\Pi$ stable under $H$, thus it is an admissible representation of 
  $H$. Let $\Gamma=H_2/H_2^{p^n}$, a finite group. The action of
  $H$ on $\Pi_1$ factors by $H_1\times \Gamma$, so 
  $\Pi_1$ is already admissible as $H_1$-representation. 
  It is clear that $\Pi_1$ is quasi-simple, so it suffices to show that 
  $d_{H}^K(\Pi^{\Kla})\leq d_{H_1}^K(\Pi_1^{\Kla})$, or equivalently 
  \[j_{D(H_1)}((\Pi_1^{\Kla})^*)+d\leq j_{D(H)}((\Pi^{\Kla})^*)=j_{D_{r_n}(H)}(M_n).\]
  
  Note that 
    $D_{r_n}(H_1\times\Gamma)\simeq D_{r_n}(H_1)\otimes_L L[\Gamma]\simeq D_{r_n}(H)/(\iota(X_1),\dots,\iota(X_d))$ where $\iota$ is the composite of the map (\ref{eq:iota}) with the embedding of Proposition \ref{distrib}. Moreover 
     it follows from the flatness of $D_{r_n}(H)$ over $U(\Lie(H)_L)$ (see the proof of Proposition~\ref{prop:flat}) that $(\iota(X_1),\dots,\iota(X_d))$
      is a regular sequence of central elements in $D_{r_n}(H)$. By Rees' Lemma and a Hochschild-Serre type argument combined with the vanishing of cohomology in characteristic zero for $\Gamma$ we obtain 
         \[ j_{D_{r_n}(H)}(M_n)=j_{D_{r_n}(H_1)\otimes_L L[\Gamma]}(M_n)+d=j_{D_{r_n}(H_1)}(M_n)+d.\]
It suffices therefore to prove that $j_{D_{r_n}(H_1)}(M_n)\geq j_{D(H_1)}((\Pi_1^{\Kla})^*)$. 
    The isomorphisms $D_{r_n}(H)/(\iota(X_1),\dots,\iota(X_d))\simeq D_{r_n}(H_1)\otimes_L L[\Gamma]$ and the fact that $(\Pi^{\Kla})^*$ is killed by $\iota(X_i)$ induce isomorphisms of 
    $D_{r_n}(H_1)$-modules
    \begin{align*}M_n&\simeq (D_{r_n}(H)/(\iota(X_1),\dots,\iota(X_d)))\otimes_{D(H)} (\Pi^{\Kla})^* \\
      &\simeq (D_{r_n}(H_1)\otimes_L L[\Gamma])\otimes_{L\br{H}} \Pi^* \\
      &\simeq (D_{r_n}(H_1)\otimes_L L[\Gamma])\otimes_{L\br{H_1\times\Gamma}} \Pi_1^*\\
                &\simeq D_{r_n}(H_1)\otimes_{D(H_1)} (\Pi_1^{\Kla})^*\end{align*}
    and using again the results recalled in paragraph \ref{STdim} we finally obtain 
     \[j_{D_{r_n}(H_1)}(M_n)=j_{D_{r_n}(H_1)}(D_{r_n}(H_1)\otimes_{D(H_1)} (\Pi_1^{\Kla})^*)\geq j_{D(H_1)}((\Pi_1^{\Kla})^*),\]
    as desired.                            \end{proof}

               \subsection{End of the proof of part b) of theorem \ref{astuce}}

                Replacing $L$ by a finite extension, we may assume that $G_L$ is split of rank $r$. 
                              It suffices to rule out the case 
               $j_{D(H,L)^{\Kla}}((\Pi^{\Kla})^*)=r$. Assume that this holds and let 
                $J$ be the $Z(\mathfrak g)$-annihilator of $\Pi^{\Kla}$. It is of finite 
                codimension by assumption. 
                Without loss of generality we may assume that        the subgroup  $H$ of $G(K)$ 
                in Theorem \ref{astuce} is $K$-uniform pro-$p$ subgroup  and 
                              let $\underline{x}=(x_1, \ldots, x_r)$ be a regular sequence in $\ZgL$ 
                              contained in $J$ as in Proposition \ref{prop:flat}. Applying Corollary 
                              $\ref{HomExt}$ yields 
                              \[{\rm Hom}_{D(H,L)^{\Kla}}( (\Pi^{\Kla})^*, D(H,L)^{\Kla}/(\underline{x}))\ne 0.\] 
                              Since $(\Pi^{\Kla})^*\simeq D(H,L)^{\Kla}\otimes_{L\br{H}} \Pi^*$                                (which follows from \cite[Theorem 7.1]{adm} together with \cite[Proposition 3.4]{Schmidt_vec}), we obtain 

                             \[{\rm Hom}_{L\br{H}}(\Pi^*, D(H,L)^{\Kla}/(\underline{x}))\ne 0.\]
                              Since $\ZgL/(\underline{x})$ is an artinian $L$-algebra, after possibly 
                              replacing $L$ by a finite  extension  we may assume that all the maximal ideals
                              of $\ZgL/(\underline{x})$ have residue field $L$. Since $D(H,L)^{\Kla}$ is flat over 
                              $\ZgL$ by Proposition \ref{prop:flat},  there is an $L$-algebra homomorphism 
                              $\chi: Z(\mathfrak{g}_L)\to L$, such that 
                               \[{\rm Hom}_{L\br{H}}(\Pi^*, D(H,L)_{\chi}^{\Kla})\ne 0.\]

                              As $\Pi$ is $\ZgL$-finite 
                               we have
                              $d(\Pi)<\dim H$, so that $\Pi^*$ is a
                              torsion $L\br{H}$-module (see Remark \ref{triv} and Lemma \ref{usefullater}). We obtain a
                              contradiction using the following
                              result.

               \begin{prop}\label{key}
                Let $G$ be a connected semisimple group over $K$, with Lie algebra $\mathfrak{g}$, and let $H$ be an open subgroup of $G(K)$ which is $K$-uniform pro-$p$. Let $\chi: Z(\mathfrak{g}_L)\to L$ be an $L$-algebra homomorphism.  
                For any admissible $L$-Banach representation $\Pi$ of $H$ such that 
                $\Pi^*$ is a torsion $L\br{H}$-module we have 
                \[{\rm Hom}_{L\br{H}}(\Pi^*, D(H, L)_{\chi}^{\Kla})=0.\]
                                            \end{prop}                          
     
     \begin{proof} For simplicity we assume $p>2$ (the argument is identical for $p=2$, but one needs to replace some ocurrences of $p$ by $q$ below). By replacing $L$ by a finite extension of $K$, we may assume that 
     $G$ is split over $L$. Choose a simply connected, semisimple split algebraic group scheme
     $\bold{G}$ over $\OO$ as well as an isomorphism $L\otimes_{\Qp}\mathfrak{g}\simeq L\otimes_{\OO}\mathfrak{g}_0$, where $\mathfrak{g}_0={\rm Lie}(\bold{G})$. From now on we consider this isomorphism of $L$-Lie algebras as an identification, so that $\mathfrak{g}_{H,\OO}:=\OO\otimes_{\OO_K} p^{-1}L_H$
     (an $\OO$-Lie lattice in $L\otimes_{K} \mathfrak{g}$) becomes an $\OO$-Lie lattice in 
     $L\otimes_{\OO}\mathfrak{g}_0$. Fix a positive integer $c$ such that 
     \[p^{c} \mathfrak{g}_0\subset \mathfrak{g}_{H,\OO}\subset p^{-c} \mathfrak{g}_0.\]
     For $n>0$ let 
     \[H_n=\ker(\bold{G}(\OO)\to \bold{G}(\OO/p^{n+1}\OO)),\]
     an open $L$-uniform subgroup of $\bold{G}(\OO)$ with $L_{H_n}=p^{n+1}\mathfrak{g}_0$.
     
     Denote 
     \[A_n=\widehat{U(p^n \mathfrak{g}_0)}[1/p]\simeq D_{1/p}(H_n, L)^{\Lla},\] 
     \[B_n=\widehat{U(p^n\mathfrak{g}_{H,\OO})}[1/p]\simeq D_{1/p}(H^{p^n},L)^{\Kla},\]
     the isomorphisms being induced by Lazard's isomorphism (Proposition \ref{Laz}). 
     By the equivalence of categories between powerful $\Zp$-Lie algebras and uniform pro-$p$ groups, 
  the embedding of powerful $\Zp$-Lie algebras 
     \[L_{H^{p^n}}\simeq p^n L_H\subset p^n L_H\otimes_{\OO_K} \OO\subset p^{n+1-c} \mathfrak{g}_0=L_{H_{n-c}}\] 
induces  
   a map of uniform pro-$p$ groups $H^{p^n}\to H_{n-c}$.
   By functoriality of Lazard's isomorphism, 
    the composite map \[B_n\simeq D_{1/p}(H^{p^n},L)^{\Kla}\to D_{1/p}(H_{n-c}, L)^{\Kla}\to 
     D_{1/p}(H_{n-c}, L)^{\Lla}\simeq A_{n-c}\] is induced by the composite
     \[\OO\otimes_{\OO_K} p^{-1} L_{H^{p^n}}\to \OO\otimes_{\OO_K} p^{-1} L_{H_{n-c}}\to 
     \OO\otimes_{\OO} p^{-1}L_{H_{n-c}}\simeq p^{-1}L_{H_{n-c}}=p^{n-c}\mathfrak{g}_0,\]
     which in turn is induced (by carefully unwinding definitions) by the 
          inclusion $p^n\mathfrak{g}_{H,\OO}\subset p^{n-c}\mathfrak{g}_0$, and so it is injective.
         The inclusion 
     $p^{n+c} \mathfrak{g}_0\subset p^n\mathfrak{g}_{H,\OO}$ yields a map 
     $A_{n+c}\to B_n$ and the composite $A_{n+c}\to B_n\to A_{n-c}$
 is the natural inclusion $A_{n+c}\subset A_{n-c}$. 
 Finally, we have natural maps 
 \[L\br{H^{p^n}}\to D_{1/p}(H^{p^n},L)^{\Kla}\simeq B_n, \quad L\br{H_{n-c}}\to D_{1/p}(H_{n-c}, L)^{\Lla}\simeq A_{n-c}\]  and they are compatible with the 
 maps $L\br{H^{p^n}}\to L\br{H_{n-c}}$ and $B_n\to A_{n-c}$.  In particular, since $B_n\to A_{n-c}$ and $L\br{H^{p^n}}\to B_n$ are injective, so is the map $L\br{H^{p^n}}\to L\br{H_{n-c}}$. 
 
 Next, for $m\geq n$ the maps 
 $D_{r_m}(H,L)^{\Kla}\to D_{r_n}(H,L)^{\Kla}$ and $B_{m}\to B_n$ induce maps 
 \[\pi_{m,n}: D_{r_m}(H,L)_{\chi}^{\Kla}\to D_{r_n}(H,L)_{\chi}^{\Kla}, \quad \pi_{m,n}: B_{m,\chi}\to B_{n,\chi}.\]

     If $C$ is a $Z(\mathfrak{g}_L)$-algebra, write $C_{\chi}=C/\ker(\chi)C$. 
        Let $\varphi: \Pi^*\to D(H,L)_{\chi}^{\Kla}$
     be a nonzero $L\br{H}$-morphism and let $\varphi_n: \Pi^*\to D_{r_n}(H,L)_{\chi}^{\Kla}$ be the composite of 
     $\varphi$ with the natural projection $D(H,L)_{\chi}^{\Kla}\to D_{r_n}(H,L)_{\chi}^{\Kla}$. Consider the decomposition of $B_n$-modules, in particular of $L\br{H^{p^n}}\subset B_n$-modules (cf. discussion following Proposition \ref{distrib}), where $S_n \subset H$ is a system of representatives of $H/H^{p^n}$,
     \[D_{r_n}(H,L)_{\chi}^{\Kla}\simeq \bigoplus_{h\in S_n} B_{n,\chi}\delta_h.\]
      In terms of this decomposition the maps $\pi_{m,n}$ introduced above are described by  
     \[\pi_{m,n}(\sum_{h\in S_{m}} b_h \delta_h)=\sum_{g\in S_n} (\sum_{h\in S_{m}, \bar{h}=g} \pi_{m,n}(b_h)) \delta_g,\]
     where the equality $\bar{h}=g$ means the equality of cosets $hH^{p^n}=g H^{p^n}$.
          We deduce the existence and uniqueness of $L\br{H^{p^n}}$-linear maps 
     $\varphi_{n,h}: \Pi^*\to B_{n,\chi}$
     such that for all $x\in \Pi^*$
     \[\varphi_n(x)=\sum_{h\in S_n} \varphi_{n,h}(x) \delta_h.\]
   Since the morphisms $\varphi_n$ are compatible with respect to the maps 
     \[\pi_{m,n}: D_{r_{m}}(H,L)_{\chi}^{\Kla}\to D_{r_n}(H,L)_{\chi}^{\Kla},\]
      we have the ``distribution relation'' for $m\geq n$
     \[\varphi_{n,g}=\sum_{h\in S_{m}, \bar{h}=g} \pi_{m,n}\circ \varphi_{m,h}.\]

     \begin{lem} For $n$ large enough and any $h\in S_n$ the composite 
     \[\Pi^*\to B_{n,\chi}\to A_{n-c,\chi}\]
     is $0$, where the first map is $\varphi_{n,h}$ and the second is the one induced by 
     $B_n\to A_{n-c}$.
     \end{lem}
     
     \begin{proof}  We first claim that for $n$ sufficiently large the ring $A_{n-c,\chi}$ is a domain and 
     the map $L\br{H^{p^n}}\to B_n\to A_{n-c}\to A_{n-c,\chi}$ is injective. The first part follows from Theorem 
     \ref{domaineasy}. For the second one, we observe that the map $L\br{H^{p^n}}\to B_n\to A_{n-c}\to A_{n-c,\chi}$ is the composite 
    $L\br{H^{p^n}}\to L\br{H_{n-c}}\to A_{n-c,\chi}\simeq D_{1/p}(H_{n-c}, L)^{\Lla}_{\chi}$. We have already seen that the first map is injective, and the second one is injective for $n$ large enough by Theorem 
     \ref{affverma}. We work with such $n$ in the sequel. 
     Let
     $h\in S_n$ and $l\in \Pi^*$. Since $\Pi^*$ is torsion as 
     $L\br{H^{p^{n}}}$-module, there is $\lambda\in L\br{H^{p^{n}}}$ nonzero such that $\lambda l=0$. Since $\varphi_n$ is $L\br{H^{p^n}}$-linear, we have $\lambda \varphi_n(l)=0$ and so 
     $\lambda \varphi_{n,h}(l)=0$ by the above decomposition of $B_n$-modules. Let $x=\varphi_{n,h}(l)\in B_{n,\chi}$ and let 
     $\bar{\lambda}\in B_{n,\chi}$ be the image of $\lambda$ in $B_{n,\chi}$, thus 
     $\bar{\lambda}\cdot x=0$ in $B_{n,\chi}$. Now project this relation in the domain $A_{n-c,\chi}$ via the 
     map $B_{n,\chi}\to A_{n-c,\chi}$. It follows that the image of $\bar{\lambda}$ is $0$ or the image of 
     $x$ is $0$. The first is impossible, since the map $L\br{H^{p^n}}\to A_{n-c,\chi}$ is injective. Thus the image of 
     $x=\varphi_{n,h}(l)$ in $A_{n-c,\chi}$ is $0$ and we are done since $l$ was arbitrary. 
              \end{proof}
     
     We continue the proof. By the previous Lemma for $n$ large enough the composite 
     $\Pi^*\to B_{n,\chi}\to A_{n-c,\chi}\to B_{n-2c,\chi}$ is $0$ for all $h\in S_n$, the first map being 
     $\varphi_{n,h}$. Since the composite $B_{n,\chi}\to A_{n-c,\chi}\to B_{n-2c,\chi}$ is nothing but 
     $\pi_{n,n-2c}$, we deduce that $\pi_{n, n-2c}\circ \varphi_{n,h}=0$ for all 
     $h\in S_n$. Using the distribution relation above we obtain for 
     $n$ big enough and $g\in S_{n-2c}$ 
     \[\varphi_{n-2c, g}=\sum_{h\in S_n, \bar{h}=g} \pi_{n, n-2c}\circ \varphi_{n,h}=0.\]
     Thus $\varphi_{n,g}=0$ for all $n$ big enough and all $g\in S_{n}$, which in turn yields 
     $\varphi=0$.      
     \end{proof}
 
\begin{remar}
 The proof would be much easier if we could prove that 
 $D(H,L)_{\chi}^{\Kla}$ is a domain (when $H$ is sufficiently small) or at least that $D_{1/p}(H^{p^n},L)_{\chi}^{\Kla}$ is 
 a domain for $n$ large enough. This second result would be a consequence of conjecture \ref{condomain}.
\end{remar}

\section{Scholze's functor}    

In this section let $F$ be a finite extension of $\Qp$. We continue to denote by $L$ a (sufficiently large) finite extension of $F$ with the ring of integers $\OO$, uniformiser $\varpi$ and residue field $k$. 
We fix an algebraic closure $\Fbar$ of $F$ and let $\Gal_F=\Gal(\Fbar/ F)$. Let $\Cp$ be the completion of $\Fbar$ and 
let $\Fbreve$ be the completion of the maximal unramified extension of $F$ in $\Fbar$. Let $G=\GL_n(F)$ and let $D/F$ be a central division algebra over $F$ with invariant $1/n$. 

To a smooth representation $\pi$ of $G$ on an 
$\OO$-torsion  module, Scholze associates a Weil-equivariant sheaf $\mathcal F_{\pi}$ on the \'etale site of the adic space $\PP^{n-1}_{\Fbreve}$, see \cite[Proposition 3.1]{scholze}.
 If $\pi$ is admissible 
then he shows that for any $i\ge 0$ the \'etale cohomology groups $H^i_\et(\PP^{n-1}_{\Cp}, \FF_{\pi})$ carry a continuous $D^{\times}\times \Gal_F$-action, which make them into 
 smooth admissible representations of $D^{\times}$. 
 
 Recall  that a smooth representation of $G$ or $D^{\times}$ on an $\OO$-torsion module is \textit{locally admissible} if it is equal to the union of its admissible subrepresentations.  
 We denote the respective categories by $\Mod^{\ladm}_G(\OO)$ and $\Mod^{\ladm}_{D^\times}(\OO)$.  The functors $\pi\mapsto H^i_\et(\PP^{n-1}_{\Cp}, \FF_{\pi})$ commute with direct limits and thus 
 sends locally admissible representations of $G$ to locally admissible representations of $D^{\times}$, see \cite[Section 3.1]{ludwig} for details. 
 
 Let $\dualcat_G(\OO)$ be the category anti-equivalent to $\Mod^{\ladm}_G(\OO)$
via Pontryagin duality and let $\dualcat_{D^\times}(\OO)$ be the category anti-equivalent to $\Mod^\ladm_{D^\times}(\OO)$ via the Pontryagin duality. We define a covariant homological $\delta$-functor $\{ \Shat^i\}_{i\ge 0}$ by 
\[\Shat^i: \dualcat_G(\OO)\rightarrow \dualcat_{D^\times}(\OO), \quad
 M\mapsto H^i_{\et}(\PP^{1}_{\Cp}, \FF_{M^{\vee}})^{\vee}\]
where $M^\vee=\Hom^{\cont}_{\OO}(M, L/\OO)$ denotes the Pontryagin dual of $M$. Note that 
$(M^{\vee})^{\vee}\cong M$.  We introduce these dual categories, because it is much more convenient to work with compact torsion-free $\OO$-modules than with discrete divisible $\OO$-modules.  Since $\pi\mapsto H^i_\et(\PP^{n-1}_{\Cp}, \FF_{\pi})$ commutes with direct limits, the functor $\Shat^i$ commutes with projective limits. We also note that $\Shat^i(M)$ carries a continuous $\OO$-linear action of $\Gal_F$, which commutes with the action of $D^{\times}$.
 
\begin{lem}\label{natural} Let $R$ be a complete local noetherian $\OO$-algebra with residue field
 $k$. Let $M\in \dualcat(\OO)$ with a ring homomorphism $R\rightarrow \End_{\dualcat(\OO)}(M)$. 
Let $\md$ be a compact $R$-module then for all $i\ge 0$ there is a natural map:
\[ \md\wtimes_R \Sch^i(M)\rightarrow \Sch^i(\md\wtimes_R M),\]
which is an isomorphism if $\md=\prod_{i\in I} R$ for some set $I$.
\end{lem} 
\begin{proof}For every $x\in \md$ we have morphism $\phi_x: M\rightarrow \md\wtimes_R M$, 
 $v\mapsto x\wtimes v$. Moreover, we have $\phi_{\lambda x}= \lambda \phi_x= \phi_x \lambda$
 for all $\lambda\in R$. This induces a map \[\Sch^{i}(\phi_x): \Sch^{i}(M)\rightarrow \Sch^{i}(\md\wtimes_R M)\] and hence an element in 
 \[\Hom^{\cont}_R(\md,  \Hom^{\cont}_R(\Sch^{i}(M), \Sch^{i}(\md\wtimes_R M))),\] which by adjunction is isomorphic to
 $\Hom^{\cont}_R(\md\wtimes_R \Sch^{i}(M), \Sch^{i}(\md\wtimes_R M))$. Hence, we 
 obtain a natural map $\md\wtimes_R \Sch^{i}(M)\rightarrow \Sch^{i}(\md\wtimes_R M)$.
 This map is an isomorphism if $\md=R$. Since both $\Sch^i$ and $\wtimes_R M$ 
 commute with projective limits, we deduce that the map is an isomorphism if $\md=\prod_{i\in I} R$ 
 for some set $I$. 
 \end{proof}

We will denote by $\widehat{\Tor}^R_i(\md, M)$ the $i$-th right derived functor 
of $\md \mapsto \md\wtimes_R M$. 

 \begin{lem}\label{yh} Let $R$, $M$ and $\md$ be as in Lemma \ref{natural}. Assume that 
 $\widehat{\Tor}^R_i(\md, M)=0$ for all $i\ge 1$. Suppose that  $\Sch^{i}(M)=0$ for $0\le i < q_0$. 
 Then 
there is a natural isomorphism 
\[\md\wtimes_R \Sch^{q_0}(M) \overset{\cong}{\longrightarrow} \Sch^{q_0}(\md\wtimes_R M)\]
and $\Sch^i(\md\wtimes_R M)=0$ for $0\le i <q_0$.
In particular, if $R\rightarrow R'$ is a map of complete local 
 $\OO$-algebras with residue field $k$ and either $R'$ or $M$ are $R$-flat 
 then $R'\wtimes_R \Sch^{q_0}(M)\overset{\cong}{\longrightarrow}
 \Sch^{q_0}(R'\wtimes_R M)$ is an isomorphism of $R'$-modules. 
 \end{lem} 
 \begin{proof} We choose a resolution $F_{\bullet}\twoheadrightarrow \md$ of compact $R$-modules, 
 such that $F_n= \prod_{i\in I_n} R$ for some set $I_n$ for all $n\ge 0$. Let $\md_n$ be the image 
 of $F_{n}\rightarrow F_{n-1}$ for all $n\ge 1$ and let $\md_0=\md$. Since $F_n$ are projective in the category of compact 
 $R$-modules, $\widehat{\Tor}^R_{i}(F_n, M)=0$ for $i\ge 1$ and $n\ge 0$. By considering exact sequences 
 $0\rightarrow \md_{n+1}\rightarrow F_n \rightarrow \md_n \rightarrow 0$ for $n\ge 0$ and arguing by induction on $n$ 
 we deduce that $\widehat{\Tor}^R_{i}(\md_n, M)=0$ for $i\ge 1$ and
 we have exact sequences 
 \[0\rightarrow \md_{n+1} \wtimes M \rightarrow F_{n}\wtimes_R M \rightarrow \md_{n} \wtimes_R M\rightarrow 0,\]
 for all $n\ge 0$. 
The assumption $\Sch^i(M)=0$ for $0\le i< q_0$ and Lemma \ref{natural} imply that 
$\Sch^i(  F_{n}\wtimes_R M)$ vanishes for $1\le i <q_0$ and all $n\ge 0$. We obtain 
\[ \Sch^i(\md_{n}\wtimes_R M)\cong \Sch^{i-1}(\md_{n+1}\wtimes_R M)\cong \ldots \cong \Sch^0(\md_{n+i}\wtimes_R M)=0\]
by devissage, and right exactness of $\Sch^0$. Hence,  we also obtain exact sequences 
\[\Sch^{q_0}(\md_{n+1}\wtimes_R M)\rightarrow \Sch^{q_0}(F_{n}\wtimes_R M)\rightarrow \Sch^{q_0}(\md_n\wtimes_R
M)\rightarrow 0.\]
We deduce that the sequence 
\[  \Sch^{q_0}(F_1\wtimes_R M)\rightarrow \Sch^{q_0}(F_0\wtimes_R M)\rightarrow \Sch^{q_0}(\md\wtimes_R M)\rightarrow 0\]
is exact. Using Lemma \ref{natural} we deduce that the map 
$\md \wtimes_R \Sch^{q_0}(M)\rightarrow  \Sch^{q_0}(\md\wtimes_R M)$ is an isomorphism. The 
last part follows from \cite[Lemma 3.8]{ludwig}. 
 \end{proof} 
 
 Let $\pi$ be a smooth representation of $G$ on an $\OO$-torsion module. We define two actions of $F^{\times}$ on  $H^i_\et(\PP^{n-1}_{\Cp}, \FF_{\pi})$ as follows. 
 We may identify $F^{\times}$ with the centre of $G$ and every $z\in F^{\times}$ defines $\varphi_z\in \End_G(\pi)$, $\varphi_z(v):= \pi(z) . v$. Since 
 $\pi \mapsto \FF_{\pi}$ and $H^i_{\et}$ are functors this induces an action of $F^{\times}$ on $H^i_\et(\PP^{n-1}_{\Cp}, \FF_{\pi})$. The second action 
 is obtained by identifying $F^{\times}$ with the centre of $D^{\times}$ and restricting the action of $D^{\times}$ on $H^i_\et(\PP^{n-1}_{\Cp}, \FF_{\pi})$ to $F^{\times}$. 
 
 \begin{lem}\label{central} The two actions of $F^{\times}$ on $H^i_\et(\PP^{n-1}_{\Cp}, \FF_{\pi})$ defined above coincide. In particular, if the center of $G$ acts on 
 $\pi$ by a character then the centre of $D^{\times}$ acts on $H^i_\et(\PP^{n-1}_{\Cp}, \FF_{\pi})$ by the same character. 
 \end{lem}
 
 \begin{proof}
    Let $z.s$ denote the action of $z\in F^{\times}=Z(D^{\times})$ on a local section of $\FF_{\pi}$, and let 
    $z*s$ denote the action of $z\in F^{\times}=Z(G)$ induced by functoriality of the construction 
    $\pi\mapsto \FF_{\pi}$ and the endomorphism $v\mapsto \pi(z).v$ of $\pi$. It suffices to prove that these two actions are the same. If $U\to \PP^{n-1}_{\Cp}$ is \'etale, with pullback 
    $U_{\infty}\to \mathcal{LT}_{\infty}$ (where $\mathcal{LT}_{\infty}$ is the infinite level Lubin--Tate space), then by definition 
    $\FF_{\pi}(U)=\mathcal{C}_G(|U_{\infty}|, \pi)$, the space of continuous $G$-equivariant maps from the topological space attached to 
    $U_{\infty}$ to $\pi$. If $s: |U_{\infty}|\to \pi$ is such a map, then $(z.s)(x)=s(z^{-1}.x)$ for $z\in F^{\times}\subset Z(D^{\times})$.    On the other hand, it follows from the modular description of $\mathcal{LT}_{\infty}$ that $\{(x,x)|\, x\in F^{\times}\}\subset Z(G)\times Z(D^{\times})$ acts trivially on $\mathcal{LT}_{\infty}$, thus $s(z^{-1}.x)=s(z*x)$, where 
    $*$ denotes the action of the center of $G$ on $\mathcal{LT}_{\infty}$. But $s$ is $G$-equivariant, so $s(z*x)=z*s(x)$, hence 
    $z.s=z*s$, as desired.    
 \end{proof}
 
\begin{lem}\label{olivertwist}
For any continuous character $\delta: F^{\times}\to \OO^{\times}$, any $i\geq 0$ and any smooth representation $\pi$ of $G$ on an 
$\OO$-torsion module there is a natural isomorphism of 
$\Gal_F\times D^{\times}$-representations
\[H^i_\et(\PP^{n-1}_{\Cp}, \FF_{\pi\otimes \delta})\simeq H^i_\et(\PP^{n-1}_{\Cp}, \FF_{\pi})\otimes (\delta^{-1}\boxtimes \delta).\]
\end{lem}

\begin{proof} We first construct a natural map $\FF_{\pi}\otimes_{\OO} \FF_{\pi'}\to \FF_{\pi\otimes_{\OO}\pi'}$ for any 
smooth representations $\pi,\pi'$ on $\OO$-torsion modules. If $U\to\PP^{n-1}_{\Cp}$ is an \'etale map with pullback 
$U_{\infty}\to \mathcal{LT}_{\infty}$ to the infinite-level Lubin--Tate space, we need to construct a natural map 
\[\mathcal{C}_G(|U_{\infty}|, \pi)\otimes_{\OO} \mathcal{C}_G(|U_{\infty}|, \pi')\to \mathcal{C}_G(|U_{\infty}|, \pi\otimes_{\OO} \pi').\] We simply map $f\otimes f'$ to the map $x\mapsto f(x)\otimes f'(x)$, where $f\in \mathcal{C}_G(|U_{\infty}|, \pi), f'\in \mathcal{C}_G(|U_{\infty}|, \pi')$. 

 The above map combined with cup-product induces a map \[H^i_\et(\PP^{n-1}_{\Cp}, \FF_{\pi})\otimes_{\OO} H^j_\et(\PP^{n-1}_{\Cp}, \FF_{\pi'})\to H^{i+j}_\et(\PP^{n-1}_{\Cp}, \FF_{\pi\otimes_{\OO} \pi'}).\] Suppose now that $\pi'=\delta$ is a continuous character, and take $j=0$. We obtain a $D^{\times}\times \Gal_F$-equivariant map \[\iota_{\delta}: 
H^i_\et(\PP^{n-1}_{\Cp}, \FF_{\pi})\otimes_{\OO} H^0_\et(\PP^{n-1}_{\Cp}, \FF_{\delta})\to H^i_\et(\PP^{n-1}_{\Cp}, \FF_{\pi\otimes \delta}).\] On the other hand, \[H^0_\et(\PP^{n-1}_{\Cp}, \FF_{\delta})=H^0(\mathcal{LT}_{\infty}, \underline{\delta})^G={\rm Hom}_G(\delta^{-1}, H^0(\mathcal{LT}_{\infty}, \Z_p))\] is isomorphic to $\delta^{-1}\boxtimes \delta$ as $\Gal_F\times D^{\times}$-representations, as follows from Strauch's computation of $\pi_0(\mathcal{LT}_{\infty})$, 
\cite[Theorem 4.4]{strauch}. We obtain therefore a $\Gal_F\times D^{\times}$-equivariant map 
$j_{\delta}: H^i_\et(\PP^{n-1}_{\Cp}, \FF_{\pi})\otimes (\delta^{-1}\boxtimes \delta)\to H^i_\et(\PP^{n-1}_{\Cp}, \FF_{\pi\otimes \delta})$. It is an isomorphism since an inverse can be constructed by using the same construction for $\delta^{-1}$.  
 \end{proof}
  
 We extend the functors $\Sch^i$ to the category of admissible unitary $L$-Banach space representations of $G$, which we denote by $\Ban^{\adm}_{G}(L)$. If $\Pi\in \Ban^{\adm}_{G}(L)$ 
 then we choose an open bounded $G$-invariant $\OO$-lattice $\Theta$ in $\Pi$. Its Schickhof 
 dual $\Theta^d:=\Hom^{\cont}_{\OO}(\Theta, \OO)$ is in $\dualcat_G(\OO)$ and hence
 $\Sch^i(\Theta^d)$ is in $\dualcat_{D^{\times}}(\OO)$. We let $\Sch^i(\Pi):= \Sch^i(\Theta^d)^d \otimes_{\OO}L$. Then $\Sch^i(\Pi)\in \Ban^{\adm}_{D^{\times}}(L)$ and $\{\Sch^i\}_{i\ge 0}$ define
 a cohomological $\delta$-functor from $\Ban^{\adm}_{G}(L)$ to $\Ban^{\adm}_{D^{\times}}(L)$, 
 see \cite[Section 3.3]{ludwig} for more details. 

\subsection{Ultrafilters}\label{sec_ultra}
Let $(A, \mm_A)$, $(B, \mm_B)$ be  local artinian $\OO$-algebras with residue field $k$. Let $I$ be an indexing set, and let $A_I:=\prod_{i\in I} A$. According 
to \cite[Lemma 2.2.2]{gee_newton}, there is  a bijection between the set of ultrafilters  on $I$ and prime ideals of $A_I$
given by taking an ultrafilter $\mathfrak F$ to the ideal whose elements $(a_i)$ satisfy $\{i : a_i \in \mm_A\}\in \mathfrak F$. If 
$x\in \Spec A_I$ then we let $A_{I,x}$ be the localization of $A_I$ at $x$ and $\mathfrak F_x$ be the corresponding ultrafilter. 
It follows from \cite[Lemma 2.2.2]{gee_newton} that the diagonal map $A\rightarrow A_I$ induces an isomorphism $A\cong A_{I, x}$. If $M$ is a finite $A$-module we let $M_I=\prod_{i\in I} M$ then by considering a presentation of $M$ over $A$ we also obtain
that the diagonal map $M\rightarrow M_I$ induces an isomorphism $M\cong M_{I, x}$. 

Let $M$ and $N$ be finite $A$-modules and let $N^{\vee}:=\Hom_{\OO}(N, L/\OO)$ be the Pontryagin dual of $N$. Then 
$\Hom_A(M, N^{\vee})\cong \Hom_{\OO}( M\otimes_A N, L/\OO)=(M\otimes_A N)^{\vee}$. Applying this observation with $N=A$ we get that 
\begin{equation}\label{vee}
\Hom_A(M, N^{\vee})\cong (M\otimes_A N)^{\vee}\cong (M\otimes_A N \otimes_A A)^{\vee}\cong \Hom_A(M\otimes_A N, A^{\vee})
\end{equation}

\begin{lem}\label{ultra1} Let $\{M_i\}_{i\in I}$ be a family of $A$-modules such that $|M_i| \le t$ for some $t\ge 1$ and all $i\in I$. Then there is a canonical isomorphism of 
$A_{I, x}$-modules \[ (\prod_{i\in I} M_i^{\vee})\otimes_{A_I} A_{I, x}\cong((\prod_{i\in I} M_i)\otimes_{A_I} A_{I, x})^{\vee}.\]
\end{lem}
\begin{proof} Let $c_i: M_i^{\vee}\otimes_A M_i\rightarrow A^{\vee}$ be the image of the identity map 
under   \eqref{vee} applied with $M=M_i^{\vee}$ and $N=M_i$. This yields a canonical map of $A_I$-modules
\[ (\prod_{i\in I} M_i^{\vee})\otimes_{A_I} (\prod_{i\in I} M_i)\overset{\prod c_i}{\longrightarrow} \prod_{i\in I} A^{\vee}.\]
By localizing at $x$, using the fact explained above that $( \prod_{i\in I} A^{\vee})\otimes_{A_I} A_{I, x}\cong A^{\vee}\cong A_{I, x}^{\vee}$ and \eqref{vee} we obtain a canonical 
map of $A_{I, x}$-modules 
\begin{equation}\label{vee2}
(\prod_{i\in I} M_i^{\vee})\otimes_{A_I} A_{I, x} \rightarrow ((\prod_{i\in I} M_i)\otimes_{A_I} A_{I, x})^{\vee}
\end{equation}
Moreover, this map is an isomorphism if all $M_i$ are  equal. In particular, it is an isomorphism if $M_i=A^m$ for a fixed $m$ and all $i\in I$. 
Let $m$ be the cardinality of $A$. Then for each $i\in I$ we may choose a presentation $A^{tm}\rightarrow A^t\rightarrow M_i\rightarrow 0$. 
Using these presentations we deduce that the map \eqref{vee2} is an isomorphism.  
\end{proof} 

\begin{lem}\label{useful} Let $R, S$ be two rings, $E$ a left $R$-module, $F$ a left $S$-module and $G$ an $(R, S)$-bimodule.
If $E$ is finitely presented over $R$ and $F$ is $S$-flat then the canonical map
\begin{equation}\label{brbki}
\Hom_R(E, G)\otimes_S F \rightarrow \Hom_R(E, G\otimes_S F)
\end{equation}
is an isomorphism.
\end{lem}
\begin{proof} We apply the functor $\Hom_R(\cdot, G)\otimes_S F$ to a resolution $R^n \rightarrow R^m \rightarrow E\rightarrow 0$. Since $F$ is $S$-flat the sequence
\[0\rightarrow \Hom_R(E, G)\otimes_S F\rightarrow \Hom_R(R^m, G)\otimes_S F\rightarrow \Hom_R(R^n, G)\otimes_S F\]
is exact. If $E$ is a free $R$-module of finite rank then \eqref{brbki} is an isomorphism. The assertion follows 
from a diagram chase.  
\end{proof}  

Let $K$ be a compact $p$-adic analytic group. This assumption implies that the completed group 
algebra $A\br{K}$ is noetherian. Let $\{M_i\}_{i\in I}$ be a family of compact 
modules over $A\br{K}$, such that for each open subgroup $H$ of $K$ there 
is a constant $t(H)$, such that for all $i\in I$ the cardinality of $H$-coinvariants $(M_i)_{H}$ is bounded 
above by $t(H)$. Let $\pi_i:= M_i^{\vee}$ and let $\Pi$ be the subset of smooth vectors in $\prod_{i\in I} \pi_i$, so that 
\begin{equation}\label{define_Pi}
\Pi= \bigcup_H \prod_{i\in I} \pi_i^H,
\end{equation}
where the union is taken over all open (normal) subgroups $H$ of $K$. Then $\Pi$ is an $A_I$-module with a smooth action of $K$. 

\begin{lem}\label{invariants} If $H$ is an open subgroup of $K$ then 
\[(\Pi\otimes_{A_I} A_{I,x})^H\cong \Pi^H \otimes_{A_I} A_{I,x}= (\prod_{i\in I} \pi_i^H)\otimes_{A_I} A_{I,x}.\]
\end{lem}
\begin{proof} Since $H$ is also $p$-adic analytic, $A\br{H}$ is noetherian. Thus $A$ with the trivial action of 
$H$ is a finitely presented $A\br{H}$-module. Since localizations are flat Lemma \ref{useful} implies that 
\[\Hom_{A\br{H}}(A, \Pi\otimes_{A_I} A_{I,x})\cong \Hom_{A\br{H}}(A, \Pi)\otimes_{A_I} A_{I,x},\]
where we consider $\Pi$ as an $A$-module via diagonal map $A\rightarrow A_I$. This implies the assertion. 
\end{proof}

\begin{lem}\label{ultra2} There is a canonical isomorphism of $A_{I,x}\br{K}$-modules 
\[ (\Pi\otimes_{A_I} A_{I,x})^{\vee} \cong \varprojlim_H \bigl((\prod_{i\in I} (M_i)_H)\otimes_{A_I} A_{I,x}\bigr),\]
where the projective limit is taken over all open normal subgroups $H$ of $K$. Moreover, 
\[ ((\Pi\otimes_{A_I} A_{I,x})^{\vee})_H \cong (\prod_{i\in I} (M_i)_H)\otimes_{A_I} A_{I,x}.\]
\end{lem}
\begin{proof} Since tensor products commute with direct limits we have
\[\Pi\otimes_{A_I} A_{I,x}\cong \varinjlim_{H} \bigl( (\prod_{i\in I} \pi_i^H)\otimes_{A_I} A_{I,x}\bigr).\]
The isomorphism $M_i\overset{\cong}{\longrightarrow} (M_i^{\vee})^{\vee}$ induces  an isomorphism 
$(\pi_i^H)^{\vee}\cong (M_i)_H$. Since the functor $(\cdot)^{\vee}$ converts direct limits to projective limits the first assertion 
follows from Lemma \ref{ultra1}. The second assertion follows from Lemma \ref{invariants}. 
\end{proof}

 Let $\varphi: R\rightarrow A$ be a map of local $\OO$-algebras and let $N$ be a finitely generated  $\OO\br{K}$-module. 
 We equip $N$ with the canonical topology for which the action of $\OO\br{K}$ is continuous. 
 Then $R$ acts on $\Hom^{\cont}_{\OO\br{K}}(M_i, N)$ via its action on $M_i$ via $\varphi$.
 
 \begin{lem}\label{pst_quotient} If the action of $R$ on $\Hom^{\cont}_{\OO\br{K}}(M_i, N)$ for all $i\in I$ 
 factors through the quotient $R\twoheadrightarrow R'$ then the action of $R$ on 
 \[\Hom^{\cont}_{\OO\br{K}}( (\Pi\otimes_{A_I} A_{I,x})^{\vee}, N)\] 
 also factors through $R\twoheadrightarrow R'$.
 \end{lem}
 \begin{proof} The isomorphism $N\cong \varprojlim_J N/J N$, where $J$ runs over open two-sided ideals of $\OO\br{K}$,  induces an isomorphism
 \[\Hom^{\cont}_{\OO\br{K}}( (\Pi\otimes_{A_I} A_{I,x})^{\vee}, N)\cong \varprojlim_J \Hom^{\cont}_{\OO\br{K}}( (\Pi\otimes_{A_I} A_{I,x})^{\vee}, N/JN).\]
 Thus it is enough to prove the assertion when $N$ is finite (as a set), which we now assume. Then $N^{\vee}$ is also finite, thus finitely generated over 
 $A\br{K}$, and since $A\br{K}$ is noetherian, $N^{\vee}$ is finitely presented over $A\br{K}$.   Pontryagin duality and Lemma \ref{useful} give us
 \begin{equation}
 \begin{split}
  \Hom^{\cont}_{\OO\br{K}}( (\Pi\otimes_{A_I} A_{I,x})^{\vee}, N)&\cong \Hom_{\OO\br{K}}(N^{\vee}, \Pi\otimes_{A_I} A_{I,x})\\ &\cong 
 \Hom_{\OO\br{K}}(N^{\vee}, \Pi)\otimes_{A_I} A_{I,x}.
 \end{split}
 \end{equation}
 
 Let $H$ be an open normal subgroup of $K$ which acts trivially on $N^{\vee}$. We have 
  \begin{equation}
  \begin{split}
  \Hom_{\OO\br{K}}(N^{\vee}, \Pi)&\cong \Hom_{\OO\br{K}}(N^{\vee}, \Pi^H)\cong \prod_{i\in I} \Hom_{\OO\br{K}}(N^{\vee}, \pi_i^H)\\
& \cong \prod_{i\in I}\Hom^{\cont}_{\OO\br{K}}((M_i)_H, N)\cong \prod_{i\in I}\Hom^{\cont}_{\OO\br{K}}(M_i, N).
 \end{split}
 \end{equation}
 We conclude that for finite $N$ we have an isomorphism of $R$-modules 
 \[\Hom^{\cont}_{\OO\br{K}}( (\Pi\otimes_{A_I} A_{I,x})^{\vee}, N)\cong \bigl (\prod_{i\in I}\Hom^{\cont}_{\OO\br{K}}(M_i, N)\bigr)\otimes_{A_I} A_{I,x}.\]
 Since the action of $R$ on the right-hand-side factors through $R\twoheadrightarrow R'$ we obtain the assertion. 
  \end{proof}
  
   \subsection{Scholze's functor and ultrafilters}

   \begin{lem}\label{smooth_exact} Let $0\rightarrow \sigma\rightarrow \pi \rightarrow \tau\rightarrow 0$ be an 
   exact sequence in $\Mod^{\adm}_{G'}(\ZZ/p^s)$, where $G'$ is an open subgroup of $G$. Then the sequence: 
   \[0\rightarrow (\prod_{i\in I} \sigma)^{\sm}\rightarrow (\prod_{i\in I} \pi)^{\sm}\rightarrow 
   (\prod_{i\in I} \tau)^{\sm}\rightarrow 0\]
   is exact. 
   \end{lem} 
   \begin{proof} We have 
   \begin{equation}
   \begin{split}
   (\prod_{i\in I} \pi)^{\sm}&\cong \varinjlim_H \prod_{i\in I} \pi^H \cong
   \varinjlim_H ( \pi^H \otimes_{\ZZ/p^s} (\ZZ/p^s)_I)\\&\cong 
   ( \varinjlim_H \pi^H)\otimes_{\ZZ/p^s} (\ZZ/p^s)_I \cong \pi \otimes_{\ZZ/p^s} (\ZZ/p^s)_I,
   \end{split}
   \end{equation}
   where the limit is taken over open subgroups of $G'$. Since $\ZZ/p^s$ is noetherian, it is coherent, and hence $(\ZZ/p^s)_I$ is a flat 
   $\ZZ/p^s$-module, which implies the assertion.
   \end{proof}
   
  \begin{lem}\label{resolution}Let $\{\pi_i\}_{i\in I}$ be a family in $\Mod^{\adm}_G(\ZZ/p^s)$ and 
  let $\Pi= (\prod_{i\in I} \pi_i)^{\sm}$. If for some compact open subgroup $H$ of $G$ there
  is $\pi\in \Mod^{\adm}_H(\ZZ/p^s)$ and integers $m_i$ for $i\in I$, bounded independently of $i$, such that
  $\pi_i|_H \cong \pi^{m_i}$  
  then there is a resolution of $(\ZZ/p^s)_I[H]$-modules: 
  \[0\to \Pi|_H\to \Pi_1\to\ldots\to \Pi_j\to \ldots\]
   where each $\Pi_j$ is a direct summand of $\mathcal{C}(H, \ZZ/p^s)^{n_j} \otimes_{\ZZ/p^s} (\ZZ/p^s)_I$ for some $n_j$.
   \end{lem}
   \begin{proof} Since $\pi$ is admissible $\pi^{\vee}$ is a finitely generated $\ZZ/p^s\br{H}$-module. 
   Since $H$ is $p$-adic analytic $\ZZ/p^s \br{H}$ is noetherian and thus $\pi^{\vee}$ has a projective 
   resolution by free $\ZZ/p^s \br{H}$-modules of finite rank. Since $(\ZZ/p^s\br{H})^{\vee}\cong 
  \mathcal{C}(H, \ZZ/p^s)$ by taking Pontryagin dual we obtain a resolution of $\ZZ/p^s\br{H}$-modules:
  \[0\rightarrow \pi \rightarrow \mathcal{C}(H, \ZZ/p^s)^{k_0}\rightarrow \ldots \rightarrow \mathcal{C}(H, \ZZ/p^s)^{k_j}\rightarrow \ldots\]
  This yields an exact sequence of $\ZZ/p^s\br{H}$-modules: 
  \[0\rightarrow \prod_{i\in I} \pi \rightarrow (\prod_{i\in I} \mathcal{C}(H, \ZZ/p^s))^{k_0}\rightarrow \ldots\rightarrow
  (\prod_{i\in I} \mathcal{C}(H, \ZZ/p^s))^{k_j}\rightarrow \ldots\]
  It follows from Lemma \ref{smooth_exact} that the sequence remains exact after applying 
  the functor of smooth vectors. Since $(\prod_{i\in I} \mathcal{C}(H, \ZZ/p^s))^{\sm}\cong \mathcal{C}(H, \ZZ/p^s)\otimes_{\ZZ/p^s} (\ZZ/p^s)_I$, we obtain a resolution  of $\ZZ/p^s\br{H}$-modules 
  $0\rightarrow (\prod_{i\in I} \pi)^{\sm} \rightarrow F^{\bullet}$, where $F^j=  \mathcal{C}(H, \ZZ/p^s)^{k_j}\otimes_{\ZZ/p^s} (\ZZ/p^s)_I$ for all $j\ge 0$. 
  
  Let us suppose that $m_i\le m$ for all $i\in I$. For each $0\le k\le m$ 
 let $e_k=(e_{ik})_{i\in I}\in (\ZZ/p^s)_I$ be the idempotent with 
  $e_{ik}=1$ if $m_i=k$ and $e_{ik}=0$ otherwise then 
  \[\Pi|_H\cong \bigoplus_{k=0}^m e_k ((\prod_{i\in I} \pi)^{\sm})^k\] and thus 
  $\bigoplus_{k=0}^m e_k (F^{\bullet})^k$ is the required resolution. 
  \end{proof}
  
\begin{thm}\label{main_ultra}Let $\{\pi_i\}_{i\in I}$ be a family in $\Mod^{\adm}_G(\ZZ/p^s)$
and let $\Pi=(\prod_{i\in I} \pi_i)^{\sm}$. Assume that for some compact open subgroup $H$ of $G$ there
  is $\pi\in \Mod^{\adm}_H(\ZZ/p^s)$ and integers $m_i$ for $i\in I$, bounded independently of $i$, such that
  $\pi_i|_H \cong \pi^{m_i}$ for all $i\in I$. 
Then for all $j\ge 0$ and all compact open subgroups $K\subset D^{\times}$ we have 
\[H^j((\PP^{n-1}_{\Cp}/K)_{\et}, \FF_{\Pi})\cong \prod_{i\in I} H^j((\PP^{n-1}_{\Cp}/K)_{\et}, \FF_{\pi_i}),\]
where the  cohomology groups  are defined in \cite[Section 2]{scholze}.
\end{thm}
 
 We first list some consequences of the Theorem. 
  
 \begin{cor}\label{cor1} Assume, additionally to the assumptions of Theorem \ref{main_ultra},  that each $\pi_i$ has an action of a local artinian $\OO$-algebra $(A, \mm_A)$ 
 with finite residue field, which commutes with the action of $G$, then for all $x\in \Spec A_I$  and all $j\ge 0$ we have
 an isomorphism 
 \[ H^j((\PP^{n-1}_{\Cp}/K)_{\et}, \FF_{\Pi\otimes_{A_I} A_{I,x}})\cong 
 \bigl (\prod_{i\in I} H^j((\PP^{n-1}_{\Cp}/K)_{\et}, \FF_{\pi_i})\bigr)\otimes_{A_I} A_{I,x}.\]
 \end{cor}
 \begin{proof} As in the proof of \cite[Corollary 8.5]{scholze} this follows from Theorem 
 \ref{main_ultra} as the topos $(\PP^{n-1}_{\Cp}/K)_{\et}$ is coherent and so cohomology commutes 
 with direct limits. 
 \end{proof}
 
 \begin{cor}\label{cor2} Let us assume the setup of Theorem \ref{main_ultra}. Let $\Pi'_j$ be the subset of smooth vectors 
 in $\prod_{i\in I} H^j_{\et}( \PP^{n-1}_{\Cp}, \FF_{\pi_i})$ for the action of $D^{\times}$. Then 
 \[ H^j_{\et}(\PP^{n-1}_{\Cp}, \FF_{\Pi})\cong \Pi'_j.\]
 Moreover, if we have an action of ring $A$ as in Corollary \ref{cor1} then for all $x\in \Spec A_I$ 
 and all $j\ge 0$ we have  
 \[ H^j_{\et}(\PP^{n-1}_{\Cp}, \FF_{\Pi\otimes_{A_I} A_{I,x}})\cong \Pi'_j\otimes_{A_I} A_{I,x}.\]
 \end{cor}
 \begin{proof} We have 
 \[ H^j_{\et}(\PP^{n-1}_{\Cp}, \FF_{\Pi})=\varinjlim_K H^j((\PP^{n-1}_{\Cp}/K)_{\et}, \FF_{\Pi}),\]
 \[ H^j_{\et}(\PP^{n-1}_{\Cp}, \FF_{\pi_i})=\varinjlim_K H^j((\PP^{n-1}_{\Cp}/K)_{\et}, \FF_{\pi_i}),\]
 where the limit is taken over open normal subrgoups $K$ of $D^{\times}$ by \cite[Proposition 2.8]{scholze}. Since $K$ acts trivially on $H^j((\PP^{n-1}_{\Cp}/K)_{\et}, \FF_{\pi_i})$ we obtain 
 \[\Pi_j'= \varinjlim_K \bigl(\prod_{i\in I} H^j((\PP^{n-1}_{\Cp}/K)_{\et}, \FF_{\pi_i})\bigr)\]
 and the first assertion follows from Theorem \ref{main_ultra}. The second assertion 
 is proved in the same way as Corollary \ref{cor1}. 
 \end{proof}
 
 Let us start proving Theorem \ref{main_ultra}.
 
 \begin{lem}\label{induction_step} It is enough to prove Theorem \ref{main_ultra} for $s=1$ (i.e.~under assumption that 
 $\pi_i\in \Mod^{\adm}_G(\Fp)$, for all $i\in I$).
 \end{lem}
 \begin{proof} The proof is by induction on $s$, using the assumption as the start 
 of induction. For each $i\in I$ let $x_i\in \Spec ((\ZZ/p^s)_I)$ be the prime 
 corresponding to the projection to the $i$-th component followed by the 
 map to the residue field. If $M$ is any  $(\ZZ/p^s)_I$-module then 
 we have a functorial map $M\rightarrow \prod_{i\in I} M_{x_i}$. 
 If $M=H^j((\PP^{n-1}_{\Cp}/K)_{\et}, \FF_{\Pi})$ then $M_{x_i}=H^j((\PP^{n-1}_{\Cp}/K)_{\et}, \FF_{\pi_i})$
 since $\Pi_{x_i}\cong \pi_i$ and cohomology commutes with localization.

We have an exact sequence
\[ 0\rightarrow \prod_{i\in I} \sigma_i\rightarrow \prod_{i\in I} \pi_i \rightarrow \prod_{i\in I} \pi_i/p\rightarrow 0,\]
where $\sigma_i=\Ker( \pi_i \rightarrow \pi_i/p)$. Assume that $0\le m_i\le m$ for all $i\in I$. 
Let $\sigma=\Ker (\pi\rightarrow \pi/p)$ then we may rewrite the restriction of the above sequence
to $H$ as  
\[ 0\rightarrow \bigoplus_{k=0}^m \prod_{i\in I_k} \sigma^k \rightarrow \bigoplus_{k=0}^m \prod_{i\in I_k} 
\pi^k
\rightarrow  \bigoplus_{k=0}^m \prod_{i\in I_k} 
(\pi/p)^k\rightarrow 0,\]
where $I_k=\{ i\in I: m_i=k\}$. It follows 
 from Lemma \ref{smooth_exact} that 
the sequence 
\[0\rightarrow (\prod_{i\in I} \sigma_i)^{\sm}\rightarrow (\prod_{i\in I} \pi_i)^{\sm} \rightarrow (\prod_{i\in I} \pi_i/p)^{\sm}\rightarrow 0\]
remains exact. Let $\Sigma= (\prod_{i\in I} \sigma_i)^{\sm}$ and let $P=(\prod_{i\in I} \pi_i/p)^{\sm}$. The exact 
sequence $0\rightarrow \Sigma\rightarrow \Pi\rightarrow P\rightarrow 0$ induces an 
exact sequence of sheaves $0\rightarrow \FF_{\Sigma}\rightarrow \FF_{\Pi}\rightarrow \FF_P\rightarrow 0$
and a long exact sequence in cohomology. 
Since $\sigma_i|_H \cong \sigma^{m_i}$ and $(\pi_i/p)|_H \cong (\pi/p)^{m_i}$ are killed by $p^{s-1}$,
 the induction
hypothesis implies that the maps 
\[ H^j((\PP^{n-1}_{\Cp}/K)_{\et}, \FF_{\Sigma})\rightarrow \prod_{i\in I} H^j((\PP^{n-1}_{\Cp}/K)_{\et}, \FF_{\sigma_i})\]
\[ H^j((\PP^{n-1}_{\Cp}/K)_{\et}, \FF_{P})\rightarrow \prod_{i\in I} H^j((\PP^{n-1}_{\Cp}/K)_{\et}, \FF_{\pi_i/p})\]
are isomorphisms for all $j\ge 0$. An application of (short) $5$-lemma implies the assertion.
 \end{proof}
 
 From now on we will assume that $\pi_i\in \Mod^{\adm}_G(\Fp)$ and will prove  Theorem 
\ref{main_ultra} following the proof of \cite[Theorem 8.3]{scholze}.
 
 \begin{lem}\label{cartan-serre}
 Let $m$ be a nonnegative integer and let $d=(m+1)3^{m+1}$. Let 
 \[E_{*, (k)}^{i,j}\Longrightarrow M_{(k)}^{i+j},\,\, 0\leq k\leq d\] 
 be upper-right quadrant spectral sequences of $(\Fp)_I\otimes_{\mathbb{F}_p} \OO_{\Cp}/p$-modules
 together with maps of spectral sequences between the $k$-th and the $(k+1)$-th spectral sequence for $0\leq k<d$. Suppose that for some $r$ the map 
 $E_{r, (k)}^{i,j}\to E_{r, (k+1)}^{i,j}$ factors over an almost finitely presented 
 $(\Fp)_I\otimes_{\mathbb{F}_p} \OO_{\Cp}/p$-module for all $i,j,k$. Then 
 $M_{(0)}^{k}\to M_{(d)}^k$ factors over an almost finitely presented  $(\Fp)_I\otimes_{\mathbb{F}_p} \OO_{\Cp}/p$-module for $k\leq m$. 
 \end{lem}
 
 \begin{proof}
 Follow the proof of Lemma 10.5.6 in \cite{scholzeberk} with the ring $R^+/p$ in loc.\,cit. replaced by $(\Fp)_I\otimes_{\mathbb{F}_p} \OO_{\Cp}/p$ (the key point being that this ring is almost coherent by Corollary 8.7 in \cite{scholze}, and this is what is used in the proof of the 10.5.6). 
 \end{proof}

 \begin{lem}\label{key_lemma}
 For any $m\geq 0$ there is a compact open subgroup $K_0\subset D^{\times}$ stabilizing 
 $V,U$ and the section $V\to \mathcal{M}_{\rm LT, 0, C}$ such that for all 
 $K\subset K_0$  the map 
 \[H^j((V/K)_\et, \mathcal{F}_{\Pi}\otimes \mathcal{O}^+/p)\to H^j((U/K)_\et, \mathcal{F}_{\Pi}\otimes \mathcal{O}^+/p)\]
 factors over an almost finitely presented $(\Fp)_I\otimes_{\mathbb{F}_p} \OO_{\Cp}/p$-module for $0\leq j\leq m$.
 \end{lem}
 
 \begin{proof} Let $\Pi_j$ be as in Lemma \ref{resolution} above. 
 The proof of Lemma 8.8 in \cite{scholze} (see also Remark 10.5.7 in \cite{scholzeberk}) shows that the desired result holds with $\Pi_j$ instead of 
 $\Pi$, so we only need to explain how to deduce it for $\Pi$. 
 By exactness of the functor $\pi\mapsto \mathcal{F}_{\pi}$ we have a spectral sequence
 \[E_{1,U}^{i,j}=H^j((U/K)_\et, \mathcal{F}_{\Pi_i}\otimes \mathcal{O}^+/p)\Longrightarrow 
 H^{i+j}((U/K)_\et,  \mathcal{F}_{\Pi}\otimes \mathcal{O}^+/p).\]
 Let $d=(m+1)3^{m+1}$ and pick a sequence of strict inclusions $V=U_0\subset U_2\subset...\subset U_d=U$. This induces $d+1$ spectral sequences $E_{\ast, U_k}^{\bullet, \bullet}$ as above, to which we will apply Lemma \ref{cartan-serre}. It suffices therefore to check that the maps $E_{1,U_k}^{i,j}\to E_{1,U_{k-1}}^{i,j}$ factor over almost finitely presented $(\Fp)_I\otimes_{\mathbb{F}_p} \OO_{\Cp}/p$-modules for all $i,j$ and $k$. As we have already mentioned, this is exactly what is proved in Lemma 8.8 of \cite{scholze}, for the strict inclusion $U_{k-1}\subset U_{k}$ and the representation $\Pi_j$. 
 \end{proof}
  
  \begin{proof}[Proof of Theorem \ref{main_ultra}] It follows from Lemma \ref{induction_step} 
  that it is enough to prove the Theorem for $s=1$. Lemma \ref{key_lemma} is a substitute 
  for \cite[Lemma 8.8]{scholze}. Given this, the proof of \cite[Theorem 8.3, Corollary 8.5]{scholze}, where
  Scholze assumes that $\pi_i$ are injective as $H$-representations,  
  carries over verbatim to our more general setting.
  \end{proof}
 
  \section{Global arguments}
 
 In this section we assume that $p>2$. We fix a continuous representation  \[\rhobar: \Gal_{\Qp}\rightarrow \GL_2(\Fpbar).\] We will first globalise $\rhobar$ following Appendix A in \cite{gee-kisin}. 
 
 \begin{prop}\label{globalize} Let $\mathbb{F}$ be a non-trivial finite extension of $\mathbb F_p$, such that $\GL_2(\mathbb{F})$  contains the image of $\rhobar$. There is a totally real field $F$ and a regular algebraic cuspidal automorphic weight $0$ representation 
 $\pi$ of $\GL_2(\mathbb A_F)$, such that the following hold:
 \begin{enumerate}
 \item the  Galois representation $r_{\pi}: \Gal_F\rightarrow \GL_2(\Qpbar)$ associated to $\pi$ is unramified outside $p$; 
 \item $p$ splits completely in $F$ and $\rbar_{\pi}|_{\Gal_{F_v}}\cong \rhobar$ for all $v\mid p$;
 \item $\SL_2(\mathbb{F})\subset \rbar_{\pi}(\Gal_F)\subset \GL_2(\mathbb F)$;
 \item $[F:\mathbb Q]$ is even.
 \end{enumerate}
\end{prop}
\begin{proof} Using \cite[Cor.\,A.3]{gee-kisin} we may find a totally real field $F'$ and a regular algebraic cuspidal automorphic representation 
 $\pi'$ of $\GL_2(\mathbb A_{F'})$ such that (2) holds, the conductor of $\pi'$ is potentially prime to $p$, $\rbar_{\pi'}(\Gal_{F'})=\GL_2(\mathbb{F})$. 
Since  the conductor of $\pi'$ is potentially prime to $p$ using solvable base change we replace $F'$ by a totally real solvable extension $F$, 
and $\pi$ with the base change of $\pi'$ to $F$ so that (1), (2) and (4) hold. Since $\mathbb F$ has at least $4$ elements $\SL_2(\mathbb{F})$ is equal to its own derived subgroup (Theorem 8.3 in \cite{Lang}) and thus  it will be contained in $\rbar_{\pi'}(\Gal_{F''})$ for any abelian extension $F''$ of $F'$. Thus   $\SL_2(\mathbb{F})$ will be contained in $\rbar_{\pi}(\Gal_F)=\rbar_{\pi'}(\Gal_{F})$.
 \end{proof} 
 
 We fix $\mathbb{F}$, $F$ and $\pi$ as in Proposition \ref{globalize} and assume that $[\mathbb F:\Fp]\ge 3$ and let $\rbar:=\rbar_{\pi}$. 
 
 \begin{lem}\label{diagonal}  Let $N$ be a positive integer.
Then there is a finite place $w_1$ of $F$ such that the following hold:
\begin{enumerate}
\item\label{part_1} the ratio of eigenvalues of $\rbar(\Frob_{w_1})$ is not in $\{1, \mathbf{N}(w_1), \mathbf N(w_1)^{-1}\}$;
\item\label{part_2} $\mathbf{N}(w_1)\not\equiv 1 \pmod{p}$;
\item\label{part_3} $\mathbf{N}(w_1)$ is prime to $2Np$.
\end{enumerate}
 \end{lem} 
 \begin{proof} Let $E$ be the fixed field of $\Ker \rbar$ and let $\zeta_p$ be 
 a primitive $p$-th root of unity in $\overline{E}$. We claim that we 
 may choose
 $g\in \Gal(E(\zeta_p)/F)$ such that 
  the ratio of eigenvalues of $\rbar(g)$ is not in $\{1, \mathbf{N}(w_1), \mathbf N(w_1)^{-1}\}$
% the image of $g$ in $\Gal(E/F)$ satisfies \eqref{part_1} 
  and the image of $g$ in $\Gal(F(\zeta_p)/F)$ is non-trivial. 
 Granting the claim, by Chebotarev density there will exist infinitely many $w_1$ such that $\Frob_{w_1}|_{E(\zeta_p)}$
 is conjugate to $g$ in $\Gal(E(\zeta_p)/F)$. All such $w_1$ satisfy \eqref{part_2}.  
 Since there are infinitely many we may choose $w_1$ as above 
 with $\mathbf{N}(w_1) > (Np)^{[F:\QQ]}$. Then $\mathbf{N}(w_1)$ is 
 a power of a prime $\ell$ such that $\ell> Np$. In particular \eqref{part_3} holds. 
  
  To prove the claim we let $\xi$ be a multiplicative generator of $\mathbb F^{\times}$.
  If $\zeta_p\not\in E$ then pick any $g\in \Gal(E(\zeta_p)/F)$, which maps to 
  $\bigl( \begin{smallmatrix} \xi & 0 \\ 0 & \xi^{-1}\end{smallmatrix}\bigr)$ in $\SL_2(\mathbb F)\subset 
  \Gal(E/F)$, where we have identified $\Gal(E/F)$ with $\rbar(\Gal_F)$.  Since $[\mathbb F:\Fp]\ge 3$, $\xi^2\not \in \Fp$ and thus any such $g$ will satisfy (1).  
  If the image of $g$ in $\Gal(F(\zeta_p)/F)$ is trivial then replace $g$ by $gk$ for 
  any non-trivial $k\in \Gal(E(\zeta_p)/E)$. The images of $g$ and $gk$ in $\Gal(E/F)$ are equal and the 
  image of $gk$ in $\Gal(F(\zeta_p)/F)$ is equal to the image of $k$ and is non-trivial.

     If $\zeta_p\in E$ 
  then $F(\zeta_p)$ is contained in $E^{\SL_2(\mathbb F)}$, since $F(\zeta_p)/F$ is abelian and 
  $\SL_2(\mathbb F)$ is its own derived subgroup.  Since  $\Gal(E/F)$ is a subgroup 
  of $\GL_2(\mathbb F)$ by construction, $\Gal(E/F)/\SL_2(\mathbb F)$ is a subgroup 
  of $\mathbb F^{\times}$. Thus it is generated by $\xi^d$ for some divisor $d$ of $|\mathbb F^{\times}|$. Thus $g= \bigl( \begin{smallmatrix} \xi^{d+1} & 0 \\ 0 & \xi^{-1}\end{smallmatrix}\bigr)\in \GL_2(\mathbb F)$ lies in $\Gal(E/F)$.  
  The ratio of its eigenvalues is equal to $(\xi^{d+2})^{\pm 1}$. If $\xi^{d+2}\in \mathbb F_p$ then 
  $p^f-1$ divides $(d+2)(p-1)$, where $f=[\mathbb F:\Fp]$, and hence $d$ divides $(d+2)(p-1)$. If $d$ is odd then $d$ and $d+2$ are coprime and so $d$ has to divide $p-1$ and so $p^f-1\le (p+1)(p-1)$. If $d$ is even then greatest common divisor of $d$ and $d+2$ is $2$ and so $d/2$ divides $p-1$. Thus $p^f-1\le 2p (p-1)$.
  In both cases we obtain a contradiction to $[\mathbb F: \Fp]\ge 3$. The image of $g$ will generate $\Gal(F(\zeta_p)/F)$, which is non-trivial as $F$ is totally real. 
\end{proof} 

\begin{remar}\label{kisin_ok} Let $w_1$ be any place of $F$ satisfying 
the conditions of Lemma \ref{diagonal} and let $\Sigma_p$ be the set of places of $F$ above 
$p$.  If we let $S$ be a finite set of places of $F$ containing 
$w_1$, $\Sigma_p$ and all the places above $\infty$ then
 $\rbar$, $S$, $\Sigma_p$ satisfies
 all the conditions of section (2.2) in \cite{kisin_fmc}. In particular, the issues discussed 
 in Appendix B of \cite{gee-kisin} do not 
arise in our situation.
\end{remar}

\subsection{Completed cohomology}
Let $D_0$ be a quaternion algebra with center $F$ ramified at all infinite places and 
split at all finite places. Let $\OO_{D_0}$ be a maximal order in $D$ and we fix an isomorphism
 $(\OO_{D_0})_v\cong M_2(\OO_{F_v})$ for every finite $v$. Then  
 $U_{\mathrm{max}}:= \prod_{v\nmid \infty} \GL_2(\OO_{F_v})$ 
 is an open subgroup of $(D_0\otimes_F \mathbb A_F^{\infty})^{\times}$. 
 One may write $(D_0\otimes_F \mathbb A_F^{\infty})^{\times}$ as a finite union of double cosets 
 of the form $D_0^{\times} t_i U_{\mathrm{max}} (\mathbb A_F^{\infty})^{\times}$, where we have identified 
 $(\mathbb A_F^{\infty})^{\times}$ with the center of $(D_0\otimes_F \mathbb A_F^{\infty})^{\times}$.
 Moreover, the groups 
 $(U_{\mathrm{max}} (\mathbb A_F^{\infty})^{\times}\cap t_i D_0^{\times} t_i^{-1})/F^{\times}$ 
 are finite.   Let $w_1$ be a finite place of satisfying the conditions of Lemma \ref{diagonal}
 with $N$ equal to the product of the orders of these groups. Let $U:=\prod_{w\nmid \infty} U_w$, $U^p:=\prod_{w\nmid p\infty} U_w$,
 where $U_w= \GL_2(\OO_{F_w})$ if $w\neq w_1$ and let 
 $U_{w_1}= \{ g\in \GL_2(\OO_{F_{w_1}}): g\equiv \bigl(\begin{smallmatrix} 1 & \ast\\ 0 & 1\end{smallmatrix}\bigr)\pmod{\varpi_{w_1}}\}$. It follows from \cite[Lem.\,3.2]{2adic} that 
 \begin{equation}\label{schnitt}
 (U (\mathbb A_F^{\infty})^{\times}\cap t D^{\times}_0 t^{-1})/F^{\times}=1,
  \quad \forall t\in (D_0\otimes_F \mathbb A_F^{\infty})^{\times}.
  \end{equation}
If $A$ is a topological $\OO$-module, such as $\OO$, $L$ with the $p$-adic topology or $\OO/\varpi^n$, $L/\OO$ with the discrete topology, we denote by $S(U^p, A)$ 
the space of continuous functions 
 \[ f: D_0^{\times}\backslash (D_0\otimes_F \AfF)^{\times}/U^p\rightarrow A.\]
   The groups  $(D_0\otimes\Qp)^\times$ and $(\AfF)^{\times}$ act continuously on $S(U^p, A)$ by right translations. 
   
  Let $\TT^{\univ}_{S}=\OO[ T_v, S_v]_{v\not\in S}$ 
be a commutative polynomial ring in the indicated formal variables, where $v$ ranges over all finite places of $F$ not in the subset $S$ defined in Remark \ref{kisin_ok}.
The algebra $\TT^{\univ}_{S}$ acts on $S(U^p, A)$ 
with $S_v$ acting via the double coset 
$U_v \bigl ( \begin{smallmatrix} \varpi_v & 0 \\ 0 & \varpi_v\end{smallmatrix}\bigr) U_v$ and $T_v$ acting via the double coset 
$U_v \bigl ( \begin{smallmatrix} \varpi_v & 0 \\ 0 & 1\end{smallmatrix}\bigr) U_v$. This action commutes with the action of $(D_0\otimes\Qp)^\times$ and $(\AfF)^{\times}$.

Let $\Gal_{F,S}=\Gal(F_S/F)$ be 
the Galois group of the maximal extension of $F$ in $\overline{F}$ which is unramified outside $S$.  Let $r_{\pi}: \Gal_{F,S} \rightarrow \GL_2(\Qpbar)$ be the 
Galois representation associated to the automorphic form $\pi$ in Proposition \ref{globalize}. After conjugating (and possibly replacing $L$ by a finite extension), we may 
assume that $r_{\pi}$ takes values in $\GL_2(\OO)$. Let $\mm_{r_{\pi}}$ be the ideal of  $\TT^{\univ}_{S}$ generated by $T_v -\tr r_{\pi}(\Frob_v)$, 
$\mathbf{N}(v) S_v - \det r_{\pi}(\Frob_v)$ for all $v\not \in S$ and let $\mm:= \mm_{\bar{r}_{\pi}}:= (\varpi, \mm_{r_{\pi}})$. We have  
$\TT^{\univ}_{S}/ \mm_{r_{\pi}}=\OO$ and $\TT^{\univ}_S/\mm=k$. 

Let $\psi: \Gal_{F, S}\rightarrow \OO^{\times}$ be a character such that $\det r_{\pi} = \psi \varepsilon^{-1}$, where
$\varepsilon$ is the $p$-adic cyclotomic character, so that $\psi(\Frob_v)\equiv S_v \pmod{ \mm_{r_{\pi}}}$ for all $v\not \in S$.  We consider $\psi$ as a character of $(\AfF )^{\times}/F^{\times}$ via the class field theory normalized so that the uniformizers are sent to geometric Frobenii. We will denote the restriction of $\psi$ to a decomposition subgroup at $v$ by 
$\psi_{v}$, and we will also consider $\psi_{v}$ as a character of $F_v^{\times}$ via local class field theory.   

Let $S_{\psi}(U^p, A)$ be the submodule\footnote{VP would like to thank Yongquan Hu for pointing out that in the statement of \cite[Lemma 5.3]{ludwig} and in the last line of its proof $\chi_{\cyc}$ should
be replaced by its inverse.} 
of $S(U^p, A)$ on which $(\AfF )^{\times}$ acts by $\psi$. Since the actions of $(\AfF )^{\times}$, $(D_0\otimes_{\QQ} \Qp)^{\times}$ and $\TT^{\univ}_S$ on $S(U^p, A)$ commute, $S_{\psi}(U^p, A)$ is naturally a $\TT^{\univ}_S[ (D_0\otimes_\QQ \Qp)^{\times}]$-module.

\begin{prop}\label{happy_prop} Let $\mm=\mm_{\rbar_{\pi}}$ and $\psi$ be as above. The localization $S_{\psi}(U^p, L/\OO)_{\mm}$ is non-zero. Moreover, 
it is admissible  and injective in $\Mod^{\sm}_{\psi_p}(U_p)$, where \[U_p= (\OO_{D_0}\otimes \Zp)^{\times}= \prod_{v\mid p} \GL_2(\Zp)\] and 
$\psi_p$ is restriction of $\psi$ to the center $Z_p$ of $U_p$. 
\end{prop} 
\begin{proof} We pick an isomorphism between the algebraic closure $\overline{L}$ and $\mathbb C$. Then we may  write $\pi^{\infty}=\otimes'_v \pi_v$, where the restricted tensor product is taken over the finite places of $F$, and each $\pi_v$ is a smooth representation 
of $(D_0\otimes_F F_v)^{\times}=\GL_2(F_v)$ on an $\overline{L}$-vector space.  Let us choose an open normal subgroup $U'_p=\prod_{v\mid p} U_v'$ of $U_p$ such that $\pi_v^{U'_v}\neq 0$ for all $v\mid p$, and let $U'= U'_p \times \prod_{v\nmid p} U_v$. It follows from  \cite[Lemma 1.3 (4)]{taylor} that 
$S_{\psi}(U^p, L)^{U_p'}\otimes_L \overline{L}$ contains $\pi^{U'}$ as a $\TT^{\univ}_S[U_p]$-submodule. More precisely,  $\pi^{U'}$ is identified with the Hecke eigenspace
$(S_{\psi}(U^p, L)^{U_p'}\otimes_L \overline{L})[\mm_{r_{\pi}}]$. In particular, $(S_{\psi}(U^p, L)^{U_p'})[\mm_{r_{\pi}}]$ is a non-zero finite dimensional $L$-vector space. 
It follows from \eqref{schnitt} that if $A=L$, $\OO$ or $k$ then $S_{\psi}(U^p, A)^{U_p'}$ can be identified as an $U_p$-representation with a finite direct sum of copies of  $\Ind_{Z_p U'_p}^{U_p}{ \psi_p \otimes_{\OO} A}$.  In particular,  $S_{\psi}(U^p, \OO)^{U_p'}$ is an $\OO$-lattice in $S_{\psi}(U^p, L)^{U_p'}$ and its reduction modulo $\varpi$ is equal to $S_{\psi}(U^p, k)^{U_p'}$. Hence, $(S_{\psi}(U^p, \OO)^{U_p'})[\mm_{r_{\pi}}]$ and $(S_{\psi}(U^p, k)^{U_p'})[\mm]$ are both non-zero. Since localization is an exact functor 
the inclusion $(S_{\psi}(U^p, k)^{U_p'})[\mm]\hookrightarrow S_{\psi}(U^p, L/\OO)$ implies that $S_{\psi}(U^p, L/\OO)_{\mm}\neq 0$. The assertion about injectivity and admissibility 
follows from \cite[Lemmas 5.1, 5.2]{ludwig}.\footnote{ In Equation (26) in the proof of \cite[Lemma 5.1]{ludwig} 
$\lambda^{\vee}$ should be $\Hom_{\OO}(\lambda, \OO)$.}
\end{proof}
\begin{remar} Let $\TT'$ be  a $\TT_S^{\univ}$-algebra, (for example obtained by adding 
the Hecke operator corresponding to the double coset $U_{w_1} \bigl ( \begin{smallmatrix} \varpi_{w_1} & 0 \\ 0 & 1\end{smallmatrix}\bigr) U_{w_1}$) and let $\mm'_{r_{\pi}}$ be an ideal of $\TT'$
containing $\mm_{r_{\pi}}\TT'$, such that $\TT'/\mm'_{r_{\pi}}=\OO$ and let $\mm'=(\varpi, \mm'_{r_{\pi}})$.  If $\TT'$ acts 
on $S_{\psi}(U^p, L/\OO)$ and the action extends the action of $\TT_S^{\univ}$, and commutes 
with the action of   $(D_0\otimes\Qp)^\times$ then the argument of Proposition \ref{happy_prop}
shows that $S_{\psi}(U^p, L/\OO)_{\mm'}$ is non-zero. 
\end{remar}

We fix a place $\pp$ of $F$ above $p$ and let $U^{\pp}_p= \prod_{v\mid p, v\neq \pp} U_v$. If $\lambda$ is a continuous representation of
 $U^{\pp}_p$ on a free $\OO$-module of finite rank, such that $(\AfF)^\times\cap U^{\pp}_p$ acts 
 on $\lambda$ by the restriction of $\psi$ to this group then we let 
 \[S_{\psi, \lambda}(U^{\pp}, A):= \Hom_{U^\pp_p}(\lambda, S_\psi(U^p, A)).\]
Since $p>2$ for each $v\mid p$ we may pick a smooth character $\theta_v: \OO_{F_v}^{\times} \rightarrow \OO^{\times}$ such that
$\psi_v(x)=\theta_v(x)^2$ for all $x\in 1+(\varpi_v)$. Let $\theta: (\OO_F\otimes \Zp)^{\times}\rightarrow \OO^{\times}$ be the product of $\theta_v$.

\begin{lem}\label{alt_bier} We may choose $\lambda=\otimes_{v\mid p, v\neq \pp} \lambda_v$ such that $S_{\psi, \lambda}(U^{\pp}, L/\OO)_{\mm}\neq 0$ and 
$\lambda_v= \sigma_v^0 \otimes (\theta_v\circ \det)$ such that the central character of $\lambda_v$ is equal to $\psi_v$ and 
$\sigma_v^0$ is an $\OO$-lattice in an inflation of an irreducible non-cuspidal representation $\sigma_v$ of $\GL_2(\Fp)$ to $U_v$. 
\end{lem} 
\begin{proof} After twisting by the inverse of $\theta\circ \det$ we may assume that $\psi_v(x)=1$ for all $x\in 1+(\varpi_v)$ for all $v\mid p$. 
Let $U'_v$ be the subgroup of $U_v= \GL_2(\Zp)$ consisting of upper-triangular unipotent matrices modulo $p$, and let $U'_p=\prod_{v\mid p} U_v'$.
Let $Z_p'$ be the centre of $U'_v$. It follows from   Proposition \ref{happy_prop} that $S_{\psi}(U^p, L/\OO)_{\mm}$ is an admissible injective representation 
of $U'_p/ Z'_p$. Since $U'_p/ Z'_p$ is pro-$p$, $S_{\psi}(U^p, L/\OO)_{\mm}$ as a representation of $U'_p$ is isomorphic to a finite direct sum of 
copies of $\mathcal C(U'_p/ Z'_p, L/\OO)$, where $\mathcal C$ denotes continuous functions. The Pontryagin dual of $\mathcal C(U'_p/ Z'_p, L/\OO)$ is the completed group algebra
$\OO\br{U'_p/ Z'_p}$, and its Schikhof dual is isomorphic to $\mathcal C(U'_p/ Z'_p, \OO)$. It follows from \cite[Equation (27)]{ludwig} that 
$S_{\psi}(U^p, \OO)_{\mm}|_{U'_p}$ is isomorphic to a finite direct sum of copies of $\mathcal C(U'_p/ Z'_p, \OO)$. In particular, $(S_{\psi}(U^p, \OO)_{\mm})^{U'_p}$
is a non-zero $\OO$-module of finite rank. Let $W$ be the $U_p^{\pp}$-invariant subspace of $S_{\psi}(U^p, \OO)_{\mm} \otimes_{\OO} L$ generated by $(S_{\psi}(U^p, \OO)_{\mm})^{U'_p}$ and let $\sigma$ be any irreducible subquotient of $W$. 
The centre of $U_p^{\pp}$ acts on $\sigma$ via 
$\psi_p^{\pp}$. We may write $\sigma= \otimes_{v\mid p, v\neq \pp} \sigma_v$, such that 
$\sigma_v$ is an irreducible representation of $U_v$ satisfying $\sigma_v^{U_v'}\neq 0$.
In particular, $\sigma_v$ is an inflation of an irreducible non-cuspidal representation 
of $\GL_2(\Fp)$ to $U_v$. Let $\sigma_v^0$ be any $U_v$-invariant $\OO$-lattice 
in $\sigma_v$ and let $\lambda=\otimes_{v\mid p, v\neq \pp} \sigma_v^0$. Then by construction 
$\lambda$ is an $\OO$-lattice in $\sigma$, hence $\Hom_{U^{\pp}_p}(\lambda,  S_{\psi}(U^p, \OO)_{\mm})$ is a non-zero $\OO$-module of finite rank. This implies that 
$\Hom_{U^{\pp}_p}(\lambda,  S_{\psi}(U^p, k)_{\mm})$ is also non-zero, which implies that 
$\Hom_{U^{\pp}_p}(\lambda,  S_{\psi}(U^p, L/\OO)_{\mm})$ is non-zero. 
\end{proof} 
\begin{remar} Standard arguments with Serre weights show that if $\rbar_{\pi}|_{G_{F_v}}$ is 
not tr\`es ramifi\'ee then one may choose $\sigma_v$ in Lemma \ref{alt_bier} to be either a character or a principal series
representation of $\GL_2(\Fp)$. 
\end{remar}

\begin{lem}\label{non-Eis} Let $\lambda$ be a continuous representation of $U_p^{\pp}$  with a central character $\psi_p^{\pp}$ on a 
free finite rank $\OO$-module.  Let $G_{\pp}'$ be 
the derived subgroup of $(D_0\otimes_F F_{\pp})^{\times}$. Then 
$S_{\psi, \lambda}(U^{\pp}, k)_{\mm}^{G_{\pp}'}=0$. 
\end{lem}
\begin{proof} We identify $(D_0\otimes_F F_{\pp})^{\times}$ with $\GL_2(\Qp)$ and 
$G_{\pp}'$ with $\SL_2(\Qp)$. We may assume that $S_{\psi, \lambda}(U^{\pp}, k)_{\mm}\neq 0$. 
We may further assume that $\psi_{\pp}(-1)=1$. 

If $K$ is a compact open subgroup of $\SL_2(\Qp)$ containing $\{\pm 1\}$ then it follows from \cite[Lemma 5.3]{ludwig}
and the argument explained in the proof of Proposition \ref{happy_prop} that 
\[S_{\psi, \lambda}(U^{\pp}, \OO)|_K \cong \mathcal C( K/\{\pm 1\}, \OO)^{\oplus m}\]
for some $m\ge 1$. Since $S_{\psi, \lambda}(U^{\pp}, \OO)_{\mm}$ is a direct summand
of $S_{\psi, \lambda}(U^{\pp}, \OO)$, we obtain an isomorphism 
\begin{equation}\label{inv_modp}
  (S_{\psi, \lambda}(U^{\pp}, \OO)_{\mm}^K)\otimes_{\OO} k \cong S_{\psi, \lambda}(U^{\pp}, k)_{\mm}^K.
\end{equation}  
Let $\alpha= \bigl (\begin{smallmatrix} 0 & 1\\ p & 0\end{smallmatrix}\bigr)$. Since $\SL_2(\Qp)$ 
is generated by its subgroups \[K_0:=\SL_2(\Zp)\,\, \text{and}\,\, K_1:=\alpha \SL_2(\Zp)\alpha^{-1}\]
for any $\SL_2(\Qp)$-representation $\pi$ we have an exact sequence 
\[0\rightarrow \pi^{\SL_2(\Qp)}\rightarrow \pi^{K_0} \oplus \pi^{K_1} \rightarrow \pi^{K_0\cap K_1}.\]
Applying this observation to $\pi=S_{\psi, \lambda}(U^{\pp}, \OO)_{\mm}$, $\pi= S_{\psi, \lambda}(U^{\pp}, k)_{\mm}$ and using \eqref{inv_modp}
gives us an isomorphism 
\[(S_{\psi, \lambda}(U^{\pp}, \OO)_{\mm}^{\SL_2(\Qp)})\otimes_{\OO} k \cong 
S_{\psi, \lambda}(U^{\pp}, k)_{\mm}^{\SL_2(\Qp)},\]
which implies that $(S_{\psi, \lambda}(U^{\pp}, \OO)_{\mm}^{\SL_2(\Qp)})\otimes_{\OO} L$ is a
non-zero finite dimensional $L$-vector space. Thus it will contain a non-zero Hecke eigenspace
 corresponding to an $\OO$-algebra homomorphism $x: (\TT^{\univ}_S)_{\mm}\rightarrow \overline{L}$. 
 Since the action of $\TT^{\univ}_S$ preserves the $\OO$-lattice $S_{\psi, \lambda}(U^{\pp}, \OO)_{\mm}^{\SL_2(\Qp)}$, after enlarging $L$, we may assume that $x$ takes values in $\OO$.
 Since local factors of cuspidal automorphic representations are generic, it follows from \cite[Lemma 1.3 (4)]{taylor} that there is a continuous character $\chi: (\AfF)^{\times}/F^{\times}\rightarrow \OO^{\times}$
 such that $x(T_v)= (\mathbf{N}(v)+1)\chi(\varpi_v)$ and $x(S_v)= \chi(\varpi_v)^2$ for all $v\not\in S$.
 This implies that $\tr \rbar(\Frob_v)= (\mathbf{N}(v)+1)\bar{\chi}(\varpi_v)$ and $\det \rbar(\Frob_v)= 
 \mathbf{N}(v)\bar{\chi}(\varpi_v)^2$ for all $v\not \in S$, where $\bar{\chi}$ is the reduction of 
 $\chi$ modulo $\varpi$. Hence, the ratio of eigenvalues of $\rbar(\Frob_v)$ is equal to $\mathbf{N}(v)^{\pm 1}$ for all $v\not \in S$, which contradicts Lemma \ref{diagonal}.
 \end{proof}

We fix an infinite place $\infty_F$ of $F$ and let  $D$ be the quaternion algebra over $F$, which is 
ramified at $\pp$, split at $\infty_F$ and has the same ramification behaviour as $D_0$ at all the other places.
We fix an isomorphism 
\[ D_0\otimes_F \mathbb A_F^{\pp, \infty_F}\cong D \otimes_F \mathbb A_F^{\pp, \infty_F}.\]
This allows us to view the subgroup $U^{\pp}$ of $(D_0\otimes_F \AfF)^\times$, considered above,  as a subgroup of  $(D \otimes_F \AfF)^{\times}$. Let $D_{\pp}:=D\otimes_F F_{\pp}$. Then $D_{\pp}$ is the non-split quaternion algebra over $F_\pp=\Qp$.
If $K$ is an open subgroup of $\OO_{D_\pp}^{\times}U^\pp$ then we let $X(K)$ be the corresponding Shimura curve for $D/F$ defined over $F$.  We let 
\[\widehat{H}^i(U^p, \OO/\varpi^n):=\varinjlim_{K_{p}} 
H^i_{\et}(X(K_{p} U^{p})_{\overline{F}}, \OO/\varpi^n),\]
where the limit is taken over open subgroups of $U_{p}^{\pp} \OO_{D_\pp}^{\times}$. 
This group carries commuting actions of the Hecke algebra $\mathbb T^{\univ}_S$, $\Gal_{F, S}$, 
$(D\otimes_{\mathbb Q} \Qp)^{\times}$ and the centre $(\mathbb A_F^{\infty})^{\times}$. 
We let $\widehat{H}^i(U^p, \OO/\varpi^n)_{\mm}$ be the localisation of  
$\widehat{H}^i(U^p, \OO/\varpi^n)$ at $\mm$ and let
$\widehat{H}^i_{\psi}(U^p, \OO/\varpi^n)_{\mm}$ be the submodule on which 
$(\mathbb A_F^{\infty})^{\times}$ acts by $\psi$, and for $\lambda$ as in Lemma \ref{non-Eis} we let 
\[\widehat{H}^i_{\psi,\lambda}(U^{\pp}, \OO/\varpi^n)_{\mm}:=\Hom_{U^{\pp}_p}(\lambda, \widehat{H}^i_{\psi}(U^p, \OO/\varpi^n)_{\mm}).\]
After identifying $\frac{1}{\varpi^n}\OO/\OO \subset L/\OO$ with $\OO/\varpi^n$, we obtain an isomorphism of $\OO$-modules $\varinjlim_n \OO/\varpi^n \cong L/\OO$ and use the same transition maps to define
\[\widehat{H}^i_{\psi, \lambda}(U^{\pp}, L/\OO)_{\mm}:= \varinjlim_n \widehat{H}^i_{\psi, \lambda}(U^{\pp}, \OO/\varpi^n)_{\mm}.\]
It is explained in \cite[Lemma 6.2]{ludwig} that  Scholze's results in \cite{scholze} imply that 
there is a natural isomorphism of $\mathbb T^{\univ}_S[\Gal_{\Qp}\times D_\pp^{\times}]$-representations: 
\begin{equation}\label{scholze_0}
\mathcal{S}^1( S_{\psi, \lambda}(U^{\pp}, L/\OO)_{\mm})\cong \widehat{H}^1_{\psi, \lambda}(U^{\pp}, L/\OO)_{\mm}.
\end{equation}
We let 
\[\widehat{H}^i_{\psi, \lambda}(U^{\pp}, \OO)_{\mm}:= \varprojlim_n \widehat{H}^i_{\psi, \lambda}(U^{\pp}, \OO/\varpi^n)_{\mm},\]
equipped with the $p$-adic topology. By \cite[Lemma 9.16]{DPS} and \cite[Proposition 2.3]{Manning_Shotton}, the module $\varprojlim_n \widehat{H}^i(U^{\pp}, \OO/\varpi^n)_{\mm}$ is $\OO$-torsion free. Therefore the module $\widehat{H}^i_{\psi, \lambda}(U^{\pp}, \OO)_{\mm}$
is $\OO$-torsion free and by inverting $p$ we obtain an admissible unitary Banach space representation of  $D_{\pp}^{\times}$. Moreover, we have natural isomorphism
\begin{equation}\label{dual_0}
(\widehat{H}^1_{\psi, \lambda}(U^{\pp}, \OO)_{\mm})^d \cong (\widehat{H}^1_{\psi, \lambda}(U^{\pp}, L/\OO)_{\mm})^{\vee}
\end{equation}
 compatible with $\mathbb T^{\univ}_S[\Gal_{\Qp}\times D_\pp^{\times}]$-action. To ease the notation 
 we will omit the outer brackets, when taking the duals. It is proved in 
 \cite[Proposition 6.3]{ludwig} that \eqref{scholze_0} and \eqref{dual_0} induce a natural isomorphism
 \begin{equation}\label{sch1}
 \Sch^1(S_{\psi, \lambda}(U^{\pp}, \OO)_{\mm}^d)\cong \widehat{H}^1_{\psi, \lambda}(U^{\pp}, \OO)_{\mm}^d
 \end{equation}
 compatible with $\mathbb T^{\univ}_S[\Gal_{\Qp}\times D_\pp^{\times}]$-action.

\subsection{Patching}

If $R$ is a complete local noetherian $\OO$-algebra with residue field $k$  and $G$ is a $p$-adic analytic group then, following the notational scheme of \cite{ord1}, 
we let $\Mod^{\fgaug}_G(R)$ be the category of $R[G]$-modules $M$, such that for some open compact subgroup $H$ of $G$, the action of
$R[H]$ on $M$ extends to the action of $R\br{H}$, and $M$ is a finitely generated $R\br{H}$-module with respect to this action. We note that 
$M$ carries a canonical topology for which the action of $R\br{H}$ on $M$ is continuous. 

Restriction of deformations to the decomposition subgroup at $\pp$ induces a map $R^{\square, \psi_{\pp}}_{\rhobar} \rightarrow R^{\square, \psi}_{F, S}$. 
Since  the image of $\rbar$ contains $\SL_2(\mathbb F)$, $\rbar$ is absolutely irreducible and the map $R^{\psi}_{F, S}\rightarrow R^{\square, \psi}_{F, S}$, induced by forgetting the framings, is formally smooth 
and we may non-canonically identify $R^{\square, \psi}_{F, S}\cong R^{\psi}_{F, S}\br{z_1, \ldots, z_{4|S|-1}}$. By quotienting out the variables 
we obtain a surjection  $R^{\square, \psi}_{F, S}\twoheadrightarrow R^{\psi}_{F, S}$, which makes $R^{\psi}_{F, S}$ into a $R^{\square, \psi_{\pp}}_{\rhobar}$-algebra. 

\begin{defi}\label{thick} Let $G$ be a $p$-adic analytic group and let $M\in \Mod^{\fgaug}_G(R^{\psi}_{F, S})$.  A thickening of $M$ consists of the data 
$(R_{\infty}, M_\infty, \underline{x}, \varphi)$, where 
\begin{enumerate} 
\item $R_{\infty}$ is a complete local noetherian faithfully flat  $R^{\square, \psi_{\pp}}_{\rhobar}$-algebra with residue field $k$;
\item $M_{\infty}\in \Mod^{\fgaug}_G(R_{\infty})$;
\item\label{part3} $\underline{x}=(x_1, \ldots, x_h)$ is an $M_{\infty}$-regular sequence contained in the maximal ideal of $R_{\infty}$, such that
\[ M_{\infty}/(\underline{x})M_{\infty}\cong M;\]
\item $\varphi: R_{\infty}/(\underline{x})\twoheadrightarrow R^{\psi}_{F, S}$ is a surjection of $R^{\square, \psi_\pp}_{\rhobar}$-algebras, such that the 
action of $R_{\infty}$ on $M_{\infty}/(\underline{x})M_{\infty}$ factors through $\varphi$ and the resulting action of $R^{\psi}_{F, S}$ on $M_{\infty}/(\underline{x})M_{\infty}$ 
is compatible with the isomorphism in \eqref{part3}. 
\end{enumerate}
\end{defi}

Recall that $\theta_{\pp}: \Zp^{\times}\rightarrow \OO^{\times}$ is a character, such that
$\psi_{\pp}\theta_{\pp}^{-2}$ is trivial on $1+p\Zp$, and hence we may consider $\psi_{\pp}\theta_{\pp}^{-2}$ as a character of $\Fp^{\times}$. 
Let $\mu: \mathbb F_{p^2}^{\times} \rightarrow \OO^{\times}$ be a character such that the restriction of $\mu$ to $\Fp^{\times}$ is equal to $\psi_{\pp}\theta_{\pp}^{-2}$ and $\mu\neq \mu^p$. Let 
$\sigma(\mu)$ be the cuspidal representation of $\GL_2(\Fp)$ corresponding to the 
orbit $\{ \mu, \mu^p\}$, we also denote by $\sigma(\mu)$ its inflation to $\GL_2(\Zp)$ and let $\lambda_{\pp}:=\sigma(\mu)\otimes (\theta_{\pp}\circ \det)$. We inflate $\mu$ to the character of 
$\OO_{D_{\pp}}^{\times}$ via $\OO_{D_{\pp}}^{\times}/(1+\mm_{D_{\pp}})\cong \mathbb{F}_{p^2}^{\times}$. 
Let $\lambda'_{\pp}:= \mu (\theta_{\pp} \circ \Nrd)$ be a character of $\OO_{D^{\times}_{\pp}}$.
We note that $\lambda_{\pp}$ and $\lambda_{\pp}'$ have central character $\psi_{\pp}$, and it follows 
from \cite{BH} that the classical Jacquet--Langlands correspondence induces a bijection between the 
isomorphism classes of irreducible smooth representations $\pi_{\pp}$ of $\GL_2(\Qp)$ on 
$\overline{L}$-vector spaces such that $\Hom_{\GL_2(\Zp)}(\lambda_{\pp}, \pi_{\pp})\neq 0$ and 
the isomorphism classes of irreducible smooth representations $\pi_{\pp}'$ of $D_{\pp}^{\times}$ on 
$\overline{L}$-vector spaces such that 
$\Hom_{\OO_{D_{\pp}}^{\times}}(\lambda_{\pp}', \pi_{\pp}')\neq 0$. Moreover, all representations
$\pi_{\pp}$ as above are supercuspidal and their isomorphism classes  correspond bijectively to 
the isomorphism classes of irreducible  representations 
of the Weil group $W_{\Qp}$ of $\Qp$ of the form 
$\Ind^{W_{\Qp}}_{W_{\mathbb{Q}_{p^2}}}\chi$, such that  $\chi(\Art_{\mathbb{Q}_{p^2}}(x))= \mu(x) \theta_{\pp}(N^{\mathbb{Q}_{p^2}}_{\Qp} x)$ for all $x\in \ZZ_{p^2}^{\times}$, under 
the classical local Langlands correspondence.

Let $a> b$ be integers, and let $\sigma_{a, b}:=\lambda_{\pp} \otimes \Sym^{a-b-1} L^2 \otimes \det^{b+1}$ and let   $\sigma_{a, b}^0$ be a $\GL_2(\Zp)$-equivariant $\OO$-lattice
in $\sigma_{a, b}$. If $M_{\infty}$ is a thickening of $S_{\psi, \lambda}(U^{\pp}, \OO)^d_{\mm}$ with central character $\psi_{\pp}^{-1}$ 
then we let 
\[M_{\infty}(\sigma_{a,b}^0):= \Hom_{\OO\br{\GL_2(\Zp)}}^{\cont}( M_{\infty}, (\sigma_{a, b}^0)^d)^d\]
where $(\cdot)^d:=\Hom_{\OO}^{\cont}(\cdot, \OO)$. If $a+b+1=0$ then the central character of $\sigma_{a,b}$ is equal to $\psi_{\pp}$ and $(\varpi, \underline{x})$ is a regular 
system of parameters  for $M_{\infty}(\sigma_{a,b}^0)$, see \cite[Proposition 2.40]{duke_bm} for details.  In particular, $M_{\infty}(\sigma_{a,b}^0)$ is a finitely generated Cohen--Macaulay $R_{\infty}$-module. If $b+a+1\neq 0$ then $M_{\infty}(\sigma_{a,b}^0)=0$, since the centre acts on $M_{\infty}$ by $\psi_{\pp}^{-1}$, which 
is smooth. The $R_{\infty}[1/p]$-module 
$M_{\infty}(\sigma_{a,b}):= M_{\infty}(\sigma_{a,b}^0)\otimes_{\OO} L$ does not depend on the choice of lattice $\sigma_{a,b}^0$. Let $R_{\infty}(\sigma_{a,b})$ 
be the quotient of $R_{\infty}$, which acts faithfully on $M_{\infty}(\sigma_{a,b})$.

Similarly, we let $\sigma'_{a,b}= \lambda'_{\pp} \otimes \Sym^{a-b-1} L^2 \otimes \det^{b+1}$ and if $M_{\infty}'$ is a thickening of $\widehat{H}^1_{\psi, \lambda}(U^{\pp}, \OO)^d_{\mm}$ with central character $\psi_{\pp}^{-1}$ then we define $M'_{\infty}(\sigma'_{a,b})$ and $R_{\infty}(\sigma'_{a,b})$ analogously by replacing $\GL_2(\Zp)$ with $\OO_{D_{\pp}}^{\times}$. 

Let $R^{\square, \psi_{\pp}}_{\rhobar}(\sigma_{a,b})$ be the reduced and $\OO$-torsion free quotient of $R^{\square, \psi_{\pp}}_{\rhobar}$ parameterizing 
potentially semi-stable lifts of $\rhobar$ with Hodge--Tate weights $(a, b)$ and inertial Galois type defined by $\lambda_{\pp}$. Since the determinant in the deformation problem is fixed and is equal to $\psi_{\pp} \varepsilon^{-1}$, the ring $R^{\square, \psi_{\pp}}_{\rhobar}(\sigma_{a,b})$ is 
zero, unless $a+b=-1$.

\begin{thm}\label{the_patch} There is a thickening $(R_{\infty}, M_\infty, \underline{x}, \varphi)$ of $S_{\psi, \lambda}(U^{\pp}, \OO)^d_{\mm}$ and a thickening $(R_{\infty}, M'_\infty, \underline{x}, \varphi)$ of 
$\widehat{H}^1_{\psi, \lambda}(U^{\pp}, \OO)^d_{\mm}$ such that the following hold:
\begin{enumerate}
\item\label{ptch1} $M_{\infty}$ is projective in $\Mod^{\pro}_{\GL_2(\Zp), \psi_{\pp}}(\OO)$;
\item \label{ptch2}$M_{\infty}'$ is projective in $\Mod^{\pro}_{\OO_{D_{\pp}}^{\times}, \psi_{\pp}}(\OO)$;
\item  \label{ptch3}for all integers  $b>a$ the rings $R_{\infty}(\sigma_{a,b})$, $R_{\infty}(\sigma_{a,b}')$ are reduced and their spectra are equal to 
 a union of irreducible components of \[\Spec (R^{\square, \psi_{\pp}}_{\rhobar}(\sigma_{a,b})\otimes_{R^{\square, \psi_{\pp}}_{\rhobar}} R_{\infty});\]
\item \label{ptch4} $\Sch^1(M_{\infty})\cong M_{\infty}'$;
\end{enumerate}
where $\lambda$ is as in Lemma \ref{alt_bier}.
\end{thm}
\begin{proof} A thickening of $S_{\psi, \lambda}(U^{\pp}, \OO)^d_{\mm}$ satisfying \eqref{ptch1} and  \eqref{ptch3}  has been constructed by Dotto--Le in the proof of \cite[Theorem 6.2]{dotto-le} by using 
Scholze's improvement with ultrafilters of the patching argument in \cite{6auth} in the setting of quaternionic Shimura sets. Parts \eqref{ptch1} and  \eqref{ptch3} correspond to \cite[Proposition 2.10]{6auth} and \cite[Lemma 4.17  (1)]{6auth}, respectively.  Dotto--Le also carry out the patching  in the setting of Shimura curves. They assume that the quaternion algebra is split at $p$. However, this does not play a role 
in their arguments in the proof of \cite[Theorem 6.2]{dotto-le}, and hence we also obtain a thickening $(R_{\infty}, M'_\infty, \underline{x}, \varphi)$  of 
$\widehat{H}^1_{\psi, \lambda}(U^{\pp}, \OO)^d_{\mm}$ satisfying \eqref{ptch2} and \eqref{ptch3}. The main point of the Proposition is that 
if  we take the same utrafilter in these constructions then \eqref{ptch4} holds, so that $\Sch^1(M_{\infty})\cong M'_{\infty}$ as $R_{\infty}[\Gal_{\Qp}\times D^{\times}_p]$-modules. To prove that we need to unravel the construction of $M_{\infty}$ 
and $M_{\infty}'$ in \cite{dotto-le}. We employ the notations 
of \textit{loc.\,cit.}, albeit their $L$ is our $\Qp$, their $E$ is our $L$ and their $K^v$, $K(N)^v$,  $K_v$ is our $U^{\pp}$, $U(N)^{\pp}$, $U_{\pp}$. Moreover, when 
defining spaces of automorphic forms $S(K^v K_v, V)$ on definite quaternion algebras in \cite[Section 6.1]{dotto-le} they take functions with values in the Pontryagin dual of $V$ and 
we take functions with values in $V$, so that their $S(K^v K_v, \OO)$ is our $S(U^{\pp} U_{\pp}, L/\OO)$.

The ring $R_{\infty}$ is denoted by $R_{\infty}^{\psi}$ in \cite{dotto-le}. They show 
that $M_{\infty}$ is a finitely generated $S_{\infty}\br{\GL_2(\Zp)}$-module, where $S_{\infty}$ is a subring of $R_{\infty}$ of the form
\[S_{\infty}=\OO\br{z_1, \ldots, z_{4|S|-1}, y_1, \ldots, y_q}\] and 
\begin{equation}\label{kill_patching_var}
M_{\infty}/(z_1, \ldots, z_{4|S|-1}, y_1, \ldots, y_q)M_{\infty}\cong S_{\psi, \lambda}(U^{\pp}, \OO)^d_{\mm}.
\end{equation}
We let $\underline{x}=(z_1, \ldots, z_{4|S|-1}, y_1, \ldots, y_q)$. It is shown in \cite{dotto-le} that $M_{\infty}$ is a projective  $S_{\infty}\br{\GL_2(\Zp)}$-module
in the category of compact  $S_{\infty}\br{\GL_2(\Zp)}$-modules with central charachter $\psi_{\pp}^{-1}$. This implies that $\underline{x}$ is $M_{\infty}$-regular. 

The ring $R_{\infty}$ is a power series ring over $R^{\square,\psi_{\pp}}_{\rhobar}\wtimes_{\OO} \widehat{\bigotimes}_{v\mid p, v\neq \pp} R^{\square, \psi_v}_{\rbar_v}(\lambda_v)$ in carefully chosen number of variables, where $\rbar_v$ is the restriction of $\rbar$ to the decomposition group at $v$ and the completed tensor product is taken over $\OO$.  For each $N\ge 1$, $Q_N$ is a set of cardinality $q$ of 
Taylor--Wiles primes, see  \cite[Section 6.2.3]{dotto-le}. Let $\Delta_{Q_N}=\prod_{w\in Q_N} k_w^{\times}(p)$, where $k_w^{\times}(p)$ is the Sylow $p$-group 
of the multiplicative group of the residue field at $w$. Then by construction of $Q_N$, $p^N$ divides the order of $k_w^{\times}$ and hence  is a surjection $\Delta_{Q_N} \twoheadrightarrow (\ZZ/p^N)^q$. Let $\OO_{L, N}=\OO[\Delta_{Q_N}]$ be the  group ring of $\Delta_{Q_N}$. 
By choosing   surjections $\Zp^q\twoheadrightarrow \Delta_{Q_N}$  we obtain surjections
$\OO_{L, \infty}:=\OO\br{y_1, \ldots, y_q}=\OO\br{\Zp^q} \twoheadrightarrow \OO_{L, N}$ for all $N\ge 1$.  
Let 
\begin{equation}\label{ten}
\begin{split}
 M(N):&= S_{\psi, \lambda}(U(N)^{\pp}, \OO)^d_{\mm'_{Q_N}} \wtimes_{R^{\psi}_{F, S_{Q_N}}} R^{\square, \psi}_{F, S_{Q_N}}\\
 &\cong 
S_{\psi, \lambda}(U(N)^{\pp}, \OO)^d_{\mm'_{Q_N}} \wtimes_{\OO} \OO\br{z_1, \ldots, z_{4|S|-1}},
\end{split}
\end{equation}
\begin{equation}
\begin{split}
 M'(N):&= \widehat{H}^1_{\psi, \lambda}(U(N)^{\pp}, \OO)^d_{\mm'_{Q_N}} \wtimes_{R^{\psi}_{F, S_{Q_N}}} R^{\square, \psi}_{F, S_{Q_N}}\\
 &\cong 
\widehat{H}^1_{\psi, \lambda}(U(N)^{\pp}, \OO)^d_{\mm'_{Q_N}} \wtimes_{\OO} \OO\br{z_1, \ldots, z_{4|S|-1}},
\end{split}
\end{equation}
where we have fixed an isomorphism $R^{\psi, \square}_{F, S_{Q_N}}\cong R^{\psi}_{F, S_{Q_N}}\br{z_1, \ldots, z_{4|S|-1}}$ as in \cite[Section 6.2.4]{dotto-le}. 
The ring $S_{\infty}=\OO_{L, \infty}\br{z_1, \ldots, z_{4|S|-1}}$ acts on $M(N)$ and $M'(N)$ and 
the action factors through the quotient $S_N= \OO_{L,N}\br{z_1, \ldots, z_{4|S|-1}}$.
If $J$ is an open ideal of $\OO_{L, \infty}$ and $\mathfrak a$ is an open ideal of $\OO\br{z_1, \ldots, z_{4|S|-1}}$ then we let 
\[M(\mathfrak a, J, N):= M(N)/(\mathfrak a, J) M(N), \quad M'(\mathfrak a, J, N):= M'(N)/(\mathfrak a, J) M'(N).\]
If $U_{\pp} \subset \GL_2(\Qp)$, $U_{\pp}' \subset D^{\times}_v$ are compact open subgroups then we 
let 
\[M(\mathfrak a, J, U_{\pp}, N):= M(\mathfrak a, J, N)_{U_{\pp}}, \quad 
M'(\mathfrak a, J, U_{\pp}', N):= M(\mathfrak a, J, N)_{U_{\pp}'}.\]
denote the coinvariants. Note that the modules $M(\mathfrak a, J, U_{\pp}, N)$ and $M'(\mathfrak a, J, U_{\pp}', N)$ are of finite cardinality. 
For each open ideal $J$ of $\OO_{L, \infty}$ let $I_J$ be a cofinite subset of $I$ such that the kernel of $\OO_{L, \infty}\rightarrow \OO_{L, N}$ is contained in 
$J$.  We fix a non-principal ultrafilter $\mathfrak F$ on the set of natural numbers. Then we let 
\[  M(\mathfrak a, J, U_{\pp}, \infty):=(\OO_{L, \infty}/J)_{I_J, x} \otimes_{(\OO_{L, \infty}/J)_{I_J}} \prod_{N\in I_J} M(\mathfrak a, J, U_{\pp},  N),\]
 \[  M'(\mathfrak a, J, U_{\pp}', \infty):=(\OO_{L, \infty}/J)_{I_J, x} \otimes_{(\OO_{L, \infty}/J)_{I_J}} \prod_{N\in I_J} M'(\mathfrak a, J, U_{\pp}',  N),\]
where $x\in \Spec (\OO_{L, \infty}/J)_{I_J}$ corresponds to the ultrafilter $\mathfrak F$. Finally, we let  
\[M(\mathfrak a, J, \infty)=\varprojlim_{U_{\pp}} M(\mathfrak a, J, U_{\pp}, \infty), \quad M'(\mathfrak a, J, \infty)=\varprojlim_{U_{\pp}'} M'(\mathfrak a, J, U_{\pp}', \infty),\]
\[ M_{\infty} =\varprojlim_{\mathfrak  a, J} M(\mathfrak a, J, \infty ),\quad M'_{\infty} =\varprojlim_{\mathfrak  a, J} M'(\mathfrak a, J,  \infty ),\]
where the limits are taken over open ideals $\mathfrak a$ of $\OO\br{z_1, \ldots, z_{4|S|-1}}$, $J$ of $\OO_{L, \infty}$,  and compact open subgroups $U_{\pp}$ of $\GL_2(\Qp)$ and 
$U_{\pp}'$ of $D_v^\times$.

It  follows \cite[Corollary 7.5]{KW2} that $S_{\psi, \lambda}(U(N)^{\pp}, \OO)^d_{\mm'_{Q_N}}$ is a
flat $\OO_{L,N}$-module and if we base change it along the map $\OO_{L,N}\rightarrow \OO$,
which sends all the elements of $\Delta_{Q_N}$ to $1$,  then we obtain an isomorphism.  
\[\OO\otimes_{\OO_{L,N}}S_{\psi, \lambda}(U(N)^{\pp}, \OO)^d_{\mm'_{Q_N}}\cong S_{\psi, \lambda}(U^{\pp}, \OO)_{\mm}^d.\]
Since $S_{\psi, \lambda}(U^{\pp}, \OO)_{\mm}^d$ is projective in $\Mod^{\pro}_{\GL_2(\Zp), \psi_{\pp}}(\OO)$, the above isomorphism and flatness over $\OO_{L,N}$ imply that $S_{\psi, \lambda}(U(N)^{\pp}, \OO)^d_{\mm'_{Q_N}}$ is projective in the category $\Mod^{\pro}_{\GL_2(\Zp), \psi_{\pp}}(\OO_{L,N})$. 
We deduce that if $A$ is an artinian quotient of $\OO_{L, N}$ then 
$A\otimes_{\OO_{L,N}} S_{\psi, \lambda}(U(N)^{\pp}, \OO)^d_{\mm'_{Q_N}}$ is projective 
in $\Mod^{\pro}_{\GL_2(\Zp), \psi_{\pp}}(A)$. For a fixed $A$ we may choose 
a sufficiently small  open pro-$p$ subgroup $H$ of $\GL_2(\Qp)$, so that  
 $\psi_{\pp}: H\cap Z \rightarrow \OO^{\times} \rightarrow A^{\times}$ is trivial, 
then $A\otimes_{\OO_{L, N}} S_{\psi, \lambda}(U(N)^{\pp}, \OO)^d_{\mm'_{Q_N}}$ is a projective $A\br{H/H\cap Z}$-module. 
Since 
\begin{equation}\label{fibre}
k\otimes_{\OO_{L, N}} S_{\psi, \lambda}(U(N)^{\pp}, \OO)^d_{\mm'_{Q_N}}\cong S_{\psi, \lambda}(U^{\pp}, k)^{\vee}_{\mm},
\end{equation}
 the claim implies that 
 \[A\otimes_{\OO_{L, N}} S_{\psi, \lambda}(U(N)^{\pp}, \OO)^d_{\mm'_{Q_N}}\cong A\br{H/H \cap Z}^m,\]
  where $m=\dim_k (S_{\psi, \lambda}(U^{\pp}, k)^{\vee}_{\mm})_{H}=\dim_k S_{\psi, \lambda}(U^{\pp}H, k)_{\mm}$. This implies that if we fix $\mathfrak a$ and $J$ then there is an open subgroup 
  $H$ of $\GL_2(\Qp)$ such that 
  $M(\mathfrak a, J, N)$
  are  isomorphic as $\OO\br{H}$-modules for all $N\in I_J$. 
  
  Lemma \ref{non-Eis} with \cite[Proposition 4.7]{scholze} imply that 
  that $\Sch^0(S_{\psi, \lambda}(U(N)^{\pp}, k)^{\vee}_{\mm'_{Q_N}})$ and 
  hence 
  $\Sch^0(S_{\psi, \lambda}(U(N)^{\pp}, \OO)^d_{\mm'_{Q_N}})$ are equal to $0$.  
  Since 
  \[\Sch^1(S_{\psi, \lambda}(U(N)^{\pp}, \OO)^d_{\mm'_{Q_N}})\cong \widehat{H}^1_{\psi, \lambda}(U(N)^{\pp}, \OO)^d_{\mm'_{Q_N}}\]
  by \cite[Proposition 6.3]{ludwig} and $R^{\square, \psi}_{F, S_{Q_N}}$ is flat over $R^{\psi}_{F, S_{Q_N}}$, Lemma \ref{yh}
  implies that 
    that \[\Sch^1(M(N))\cong M'(N).\] Since $S_{\psi, \lambda}(U(N)^{\pp}, \OO)^d_{\mm'_{Q_N}}$
    is flat over $\OO_{L,N}$, we deduce using \eqref{ten} that  $M(N)$ is flat over $S_N$. Since 
    the action of $S_{\infty}$ on $M(N)$ factors through $S_N$, for all compact $S_{\infty}$-modules 
    $\md$ we have 
    \[\md\wtimes_{S_{\infty}} M(N)\cong (\md\wtimes_{S_{\infty}} S_N)\wtimes_{S_N} M(N)\]
    and hence $\Sch^1(\md\wtimes_{S_{\infty}} M(N))\cong \md\wtimes_{S_{\infty}} M'(N)$ by Lemma 
    \ref{yh}. We apply this observation to $\md=S_{\infty}/(\mathfrak a, J)$ to deduce that 
  \begin{equation}\label{fedup}
  \Sch^1(M(\mathfrak a, J, N))\cong M'(\mathfrak a, J, N).
  \end{equation}
   Let $\Pi=\bigl ( \prod_{N\in I_J}  M(\mathfrak a, J, N)^{\vee}\bigr)^{\sm}$. It follows from Lemma \ref{ultra2} that 
  \[ M(\mathfrak a, J, \infty)^{\vee} \cong (\OO_{L, \infty}/J)_{I_J, x} \otimes_{(\OO_{L, \infty}/J)_{I_J}} \Pi.\]  
  Since the restrictions of $M(\mathfrak a, J, N)^{\vee}$ to $H$ are all isomorphic, Corrollary \ref{cor2} together with \eqref{fedup} 
  imply that 
  \[\Sch^1(M(\mathfrak a, J, \infty))\cong M'(\mathfrak a, J, \infty).\]
  Since $\Sch^1$ commutes with projective limits we obtain $\Sch^1(M_{\infty})\cong M'_{\infty}$. 
  \end{proof} 
  
  \begin{remar} Let us point out how a part of the proof of  \cite[Lemma 4.17  (1)]{6auth} can be rewritten using ultrafilters. It follows from local-global compatibility that the action of $R^{\square, \psi_{\pp}}_{\rhobar}$ on 
  \[\Hom_{\OO\br{\GL_2(\Zp)}}^{\cont}( M(N), (\sigma_{a, b}^0)^d), \quad \Hom_{\OO\br{\OO_{D_{\pp}}^{\times}}}^{\cont}( M'(N), ((\sigma'_{a, b})^0)^d)\]
factors through the quotient $R^{\square, \psi_{\pp}}_{\rhobar}(\sigma_{a,b})$. Lemma \ref{pst_quotient} implies that the same holds for the modules $M_{\infty}(\sigma_{a,b})$, $M_{\infty}'(\sigma'_{a,b})$. 
\end{remar}

\begin{lem}\label{Sch00} $\Sch^0(M_{\infty})=0$.
\end{lem}
\begin{proof} We give two proofs of the claim. It follows from \eqref{kill_patching_var} that 
\[M_{\infty}\wtimes_{R_{\infty}} k \cong S_{\psi, \lambda}(U^{\pp}, k)_{\mm}^{\vee}\wtimes_{R^{\psi}_{F, S}}k.\] By Lemma \ref{non-Eis} the $\SL_2(\Qp)$-coinvariants of $M_{\infty}\wtimes_{R_{\infty}} k$ are zero and thus $\Sch^0(M_{\infty}\wtimes_{R_{\infty}} k)=0$. 
Since
$\Sch^0$ is right exact we obtain $\Sch^0(M_{\infty})\wtimes_{R_{\infty}} k\cong \Sch^0(M_{\infty}\wtimes_{R_{\infty}}k)=0$. Since $\Sch^0(M_{\infty})$ is a compact $R_{\infty}$-module we deduce
that $\Sch^0(M_{\infty})=0$.

The second proof (which is an overkill for $\GL_2$, but would also work in the setting of unitary groups, 
see Remark \ref{unitary}, and show the vanishing 
$\Sch^i(M_\infty)$ for $i$ up to the middle degree) uses the non-Eisenstein property of the ideal $\mm$ after applying the functor $\Sch^0$. It follows from \cite[Theorem 6.2]{scholze} that 
$\Sch^0(S_{\psi, \lambda}(U(N)^{\pp}, \OO)^d_{\mm'_{Q_N}})= \widehat{H}^0_{\psi, \lambda}(U(N)^{\pp}, \OO)^d_{\mm'_{Q_N}}$. Since $\mm'_{Q_N}$ is non-Eisenstein this implies 
that the group is zero. Using Lemma \ref{yh} and the flatness arguments explained in the course
of the proof of Theorem \ref{the_patch}, we obtain that $\Sch^0(M(N))=0$ and $\Sch^0(M(\mathfrak a, J, N))=0$. It then follows from Corollary \ref{cor2} that $\Sch^0(M_{\infty})=0$. 
\end{proof}

\begin{lem}\label{Sch11} Let $I$ be an ideal of $R_{\infty}$ then the natural map 
\begin{equation}\label{Itorsion}
\Sch^1(M_{\infty})\otimes_{R_{\infty}} R_{\infty}/I \rightarrow \Sch^1(M_{\infty}\otimes_{R_{\infty}} R_{\infty}/I)
\end{equation} is surjective. The kernel is a finitely generated $R_{\infty}/I$-module on 
which the subgroup of reduced norm $1$ elements in $\OO_{D_{\pp}}^{\times}$ acts trivially. 
\end{lem}
\begin{proof} The proof of \cite[Proposition 7.7]{scholze} goes through verbatim here with 
Lemma \ref{Sch00} as an input to show that \eqref{Itorsion} is surjective and the norm $1$
subgroup of $\OO_{D_{\pp}}^{\times}$ acts trivially on the kernel, which we denote by 
$K$. Moreover, it follows 
from \cite[Theorem 4.4]{scholze} that $\Sch^1(M_{\infty})$ is a finitely generated 
$R_{\infty}\br{\OO_{D_{\pp}}^{\times}}$-module. Since the ring is noetherian, 
$K$ is also a finitely generated $R_{\infty}\br{\OO_{D_{\pp}}^{\times}}$-module. 
Since the center of $\OO_{D_{\pp}}^{\times}$ acts on $K$ by $\psi_{\pp}^{-1}$ and 
the norm $1$ subgroup acts trivially we deduce that $K$ is a finitely generated 
$R_{\infty}$-module. Since $I$ acts trivially on $K$ we deduce that $K$ is a finitely 
generated $R_{\infty}/I$-module.
\end{proof} 
 
 \subsection{Infinitesimal character}

Let $\mathfrak g$ be the $\Qp$-Lie algebra of $D^{\times}_{\pp}$. Then we may identify $\mathfrak g$ 
with $D_{\pp}$, $\mathfrak g\otimes_{\Qp} L= \gl_2\otimes_{\Qp} L$ and $Z(\mathfrak g_L)=Z(\gl_2)_L$. 
The dual group $\Ghat$ in this case is $\GL_2$ and we identify its Lie algebra $\ghat$ with $M_2(L)$, and let 
$e=\bigl( \begin{smallmatrix} 0 & 1 \\ 0 & 0\end{smallmatrix} \bigr)$, $f= \bigl( \begin{smallmatrix} 0 & 0 \\ 1 & 0\end{smallmatrix} \bigr)$, 
$h= \bigl( \begin{smallmatrix} 1 & 0 \\ 0 & -1\end{smallmatrix} \bigr)$, $z= \bigl( \begin{smallmatrix} 1 & 0 \\ 0 & 1\end{smallmatrix} \bigr)$
be a basis of $\ghat$ as an $L$-vector space and let $e^*, f^*, h^*, z^*$ be the dual basis of $\ghat^*$. 
If $A$ is an $L$-algebra then 
a matrix $M= \bigl( \begin{smallmatrix} \alpha & \beta \\ \gamma & \delta \end{smallmatrix} \bigr)\in M_2(A)$ defines 
an $L$-algebra  homomorphism $\ev_M: S(\ghat^*)\rightarrow A$, which send $e^*\mapsto \beta$, $f^*\mapsto \gamma$, 
$h^*\mapsto \alpha-\beta$, $z^*\mapsto \alpha+\beta$. 
We may identify 
$Z(\mathfrak g_L)$ with a polynomial ring in variables $C$ and $Z$, where $C$ is the Casimir operator and 
$Z= \bigl( \begin{smallmatrix} 1 & 0 \\ 0 & 1\end{smallmatrix} \bigr)\in \mathfrak g_L$. The ring
homomorphism $Z(\mathfrak g_L)\rightarrow S(\ghat^*)^{\Ghat} \subset S(\ghat^*)$ used in \cite[Definition 4.19]{DPS}
sends $Z$ to $z^*$ and $C$ to $e^* f^* + f^* e^* +\frac{1}{2} (h^*)^2-\frac{1}{2}$, which is the unique element 
$c\in S(\ghat^*)^{\Ghat}$ such that $\ev_M(c)= \frac{1}{2}(\alpha-\delta-1)^2 +(\alpha-\delta-1)= 
\frac{1}{2}(\alpha-\delta)^2 -\frac{1}{2}=\frac{1}{2}\tr(M)^2 -2 \det M -\frac{1}{2},$
for all $M=\bigl( \begin{smallmatrix} \alpha & 0 \\ 0 & \delta\end{smallmatrix} \bigr)$. We denote the 
composition $Z(\mathfrak g_L)\rightarrow S(\ghat^*)^{\Ghat}\overset{\ev_M}{\longrightarrow} R$ by $\varphi_M$.

Let $R$ be a complete local noetherian $\OO$-algebra with finite residue field, and let $R^{\rig}$ be the ring of global functions
of the rigid analytic space $(\Spf R)^{\rig}$. Let $\rho: \Gal_{\Qp}\rightarrow \GL_2(R)$ be a continuous representation. 
We let $\tilde{\delta}: \mathbb{G}_m\rightarrow \That$, $t \mapsto \bigl( \begin{smallmatrix} 1& 0 \\ 0& t^{-1} \end{smallmatrix} \bigr)$, 
where $\That$ is the subgroup of diagonal matrices in $\Ghat$. The choice of $\tilde{\delta}$ allows us to define a 
representation of $\Gal_{\Qp}$ into the $C$-group of $\GL_2$ :
\[\rho^C: \Gal_{\Qp} \rightarrow {^C}\GL_2(R), \quad g\mapsto ( \rho(g) \tilde{\delta}(\chi_{\cyc}(g))^{-1}, \chi_{\cyc}(g)),\] 
see \cite[Sections 2.1, 4.7]{DPS} for more details. In \cite[Definition 4.23]{DPS} to $\rho^C$ we attach an $L$-algebra homomorphism 
$\zeta^C_{\rho^C}: Z(\mathfrak g_L)\rightarrow R^{\rig}$. It follows from Equation (34) in \cite[Section 4.7]{DPS} that 
 for every open affinoid $U=\Sp(A)\subset (\Spf R)^{\rig}$
the composition \[Z(\mathfrak g_L)\rightarrow R^{\rig} \rightarrow A\rightarrow \Cp\wtimes_{\Qp} A\] coincides 
with $\varphi_M: Z(\mathfrak g_L)\rightarrow  \Cp\wtimes_{\Qp} A$ with $M=\Theta_{U} +\frac{1}{2} \bigl( \begin{smallmatrix} 1 & 0 \\ 0 & 1\end{smallmatrix} \bigr)$, 
where $\Theta_{U} \in M_2( \Cp\wtimes_{\Qp} A)$ is the Sen operator of $\rho\otimes_{R} A$. 

Let us spell out \cite[Lemma 5.12]{DPS} in our situation. If  $R$ is a ring of integers in a finite extension of $L'$ of $L$ then $\rho$ is Hodge--Tate 
with weights $a>b$ if and only if the Sen operator $\Theta\in M_2(\Cp\wtimes_{\Qp} L')$ is conjugate to the matrix  $\bigl( \begin{smallmatrix} a & 0 \\ 0 & b\end{smallmatrix} \bigr)$. 
It follows from the above that in this case $\zeta^C_{\rho^C}(C)=\frac{1}{2}(a-b)^2-\frac{1}{2}$, 
$\zeta^C_{\rho^C}(Z)= a+b+1$, and hence $\zeta^C_{\rho^C}$ coincides with the infinitesimal 
character of $\Sym^{a-b-1}L^2\otimes \det^{b+1}$. Conversely, if $\zeta^C_{\rho^C}$ is infinitesimal
character of $\Sym^{a-b-1}L^2\otimes \det^{b+1}$ then $\Theta$ has eigenvalues $a$ and $b$ and hence $\rho$ is Hodge--Tate with weights $a, b$.

Let $\rho^{\univ}: \Gal_{\Qp}\rightarrow \GL_2(R_{\rhobar}^{\square, \psi_{\pp}})$ be the universal 
framed deformation of $\rhobar$ with the fixed determinant equal to $\psi_{\pp} \varepsilon^{-1}$. 
Let $\rho_{\infty}:=\rho^{\univ}\otimes_{R_{\rhobar}^{\square, \psi_{\pp}}}R_{\infty}$ and let 
\[\zeta^C_{\rho_{\infty}^C}: Z(\mathfrak g_L)\rightarrow R_{\infty}^{\rig}\]
be the infinitesimal character defined above. If $x:R_{\infty}\rightarrow \Qpbar$ is an $\OO$-algebra
homomorphism then we let  $\rho_x: \Gal_{\Qp}\rightarrow \GL_2(\Qpbar)$ be the specialization of $\rho_{\infty}$ 
along $x$.  The image of $\rho_x$ is contained in $\GL_2(\OO')$, 
where $\OO'$ is the ring of integers in a finite extension $L'$ of $L$. The infinitesimal character 
$\zeta^C_{\rho_x^C}: Z(\mathfrak g_L) \rightarrow L'$ coincides with  a specialization of $\zeta^C_{\rho_{\infty}^C}$ at $x$.

If $M$ is a compact $\OO$-module we let $\Pi(M)= \Hom^{\cont}_{\OO}(M, L)$ be the 
$L$-Banach space with the topology induced by the supremum norm. If $x: R_{\infty}\rightarrow \Qpbar$ is an $\OO$-algebra homomorphism then we let $\mm_x$ be its kernel and 
\[\Pi_x:= \Pi(M_{\infty})[\mm_x], \quad \Pi'_x:= \Pi(M'_{\infty})[\mm_x]\]
be the closed subspaces annihilated by $\mm_x$. Then $\Pi_x\in \Ban^{\adm}_{\GL_2(\Qp), \psi_{\pp}}(L)$ and $\Pi_x'\in \Ban^{\adm}_{D_{\pp}^{\times}, \psi_{\pp}}(L)$. 

\begin{lem}\label{closed_subspace} $\Sch^1(\Pi_x)$ is a closed subspace of $\Pi'_x$. Moreover, $\Pi'_x/\Sch^1(\Pi_x)$ 
is a finite dimensional $L$-vector space on which the norm $1$ subgroup of $\OO_{D_{\pp}}^{\times}$ 
acts trivially. 
\end{lem}
\begin{proof} This follows from Lemma \ref{Sch11} applied with $I=\mm_x$ and Schikhof duality.
We note that $R_{\infty}/\mm_x$ is a compact $\OO$-subalgebra  in a finite extension of $L$, and hence a finitely generated
$\OO$-module. 
\end{proof}

\begin{prop}\label{inf_eigen} If $x: R_{\infty}\rightarrow \Qpbar$ is an $\OO$-algebra homomorphism then 
$Z(\mathfrak g_L)$ acts on $\Pi_x^{\la}$ and $(\Pi_x')^{\la}$ by the infinitesimal character 
$\zeta^C_{\rho_x^C}$.
\end{prop} 
\begin{proof} Theorem \ref{the_patch} shows that the conditions of \cite[Theorem 8.6]{DPS}
are satisfied, which implies the assertion. Let us explain why part (iii) of \cite[Theorem 8.6]{DPS}
holds in more detail. If $x$ is a closed point of $\Spec R_{\infty}(\sigma_{a, b})[1/p]$ (or $\Spec R_{\infty}(\sigma'_{a, b})[1/p]$) then 
it follows from part \eqref{ptch3} of Theorem \ref{the_patch} that $\rho_x$ has Hodge--Tate weights
$(a,b)$, which implies that $\zeta^C_{\rho_x^C}$ is the infinitesimal character of $\Sym^{a-b-1} L^2 \otimes \det^{b+1}$. 
\end{proof} 
 
 Until now we have not used in an essential way that $F_{\pp}=\Qp$. The following two results 
 make use of this assumption in their use of $p$-adic Langlands correspondence for $\GL_2(\Qp)$ 
 and the fact that  the unipotent radical of a Borel subgroup of $\GL_2(\Qp)$ has dimension $1$ as a $p$-adic analytic group.
 
 \begin{cor}\label{the_end_is_nigh} Let $\zeta: \Qp^{\times}\rightarrow \OO^{\times}$ be a character and 
 $\Pi\in \Ban^{\adm}_{\GL_2(\Qp), \zeta}(L)$ be absolutely 
 irreducible and non-ordinary and let $\rho: \Gal_{\Qp}\rightarrow \GL_2(L)$ be the absolutely 
 irreducible representation corresponding to $\Pi$. Then $Z(\mathfrak g_L)$ acts on $\Pi^{\la}$ and 
 on $\Sch^1(\Pi)^{\la}$ by the infinitesimal character $\zeta^C_{\rho^C}$. 
\end{cor}
\begin{proof} Let $\rhobar: \Gal_{\Qp} \rightarrow \GL_2(k)$ the reduction of $\Gal_{\Qp}$-invariant lattice
in $\rho$. We globalize $\rhobar$ and obtain $M_{\infty}$, $M_{\infty}'$, $\psi$  as above. Since 
$\zeta \varepsilon^{-1}\cong \det \rhobar \cong \psi_{\pp} \varepsilon^{-1}\pmod{\varpi}$, there exists 
a character $\eta: \Qp^{\times}\rightarrow \OO^{\times}$ such that $\psi_{\pp} =\zeta\eta^2$. Since
$\Sch^1$ commutes with twisting by characters by Lemma \ref{olivertwist} we may replace $\Pi$ with 
$\Pi \otimes\eta\circ \det$ and assume that $\psi_{\pp}=\zeta$.

Let $y: R_{\rhobar}^{\square, \psi_{\pp}}\rightarrow L$ be an $\OO$-algebra homomorphism corresponding to $\rho$.
Since $R_{\infty}$ is a faithfully flat $R_{\rhobar}^{\square, \psi_{\pp}}$-algebra 
there is an 
$\OO$-algebra homomorphism $x: R_{\infty}\rightarrow \Qpbar$ extending $y$. It follows from 
\cite[Theorem 7.1]{forum} that $\Pi_x\cong \Pi^{\oplus m}$ for some $m\ge 1$. Thus $\Sch^1(\Pi_x)\cong \Sch^1(\Pi)^{\oplus m}$.
Since $\Sch^{1}(\Pi_x)$ is a closed subspace of $\Pi_x'$ by Lemma \ref{closed_subspace} the assertion follows from Proposition \ref{inf_eigen}. 

It remains to show that the conditions of \cite[Theorem 7.1]{forum} are satisfied. 
Since $M_{\infty}$ is a thickening, $M_{\infty}$ is a finitely generated $R_{\infty}\br{\GL_2(\Zp)}$-module.
We have to show that 
$R_{\infty}$ acts faithfully on $M_{\infty}$; this can be verified as in \cite[Proposition 3.9]{tung}.
Let $R^{\ps, \psi_{\pp}}_{\tr \rhobar}$ be the pseudo-character deformation ring of $\tr \rhobar$ with 
fixed determinant equal to $\psi_{\pp}\varepsilon^{-1}$. We have to show that the two actions of $R^{\ps, \psi_{\pp}}_{\tr \rhobar}$ on $M_{\infty}$ coincide:  
one action defined via 
$R^{\ps, \psi_{\pp}}_{\tr \rhobar}\rightarrow R^{\square, \psi_{\pp}}_{\rhobar}\rightarrow R_{\infty}$ and
the other action as the centre of certain subcategory of $\Mod^{\ladm}_{\GL_2(\Qp), \psi_{\pp}}(\OO)$, see \cite{image}, \cite[Section 4.1]{forum}. This holds for $S_{\psi, \lambda}(U(N)^{\pp}, \OO)^d_{\mm'_{Q_N}}$ by \cite[Proposition 5.5, Corollary 5.6]{ludwig} (a similar argument appears in \cite[Theorem 3.5.5]{lue_pan1})  and hence the two actions coincide on the modules $M(N)$ defined 
in the course of the proof of Theorem \ref{the_patch}, by passing to the limit we obtain the same for $M_{\infty}$. 
\end{proof}

\begin{thm}\label{the_same_inf} Let $\Pi\in \Ban^{\adm}_{\GL_2(\Qp), \zeta}(L)$ be absolutely irreducible and non-or\-di\-na\-ry
and let $\rho: \Gal_{\Qp}\rightarrow \GL_2(L)$ be the absolutely 
 irreducible representation corresponding to $\Pi$. Let $K$ be a compact open subgroup of 
 $\GL_2(\Qp)$ and let $K'$ be a compact open subgroup of $D_{\pp}^{\times}$. If the difference 
 of the Hodge--Tate--Sen weights of $\rho$ is not a non-zero integer then $\Pi|_K$ is of finite length 
 in $\Ban^{\adm}_{K, \zeta}(L)$ and $\Sch^1(\Pi)|_{K'}$ is of finite length 
 in $\Ban^{\adm}_{K', \zeta}(L)$.
 \end{thm}
 \begin{proof} It follows from Corollary \ref{the_end_is_nigh} that $\Pi^{\la}$ and $\Sch^1(\Pi)^{\la}$  
 have an infinitesimal character $\zeta_{\rho^C}^C: Z(\mathfrak g_L) \rightarrow L$. Since both 
 $\Pi$ and $\Sch^1(\Pi)$ are admissible $L$-Banach space representations Theorem \ref{astuce}
 implies that $d(\Pi)\le 1$ and $d(\Sch^1(\Pi))\le 1$. 
 The assumption that the difference of 
 the Hodge--Tate--Sen weights of $\rho$ is not a non-zero integer implies that that $\zeta^C_{\rho^C}$ is not algebraic and the assertion 
 follows from Corollary \ref{cor:fin_length} a).
 \end{proof}
 
 \begin{remar} We note that with our arguments we cannot show that the Banach spaces $\Sch^1(\Pi)$ in Corollary \ref{the_end_is_nigh} are non-zero. 
If the residual representation $\rhobar: \Gal_{\Qp} \rightarrow \GL_2(k)$ is reducible generic then one can deduce the non-vanishing of $\Sch^1(\Pi)$ from 
the results of \cite{ludwig}. The results of \cite{ludwig} rely on vanishing of $H^2_{\et}(\PP^{1}_{\Cp}, \FF_{\pi})$, where $\pi$ is a mod $p$ principal series 
representation, which is proved by Judith Ludwig in \cite{j_ludwig}. The same vanishing has been claimed by David Hansen for $\pi$ supersingular. This would imply that 
 $\Sch^1(\Pi)$ is non-zero when $\rhobar$ is absolutely irreducible. The vanishing   of $H^2_{\et}(\PP^{1}_{\Cp}, \FF_{\pi})$ for supersingular $\pi$  
 has been recently 
 proved by Yongquan Hu and Haoran Wang in \cite{hu-wang}, under mild genericity assumptions on 
 $\pi$.
\end{remar}
 
 \begin{remar} Most of the argument carries over if $p=2$: \cite[Theorem 7.1]{forum} is available for 
 $p=2$ and the patching argument is carried out in \cite{Tung2}, but one would have to rewrite 
the proof of Theorem \ref{the_patch} to accommodate the subtleties that arise in $p=2$ case.  However, the globalization result used in Proposition \ref{globalize} is not in the literature.\footnote{It was pointed out to us by Toby Gee that it might be feasible to carry out this globalization using \cite[Proposition 6.7]{thorne}. 
We don't pursue this here, just note for the interested reader that since we twist by characters in the proof of 
Corollary \ref{the_end_is_nigh}, one would have to produce a potentially automorphic lift $r_{\pi}$ with $(\det r_{\pi}) (\Art_{\mathbb{Q}_2}(-1))=1$ 
and also a potentially automorphic lift $r_{\pi}$ with $(\det r_{\pi})(\Art_{\mathbb{Q}_2}(-1))=-1$ to make our argument work. }
 \end{remar}
 
 \begin{remar}\label{unitary} The proofs of Theorem \ref{the_patch} and Proposition \ref{inf_eigen} carry over to 
 the setting of unitary groups using the recent results of Caraiani--Scholze \cite{CS} and Kegang Liu \cite{kegang_liu} under the assumption that $p$ does not divide $2n$. If  $\rhobar: \Gal_E\rightarrow \GL_n(\Fpbar)$ is a representation, where $E$ is a finite extension 
 of $\Qp$, then we may globalize it following \cite[Section 2.1]{6auth}, and we may arrange that the resulting  global Galois representation $\rbar$ will contain $\GL_n(\mathbb F_{p^m})$ for a given $m$. Using this and Chebotarev density Theorem one may show that this representation is decomposed generic  in the sense of \cite{CS}, see \cite[Remark 1.4]{CS}. It follows from the main Theorem of \cite{CS} that the cohomology localized at the maximal ideal $\mm$ in the Hecke algebra corresponding 
 to $\rbar$ vanishes except in the middle degree $q_0$, which implies that the $q_0$-th completed 
 cohomology with $L/\OO$ -coefficients localized at $\mm$ is injective in  $\Mod^{\sm}_K(\OO)$, where
 $K$ is a maximal compact open subgroup of $G(\Qp)$, with $G$ as in \cite{CS}. The group 
 $\widehat{H}^{q_0}(U^p, L/\OO)_{\mm}$ is the analog of $\widehat{H}^1_{\psi}(U^p, L/\OO)_{\mm}$ considered in this paper, except that when working with unitary groups we don't have to 
 fix a central character. Now Kegang Liu \cite[Theorem 1.1]{kegang_liu} has extended 
 Scholze's \cite[Theorem 1.3]{scholze} to the setting of unitary groups. Using this result one may 
 obtain an analog of \eqref{sch1} in the setting of unitary groups. One may then rewrite the patching 
 argument in \cite{6auth} using ultrafilters and obtain the analog of Theorem \ref{the_patch} using the same proof. We note that since the centre does not cause problems in the unitary group setting, \cite[Footnote 7]{scholze} applies here, so that the proof is slightly easier. Once we have the patched 
 modules we are in the situation of \cite[Section 9.11]{DPS} for $\GL_n(E)$ and the patched version 
 of \cite[Section 9.10]{DPS} for the division algebra over $E$ with invariant $1/n$. The proof of Proposition \ref{inf_eigen} carries over verbatim. 
 
 We note that  if $M_{\infty}$ is flat over $R_{\infty}$
 then it follows from Lemma \ref{yh} together with the second proof of Lemma \ref{Sch00} that 
 $\Sch^{q_0}(\Pi_x)\cong \Pi'_x$, and thus $\Sch^{q_0}(\Pi_x)^{\la}$ has the same infinitesimal 
 character as $(\Pi_x)^{\la}$. If we do not know that $M_{\infty}$ is flat over $R_{\infty}$ and $q_0>1$
 then the relationship between $\Sch^{q_0}(\Pi_x)$ and $\Pi_x'$ is not as straightforward as described 
 by Lemma \ref{closed_subspace} and we cannot control the action of the centre of the universal 
 enveloping algebra on $\Sch^{q_0}(\Pi_x)^{\la}$. 
  \end{remar}

\end{document}